\documentclass[10pt]{article}

\usepackage[a4paper]{geometry}
\usepackage[T1]{fontenc} 
\usepackage{lmodern}
\usepackage[latin1]{inputenc} 
\usepackage[english]{babel} 
\usepackage{graphicx} 
\DeclareGraphicsExtensions{.pdf}
\usepackage{amsmath,amssymb,amsthm,amsfonts,mathrsfs}
\usepackage{nicefrac}
\usepackage{authblk}  
\usepackage{color}
\usepackage{tikz}
\usepackage{pgf}
\usepackage{wasysym}

\usepackage[textsize=footnotesize,backgroundcolor=yellow!70,bordercolor=orange]{todonotes}

\RequirePackage[normalem]{ulem}

\usetikzlibrary{shapes}

\usetikzlibrary{patterns}

\newcommand{\colR}{green!20}

\newcommand{\patP}{crosshatch}
\newcommand{\patUnoMenoP}{north east lines}
\newcommand{\patUno}{}
\newcommand{\patProportionalP}{crosshatch dots}

\newcommand{\tr}{\mathsf{T}}

\newcommand{\skeletongeneric}{
	\draw[help lines,dashed] (0,0) -- (3.1,3.1);
	\draw[help lines,<->] (1.57,-0.1) -- (2.93,-0.1)
	node[pos=0.5, above] {$\pisup{\rme}$};
	\draw[help lines,->,yshift=-1.475cm]
	(-0.3,0) node[left]{$\pisup{\rp}$} -- (-0.1,0);
	\draw[help lines,->,xshift=1.475cm]
	(0,-0.3) node[above]{$\pisup{\rp}$} -- (0,-0.1);
	\draw[very thick] (0.1,-0.03) -- (-0.03,-0.03) -- 
	(-0.03,3.13) -- (0.1,3.13);
	\draw[very thick] (3.1-0.1,-0.03) -- (3.1+0.03,-0.03) -- 
	(3.1+0.03,3.13) -- (3.1-0.1,3.13);
	
	\draw[help lines,xshift=1.4 cm] (0,0) -- (0,3.1);
	\draw[help lines,yshift=-1.4 cm] (0,0) -- (3.1,0);
	\draw[help lines,xshift=1.55 cm] (0,0) -- (0,3.1);
	\draw[help lines,yshift=-1.55 cm] (0,0) -- (3.1,0);
	\draw[help lines,xshift=2.95 cm] (0,0) -- (0,3.1);
	\draw[help lines,yshift=-2.95 cm] (0,0) -- (3.1,0);
}

\newcommand{\skeleton}{
	\draw[help lines,dashed] (0,0) -- (3.1,3.1);
	\draw[help lines,<->] (2.02,0.1) -- (3,0.1)
	node[pos=0.5, above] {$r$};
	\draw[very thick] (0.1,-0.03) -- (-0.03,-0.03) -- 
	(-0.03,3.13) -- (0.1,3.13);
	\draw[very thick] (3.1-0.1,-0.03) -- (3.1+0.03,-0.03) -- 
	(3.1+0.03,3.13) -- (3.1-0.1,3.13);
	\foreach \k in {1,2,3}
	{
		\draw[help lines,xshift=\k cm] (0,0) -- (0,3.1);
		\draw[help lines,yshift=-\k cm] (0,0) -- (3.1,0);
	}
}

\newcommand{\vm}{V_{\max}}
\newcommand{\vb}{v_*}
\newcommand{\va}{v^*}
\newcommand{\fb}{f_*}

\newcommand{\fa}{f^*}
\newcommand{\f}{\mathbf{f}}

\newcommand{\Dv}{\Delta v}
\newcommand{\dv}{\delta v}

\newcommand{\ds}{\displaystyle}

\newcommand{\norm}[1]{\left\lVert#1\right\rVert}
\newcommand{\abs}[1]{\left\vert#1\right\vert}
\newcommand{\pisup}[1]{\left\lceil#1\right\rceil}

\newcommand{\dvu}{\mathrm{d}v}
\newcommand{\dvua}{\mathrm{d}\va}
\newcommand{\dvub}{\mathrm{d}\vb}
\newcommand{\rp}{r^+}
\newcommand{\rme}{r^-}

\newcommand{\AVVV}[3]{{A(#1\!\!\to\!#2|#3)}}
\newcommand{\AVVVR}[4]{{A(#1\!\!\to\!#2|#3;#4)}}
\newcommand{\Avvv}{\AVVV{\vb}{v}{\va}}
\newcommand{\Avvvr}{\AVVVR{\vb}{v}{\va}{\rho}}

\newcounter{the}
\theoremstyle{remark}
\newtheorem{theorem}[the]{Theorem}
\newtheorem{lemma}[the]{Lemma}
\newtheorem{remark}[the]{Remark}

\newtheorem{assumption}{Assumption}
\newtheorem*{notation}{Notation}

\author[1]{Gabriella Puppo\thanks{gabriella.puppo@uninsubria.it}}
\author[2]{Matteo Semplice\thanks{matteo.semplice@unito.it}}
\author[3]{Andrea Tosin\thanks{andrea.tosin@polito.it}}
\author[1]{Giuseppe Visconti\thanks{giuseppe.visconti@uninsubria.it}}
\affil[1]{Universit\`a dell'Insubria, Como, Italy}
\affil[2]{Universit\`a degli Studi di Torino, Torino, Italy}
\affil[3]{Istituto per le Applicazioni del Calcolo ``M. Picone'', CNR, Roma, Italy}

\date{}

\begin{document}

\title{Kinetic models for traffic flow resulting in a reduced space of microscopic velocities}
\maketitle

\begin{abstract}
	The purpose of this paper is to study the properties of kinetic models for traffic flow described by a Boltzmann-type approach and based on a continuous space of microscopic velocities. 
	In our models, the particular structure of the collision kernel allows one to find the analytical expression of a class of steady-state distributions, which are characterized by being supported on a quantized space of microscopic speeds. The number of these velocities is determined by a physical parameter describing the typical acceleration of a vehicle and the uniqueness of this class of solutions is supported by numerical investigations. 
	This shows that it is possible to have the full richness of a kinetic approach with the simplicity of a space of microscopic velocities characterized by a small number of modes. Moreover, the explicit expression of the asymptotic distribution paves the way to deriving new macroscopic equations using the closure provided by the kinetic model.
\end{abstract}



\section{Introduction} \label{sec:Intro}
The purpose of kinetic theory for traffic flow is to provide an aggregate representation of the distribution of vehicles on the road, thanks to a detailed characterization of the microscopic interactions among the vehicles, which play an important role in the macroscopic trend of the flow. The goal is to obtain information on the macroscopic characteristics of the flow without assuming previous knowledge on the dependence of the mean velocity on the local density of traffic, as it is done in standard macroscopic traffic models. See in particular the prototype of macroscopic models, \cite{lighthill1955PRSL}, or the reviews in \cite{piccoli2009ENCYCLOPEDIA,Rosini}. More refined macroscopic models consider a system of equations, instead of a single equation, see \cite{aw2000SIAP}, and/or they prescribe different flow conditions at certain stages, building phase transitions within the flow, \cite{Colombo2002,LebacqueGSOM,MendezVelasco13}, but still it is necessary to complete the model with a closure law, derived from heuristic or physical arguments, or from experimental data. Kinetic models provide quite naturally a closure law, which is linked in general to the equilibria of the kinetic model. Another way to derive closure laws is through microscopic ``follow the leader'' models \cite{aw2002SIAP,ZhangMultiphase}. For a review on the derivation of macroscopic traffic models from the microscopic ``follow the leader'' ones and from the mesoscopic kinetic theory, see \cite{klarReview}.

Several kinetic approaches have been proposed, starting from the pioneering work of \cite{Prigogine61, PrigogineHerman} and later \cite{paveri1975TR}. These models were based on a Boltzmann-type collision term in which the cross section, giving the probability of an interaction between two particles, is replaced with a probability distribution depending on the local traffic conditions. The equation describes the relaxation of the kinetic distribution in time due to the acceleration and slowing down interactions among vehicles. However, the interaction integrals appearing in kinetic Boltzmann-type models for traffic flow based on a continuous velocity space, see \cite{KlarWegener96}, typically do not provide the analytical expression of the equilibrium distribution and they are very demanding from a computational point of view. For this reason, two main approaches have been taken into account in order to compute the time-asymptotic distribution or to reduce the computational cost: on the one hand, one may consider Vlasov-Fokker-Planck type models in which the interaction integrals are replaced by differential operators, see \cite{HertyIllner08,HertyPareschi10,KlarIllnerMaterne}; on the other hand, one may consider simplified kinetic models with a small number of velocities, namely the discrete-velocity models, see \cite{FermoTosin13,FermoTosin14}.

In this work, following the classical Boltzmann-like setting of binary interactions, we study a kinetic model based on a continuous velocity space, which does not suffer from the aforementioned drawbacks. To this end, we focus only on spatially homogeneous problems and we investigate the interplay between the microscopic rules appearing in the collision kernel and the equilibrium solutions, finding analytical expressions of the time-asymptotic distribution, which result in a realistic macroscopic model for traffic flow.

As in \cite{KlarWegener96}, our transition probability is characterized by the fact that drivers react to the presence of other vehicles, deciding whether to modify their speed according to the overall traffic conditions and to the particular velocity of the cars around them. Thus, the decision of whether and how to modify one's speed depends on the local traffic density, or better, on the free space available, as we argue in \cite{PgSmTaVg}. Clearly, other choices were considered in literature. For example, in \cite{HertyPareschi10} the authors assume that drivers react to the local mean speed and they decide to accelerate  or to brake by comparing their velocity to the speed of the flow. Here, the possible speeds available to the driver are naturally the driver's current speed, and a set of speeds which depend on the velocity of the vehicles ahead and on the local density. In particular, the probability of accelerating increases monotonically with the free space available.
Instead, the probability of braking increases as the road becomes congested.

This framework permits to take the stochasticity of the drivers' behavior into account, thanks to the probability distribution which assigns a weight to the possible driver's decisions, while maintaining the general kinetic setting, based on the deterministic evolution of the distribution function. We propose two models that follow the framework just described: one is based on quantized velocity jumps, i.e., if acceleration occurs the new speed is obtained by increasing the pre-interaction velocity of a fixed quantity $\Dv$ ($\delta$ model); the other one is based on a continuous uniform distribution defined on a bounded interval parametrized by $\Dv$ ($\chi$ model), see also \cite{KlarWegener96,klar1997Enskog}.

In this paper, $\Dv$ is a finite parameter which models the physical velocity jump performed by vehicles when they increase their speeds as a result of an interaction. Clearly, this parameter may depend on the mechanical characteristics of vehicles, see \cite{PgSmTaVg3}, but in this paper we will assume that $\Dv$ is fixed. In the section on macroscopic properties, \S \ref{sect:fundamental}, we show that $\Dv$ is related to the maximum acceleration and we discuss how this parameter can be chosen through experimental data used by Lebacque \cite{Lebacque03}.

\begin{figure}
	\centering
	
	\begin{tikzpicture}
		\node[fill=black!10] (BC) at (0,0)
		{\parbox{0.33\textwidth}
			{Boltzmann\\
				continuous-velocity: \\
				$f(t,v)$ s.t.\\
				$\partial_t f(t,v)=Q[f,f](t,v)$}
		};
		
		\node[fill=black!10] (BD) at (8,0)  { 
			\parbox{0.33\textwidth}
			{Boltzmann\\
				discrete-velocity: \\
				$f_j(t), j=1,\ldots,N$ s.t.\\
				$\frac{d}{dt} f_j(t)=Q_j[\mathbf{f},\mathbf{f}](t)$}
		};
		
		\node[fill=black!10] (EQa) at (0,-3)
		{\parbox{0.4\textwidth}
			{Existence of quantized equilibria:
				$f^\infty(v)=\sum_{j=1}^N f_j^\infty\delta_{v_j}(v)$}
		};
		
		\node[fill=black!10] (EQd) at (8,-3)
		{\parbox{0.4\textwidth}
			{All equilibria are quantized:\newline
				$f^\infty(v)=\sum_{j=1}^N f_j^\infty\delta_{v_j}(v)$}
		};

		\draw[thick,->] (BC) -- (BD) node[pos=0.5,sloped,above]{discretize};
		\draw[thick,->] (BC) -- (EQa) node[pos=0.5,sloped,above]{equilibria};
		\draw[thick,->] (BC) -- (EQa) node[pos=0.5,sloped,below]{(Th.\ref{th:continuous-eq})};
		\draw[thick,->] (BD) -- (EQd) node[pos=0.5,sloped,above]{equilibria};
		\draw[thick,->] (BD) -- (EQd) node[pos=0.5,sloped,below]{(Th.\ref{th:delta_eq})};
		\draw[thick,<->] (EQa) -- (EQd) node[pos=0.5,sloped,below]{(Th.\ref{th:delta_eq})};
		
		\node[fill=black!10] (CL) at (4,-5)
		{\parbox{0.33\textwidth}{Macroscopic closure law based on a reduced velocity space}};
		\draw[thick,->] (EQa) -- (CL);
		\draw[thick,->] (EQd) -- (CL);
		
	\end{tikzpicture}
	\caption{Connection between the $\delta$ model and a discrete-velocity model, having the same steady-state distribution.}
	\label{fig:diagcomm}
\end{figure}
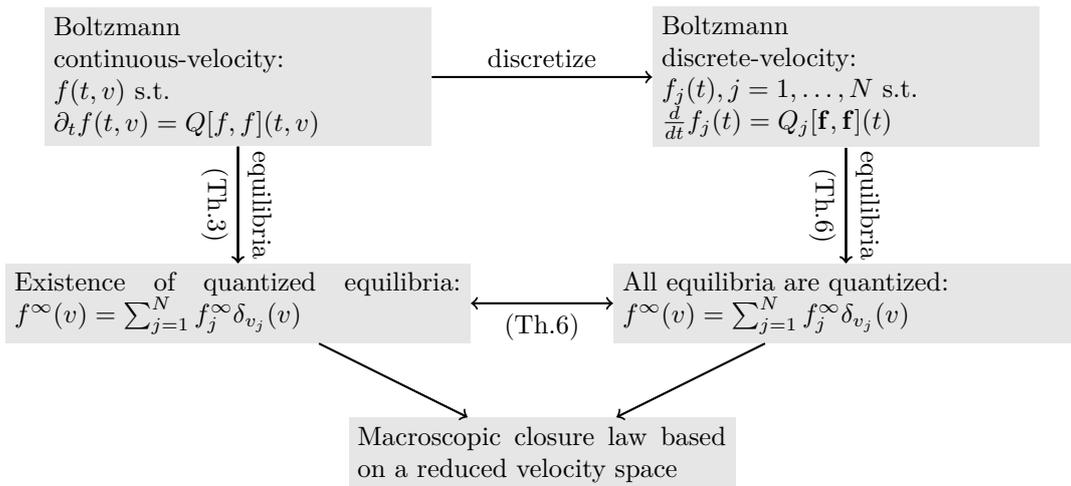

This paper also shows that the fundamental diagrams (or closure laws) obtained with the perhaps more natural, but more complex and more computationally demanding, $\chi$ model are very close to those provided by the simpler $\delta$ model. We thus investigate the equilibria of the $\delta$ model, both from an analytical and a numerical point of view. Analytically we find that velocity distributions formed by a linear combination of Dirac $\delta$'s may be equilibria only if the $\delta$'s are centered at velocities spaced by multiples of $\Dv$. Next, we compute equilibria using a numerical scheme capable of converging also to possible absolutely continuous equilibria. Here we find only the quantized equilibria described above, independently of the discretization parameters. This fact suggests that the class of discrete-velocity equilibria is the only one that the continuous-velocity Boltzmann-type $\delta$ model possesses. This situation is summarized graphically in Figure~\ref{fig:diagcomm}.

The paper is organized as follows. In \S \ref{sect:general-model} we briefly recall the Boltzmann-type kinetic equation and we specialize it by giving two sets of interaction rules. The resulting $\delta$ and $\chi$ models are discussed in depth in \S \ref{sect:delta} and in \S \ref{sect:chi}, respectively. In particular, in \S \ref{sect:delta} we prove the existence of a class of quantized steady-state distributions. Since we are unable to prove their uniqueness, we discretize the model by approximating the kinetic distribution with a piecewise constant function and we then show by numerical evidence that the class of stationary solutions of the $\delta$ model is only the one we have already studied analytically. In \S \ref{sect:chi}, we show the somewhat surprising result that the equilibrium distributions of the $\chi$ model yield a macroscopic flow that is extremely well approximated by the discrete-velocity-based closure law resulting from the $\delta$ model. This is illustrated in the final section \S \ref{sect:fundamental}, where we show that the fundamental diagrams, that is the flux-density relationships, obtained from the two models tend to coincide under grid refinement. Next we compare these diagrams with experimental data, finding that our models reproduce well experimental fundamental diagrams and thus they capture the characteristics of macroscopic  traffic flow. Further, we compute the macroscopic acceleration induced by the model, proving in particular its link with $\Dv$. We end the paper with a section summarizing the main results of this work, and proposing possible applications and further developments. Finally, the details of the matrix elements resulting from the discretization of the $\chi$ model are written in an Appendix.

\section{The general form of the kinetic model}
\label{sect:general-model}

In this section we briefly recall the general structure of a Boltzmann-type kinetic traffic model, which we will then specialize by prescribing a set of binary interaction rules in order to derive two models which differ only in the modeling of the acceleration interaction. Both models are defined on a continuous velocity space and they are characterized by a parameter $\Dv$ related to the typical acceleration of a vehicle.

We will focus on the space homogeneous case, because we want to investigate the structure of the collision term and of the resulting equilibrium distributions. In particular, we show that the simplified model ($\delta$ model) permits to describe the complexity of the equilibrium solutions with a very small number of discrete velocities.

Let $f=f(t,v):\mathbb{R}^+ \times V \to \mathbb{R}^+$ be the {\em kinetic distribution function}, where $V=[0,\vm]$ is the domain of the microscopic speeds and $\vm$ is the maximum  speed, which may depend on the mechanical characteristics of the vehicles, on imposed speed limits, environmental conditions (such as the quality of the road, the weather conditions, etc). The statistical distribution $f$ is such that $f(t,v) \dvu$ gives the number of vehicles with velocity in $[v,v+\dvu]$ at time $t$.

As usual, macroscopic quantities  are obtained as moments of the distribution function $f$ with respect to the velocity $v$: 
\[
	\rho(t)=\int_V f(t,v)\dvu,\quad (\rho u)(t)=\int_V vf(t,v)\dvu
\]
where $\rho$ is the {\em density}, i.e. the number of vehicles per unit length (tipically, kilometers), $u$ is the {\em macroscopic speed} and $\rho u$ is the {\em flux} of vehicles. Note that $\rho$ can also be interpreted as the reciprocal of the average distance between cars, see \cite{aw2000SIAP}.

In the homogeneous case, the Boltzmann-type equation can be written as
\begin{equation}
	\partial_t f(t,v)=Q[f,f](t,v)
	\label{eq:model1}
\end{equation}
where $Q[f,f](t,v)$ is the {\em collisional operator} which describes the relaxation to equilibrium due to the microscopic binary interactions among vehicles. For mass conservation to hold, the collision term must satisfy
\[
	\int_V Q[f,f](t,v)\dvu=0.
\]
In fact, this ensures that, in the space homogeneous case, the density remains constant in time.

The collisional operator is usually split into a gain term $G[f,f]$ and a loss term $L[f,f]$, that 
model statistically the  interactions which lead to gain or to loose the test speed $v$.
Denoting with $\Avvv$ the probability that the velocity $v\in V$ results from a microscopic interaction between {\em candidate vehicles} with velocity $\vb$ and {\em field vehicles} with speed $\va$, the model writes as an integro-differential equation
\begin{equation} 	\label{eq:model2}
  \partial_tf(t,v)=
  \underbrace{
	\int_V\int_V\eta(\vb,\va)\Avvv f(t,\vb)f(t,\va)\dvub \dvua
	}_{G[f,f](t,v)}
  -\underbrace{
	f(t,v)\int_V\eta(v,\va)f(t,\va)\dvua
	}_{L[f,f](t,v)}
\end{equation}
in which $\eta(\vb,\va)$ is the {\em interaction rate} possibly depending on the relative speed of the interacting vehicles, e.g. $\eta(\vb,\va)=|\vb-\va|$ as in~\cite{KlarWegener96,coscia2007IJNM}. Although such a choice would make the model richer, in~\cite{PgSmTaVg} we found that a constant interaction rate is already sufficient to account for many aspects of the complexity of traffic. Another possibility is to consider $\eta$  as dependent on the local congestion of the road, that is $\eta=\eta(\rho)$. However this is not relevant in the homogeneous case, where $\rho$ is constant, for then $\eta$ would affect only the relaxation time towards equilibrium. Thus in this paper we will set $\eta=$constant. 

\begin{notation}
In the whole paper, in order to shorten formulas, we adopt the following traditional shorthand
$f(t,\vb)=\fb$, $f(t,\va)=\fa$, etc.
Note in particular that in the space homogeneous case $\fa$ and $\fb$ are not different distribution functions, but the evaluation of the same $f(t,v)$ at two different points $\vb$ and $\va$. 
\end{notation}

We will suppose that $A$ depends also on the macroscopic density $\rho$ in order  to account for the influence of the macroscopic traffic conditions (local road congestion) on the microscopic interactions among vehicles, see \cite{PrigogineHerman,KlarWegener96,hertyillner09,PgSmTaVg}. Thus, we suppose that $A$ fulfills
\begin{assumption} \label{ass:A1}
\begin{align*}
	\Avvvr \geq 0,
	\quad\text{ and }
	\int_V \Avvvr \dvu=1,
	\quad\text{ for }
	\vb,\va,v\in V,\;\rho\in[0,\rho_{\max}]
\end{align*}
where $\rho_{\max}$ is the maximum density of vehicles, for instance the maximum number of vehicles per unit length in bumper-to-bumper conditions.
\end{assumption}

\begin{remark} \label{rem:consmass}
Any transition probability density $A$ that satisfies Assumption \ref{ass:A1} guarantees mass conservation since
\begin{align*}
	\partial_t\int_V f(t,v)\dvu=\int_V Q[f,f](t,v) \dvu&=\\
	\int_V\int_V  f(t,\vb) f(t,\va) \dvub \dvua&-\int_V f(t,v)\dvu \int_V f(t,\va) \dvua = 0.
\end{align*}
\end{remark}

\subsection{Choice of the probability density $A$}
\label{sec:modeling}
The probability density $A$ assignes a post-interaction speed in a non-deterministic way, consistently with the intrinsic stochasticity of the drivers' behavior. The construction of $A$ is at the core of a kinetic model. Here, it is obtained with a very small set of rules.

\begin{itemize}
\item If $\vb\leq\va$, i.e. the candidate vehicle is slower than the field vehicle, the post-interaction rules are:
\begin{description}
\item[Do nothing:] the candidate vehicle keeps its pre-interaction speed with probability $1-P_1$, thus $v=\vb$;
\item[Accelerate:] the candidate vehicle accelerates to a velocity $v>\vb$ with probability $P_1$.\end{description}
\item If $\vb>\va$, i.e. the candidate vehicle is faster than the field vehicle, the post-interaction rules are:
\begin{description}
\item[Accelerate:] in order to overtake the leading vehicle, the candidate vehicle accelerates to a velocity $v>\vb$ with probability $P_2$;
\item[Brake:] the candidate vehicle decelerates to the velocity $v=\va$ with probability $1-P_2$, thus following the leading vehicle.
\end{description}
\end{itemize}

From the previous rules, we observe that the probability density $A$ has a term which will be proportional to a Dirac delta function at $v=\vb$, due to the interaction which preserves the pre-interaction microscopic speed (the ``Do nothing'' alternative). Note that this is a ``false gain'' for the distribution $f$, because the number of vehicles with speed $v$ is not altered by this interaction. 

In the following, we assign the speed after braking as 
proposed  in~\cite{Prigogine61} and used also  in~\cite{FermoTosin13,FermoTosin14} in the context of a discrete velocity model. Namely, we suppose that if a vehicle brakes, interacting with a slower vehicle, it slows down to the speed $\va$ of the leading vehicle. Thus, after the interaction it gets the speed $v=\va$ without overtaking the leading field vehicle.
Instead, for the post-interaction speed due to acceleration we propose two different models.
\begin{description}
\item[Quantized acceleration ($\delta$ model):] the output velocity $v$ is obtained by accelerating instantaneously from $\vb$ to the velocity $\min\left\{\vb+\Dv,\vm\right\}$. Considering all possible outcomes, the resulting probability distribution, in this case, is
\begin{equation}
	\Avvvr=
	\begin{cases}
		(1-P_1)\delta_{\vb}(v)+P_1\delta_{\min\left\{\vb+\Dv,\vm\right\}}(v), &\text{if\; $\vb\leq\va$}\\
		(1-P_2)\delta_{\va}(v)+P_2\delta_{\min\left\{\vb+\Dv,\vm\right\}}(v), &\text{if\; $\vb>\va$}.
	\end{cases}
	\label{eq:Adelta}
\end{equation}
\item[Uniformly distributed acceleration ($\chi$ model):] the new velocity $v$ is uniformly distributed between $\vb$ and $\min\{\vb+\Dv,\vm\}$. On the whole, the resulting probability distribution becomes
\begin{equation}
	\Avvvr=
	\begin{cases}
		(1-P_1)\delta_{\vb}(v)+P_1\frac{\chi_{\left[\vb,\min\left\{\vb+\Dv,\vm\right\}\right]}(v)}{\min\left\{\vb+\Dv,\vm\right\}-\vb}, &\text{if\; $\vb\leq\va$}\\
		(1-P_2)\delta_{\va}(v)+P_2\frac{\chi_{\left[\vb,\min\left\{\vb+\Dv,\vm\right\}\right]}(v)}{\min\left\{\vb+\Dv,\vm\right\}-\vb}, &\text{if\; $\vb>\va$}.
	\end{cases}
	\label{eq:Achi}
\end{equation}
\end{description}

Note that the acceleration of a vehicle in  \eqref{eq:Adelta} is similar to the one assumed in 
\cite{FermoTosin13,FermoTosin14}, which however were based on a discrete velocity space.
In~\cite{FermoTosin13,FermoTosin14} the {\em acceleration parameter} $\Dv$ is chosen as the distance between two adjacent discrete velocities, thus $\Dv$ depends on the number of elements in the speed lattice. In this work, $\Dv$ is a physical parameter that represents the ability of a vehicle to change its pre-interaction speed $\vb$. With this choice, $\Dv$ does not depend on the discretization of the velocity space and the maximum acceleration is bounded, as in~\cite{Lebacque03}.  In contrast, deceleration can be larger than $\Dv$, and this fact reflects the hypothesis that drivers 
tend to brake immediately if the flow becomes more congested, while they react more slowly when they can accelerate (see the concept of {\em traffic hysteresis} in~\cite{ZhangMultiphase} and references therein).

The acceleration performed in the $\chi$ model has some points of contact with the microscopic rules prescribed in \cite{HertyPareschi10,KlarWegener96}. In \cite{HertyPareschi10}, however, the post-interaction speed is selected through a random process in the interval $[\vb, \vm]$. Here instead, the post-interaction speed is deterministic. In~\cite{KlarWegener96} instead, the velocity after acceleration is uniformly distributed over a range of speeds between $\vb$ and $\vb+\alpha(\vm-\vb)$, where $\alpha$ is supposed to depend on the local density; in a similar way, the output velocity from a braking interaction is assumed to be uniformly distributed in $[\beta\va,\va]$, with $\beta\in[0,1]$.

In the following, the probabilities $P_1$ and $P_2$ are taken as $P_1=P_2=:P$ and $P$ will be a function of the local density only, as assumed for instance in~\cite{PrigogineHerman} where $P=1-\rho/\rho_{\max}$. More generally, from a modeling point of view, $P$  should be a decreasing function of $\rho$, see also \cite{HertyIllner08} or \cite{Rosini}.
For instance in~\cite{PgSmTaVg} we have considered
\begin{equation}\label{eq:gamma_law}
P=1- \left( \frac{\rho}{\rho_{\max}} \right)^{\gamma},
\end{equation}
where $\gamma\in (0,1)$ can be chosen to better fit experimental data.
In~\cite{FermoTosin13} the choice $P=\alpha(1- \rho/\rho_{\max})$ is proposed, $\alpha\in[0,1]$ is a parameter describing environmental conditions, for instance road or weather conditions. In more sophisticated models, one could also choose $P$ as a function of the relative speed of interacting vehicles, but we will not explore this possibility in the present work.

The simplified choice $P_1=P_2$ and the interaction rules described at the beginning of this section guarantee the continuity of the transition probability \eqref{eq:Adelta} and \eqref{eq:Achi} along $\vb=\va$.

\begin{remark} \label{rem:falsegains}
Both choices~\eqref{eq:Adelta} and~\eqref{eq:Achi} for $A$ include terms of the form $\delta_{\vb}(v)$, which actually describe false gains mentioned above, because the velocity of the candidate vehicle does not change. They are automatically compensated by false losses, as it can be seen by rewriting the classical kinetic loss term of equation \eqref{eq:model2} in the form
$$
	L[f,f](t,v)= \int_V \int_V  \eta \delta_{\vb}(v) \fb \fa \dvua \dvub.
$$
\end{remark}


\section{The $\delta$ velocity model}
\label{sect:delta}

Now, we focus on the steady states of model \eqref{eq:Adelta}. We start with the existence of a particular set of equilibrium solutions of the continuous model, which are computed analytically. Next, we consider a finite volume discretization of the model, and we show that the discrete equilibria have precisely the structure found before analytically, thus suggesting that the particular set of equilibria found analytically are the only equilibria of the system.

It can be proven \cite{FregugliaTosin15} that the Cauchy problem associated to \eqref{eq:model2} is well posed provided the probability density $A$ is Lipschitz continuous with respect to $\vb$ and $\va$ in a suitable Wasserstein metric. This is indeed the case of the $A$ defined in \eqref{eq:Adelta}, with $P_1=P_2$. 

Using the expression~\eqref{eq:Adelta} for $A$, we rewrite the gain term in~\eqref{eq:model2} as
\begin{align*}
	G[f,f](t,v)=&\eta \int_V \int_V \left[(1-P)\delta_{\min\{\vb,\va\}}(v) + P\delta_{\min\left\{\vb+\Dv,\vm\right\}}(v)\right]\fb\fa \dvub \dvua 
\end{align*}
and the following important result on the existence of a particular class of stationary solutions holds. More precisely, it characterizes the equilibrium distributions having the form of linear combinations of Dirac's masses. This theorem establishes the connection represented by the left vertical arrow in Figure \ref{fig:diagcomm}.

\begin{theorem}\label{th:continuous-eq}
	Let $P$ be a given function of the density $\rho\in[0,\rho_{\max}]$ such that $P\in[0,1]$. Let $\{v_j\}_{j=1}^N$ be a set of velocities in $[0,\vm]$. The distribution function
	\[
	f^{\infty}(v)=\sum_{j=1}^N f^{\infty}_j \delta_{v_j}(v), \quad f^{\infty}_j > 0 \quad \forall\; j=1,\dots,N,
	\]
	with $\sum_{j=1}^N f_j^\infty=\rho$, is a weak stationary solution of the $\delta$ model provided  $v_j=v_1+j\Dv$, $j=1,\dots,N$.
\end{theorem}
\begin{proof}
	Without loss of generality suppose that $\{v_j\}_{j=1}^N$ is an ordered set of velocities such that $0\leq v_1<\dots<v_N\leq \vm$. Since the distribution function $f^{\infty}(v)=\sum_{j=1}^N f^{\infty}_j \delta_{v_j}(v)$ is a weak stationary solution of the $\delta$ model, it satisfies the following steady weak form of equation \eqref{eq:model2}:
	\begin{equation*}
	\int_V\int_V \left(\int_V \phi(v) \Avvvr \dvu\right) f^{\infty}(\vb) f^{\infty}(\va) \dvua \dvub - \rho \int_V \phi(v) f^{\infty}(v) \dvu=0,
	\end{equation*}
	where $\phi\in C_c(V)$ is a test function, with $C_c(V)$ the space of continuous functions having compact support contained in $V$, and the probability density $A$ is given in \eqref{eq:Adelta}.
	Substituting the expression of $f^{\infty}$ in the above equation we obtain
	\begin{equation}\label{eq:eq-th}
	\begin{aligned}
		&(1-P)\sum_{k=1}^N \sum_{h=1}^k \phi(v_h) f^{\infty}_h f^{\infty}_k + (1-P)\sum_{k=1}^N \sum_{h=k+1}^N \phi(v_k) f^{\infty}_h f^{\infty}_k \\
		&+ P \rho \sum_{h=1}^N \phi(\min\{v_h+\Dv,\vm\}) f^{\infty}_h - \rho \sum_{j=1}^N \phi(v_j) f^{\infty}_j=0.
	\end{aligned}
	\end{equation}
	The proof will be organized as follows: in order to determine an equation for the $f^{\infty}_j$'s, we consider a particular family of test functions $\phi_j$ defined as piecewise linear functions such that $\phi_j(v_j)=1$ and $\phi_j(v_i)=0$, $\forall\;i\neq j$; in this way, first we find an equation for $f^{\infty}_1$, then we show that $f_2^{\infty}\neq 0$ if $v_2=v_1+\Dv$ and finally by induction we prove that if $f^{\infty}_j\neq 0$, then $v_j=v_1+j\Dv$, for some $j\in\{1,\dots,N\}$.\\
	Let $j=1$, equation \eqref{eq:eq-th} with $\phi=\phi_1$ becomes
	\[
		(1-P)\sum_{k=1}^N f^{\infty}_1 f^{\infty}_k + (1-P) \sum_{h=2}^N f^{\infty}_h f^{\infty}_1 + P \rho \sum_{h=1}^N \phi_1(v_1+\Dv) f^{\infty}_h - \rho f^{\infty}_1=0.
	\]
	Due to the particular construction of $\phi_1$ and using $\sum_{j=1}^N f_j^\infty=\rho$, the above expression reduces to
	\[
		-(1-P)(f^{\infty}_1)^2+(1-2P)\rho f^{\infty}_1=0
	\]
	which admits the two roots $f^{\infty}_1=0$ and $f^{\infty}_1=\rho\frac{1-2P}{1-P}$. If  $P>1/2$, only $f^{\infty}_1=0$ is acceptable, because the other root is negative. If instead $P<1/2$, both roots can be accepted, but only $f^{\infty}_1=\rho\frac{1-2P}{1-P}>0$ is stable. This argument will be used for selecting a single root throughout the proof.\\
	Now, let $j=2$ and $\phi=\phi_2$. Equation \eqref{eq:eq-th} writes as
	\begin{equation}\label{eq:f2-th}
		-(1-P)(f^{\infty}_2)^2+\left[(1-2P)\rho-2(1-P)f^{\infty}_1\right]f^{\infty}_2+P\rho\sum_{h=1}^N \phi_2(\min\{v_h+\Dv,\vm\}) f^{\infty}_h = 0.
	\end{equation}
	Note that $\phi_2(\min\{v_h+\Dv,\vm\})=1$ if $h=1$ and $v_1+\Dv=v_2$. While $\phi_2(\min\{v_h+\Dv,\vm\})=0$ otherwise, due to the particular choice of $\phi_2$ which is centered in $v_2$ and can be taken with support smaller than $2\abs{v_2-v_1-\Dv}$ around this point. Then, if $v_2\neq v_1+\Dv$, the constant coefficient of \eqref{eq:f2-th} is zero for all $P\in[0,1]$ and the solutions of the equation are $f_2^\infty=0$ or $f_2^\infty=\rho \frac{1-2P}{1-P}-2f_1^\infty$. Exploiting the structure of $f_1^\infty$ and the fact that $f_j^\infty \geq 0$, only $f_2^\infty=0$ is the admissible solution for all values of $P\in[0,1]$.\\
	Instead, if $v_2=v_1+\Dv$, the third term in equation \eqref{eq:f2-th} is $P\rho f^{\infty}_1\geq 0$, for all $P\in[0,1]$. More precisely, if $P\geq 1/2$ then $f^{\infty}_1=0$, thus the constant coefficient vanishes and again one concludes that $f_2^\infty=0$. While if $P<1/2$, $P\rho f^{\infty}_1$ is positive and since the discriminant $\mathcal{D}=\left((1-2P)\rho-2(1-P)f^{\infty}_1\right)^2+4P(1-P)\rho f^{\infty}_1$ of equation~\eqref{eq:f2-th} is positive and the leading coefficient is negative, the equation has two real roots with opposite signs. Therefore
	\[
		f^{\infty}_2=\frac{-(1-2P)\rho+2(1-P)f^{\infty}_1-\sqrt{\mathcal{D}}}{-2(1-P)}>0
	\]
	and this is the only case in which $f_2^\infty$ can be non-zero.\\
	We now proceed by induction. Suppose that $v_k-v_1$ is an integer multiple of $\Dv$ and
	\[
	f^{\infty}_k=\begin{cases}
	0, & \text{if $P\geq 1/2$}\\
	\frac{-(1-2P)\rho+2(1-P)\sum_{l=1}^{k-1}f^{\infty}_l-\sqrt{\mathcal{D}_k}}{-2(1-P)} & \text{if $P<1/2$}
	\end{cases}
	\]
	for all $k=3,\dots,j-1$, where $\mathcal{D}_k=\left((1-2P)\rho-2(1-P)\sum_{l=1}^{k-1}f^{\infty}_l\right)^2+4P(1-P)\rho\sum_{l=1}^{k-1}f^{\infty}_l$. We show that $f^{\infty}_j$ can be non-zero only if $v_j=v_1+j\Dv$, for $j\in\{1,\dots,N\}$.\\
	Taking the test function $\phi=\phi_j$, the equation for $f^{\infty}_j$ writes as
	\begin{equation}\label{eq:fj-th}
		-(1-P)(f^{\infty}_j)^2+\left[(1-2P)\rho-2(1-P)\sum_{l=1}^{j-1} f^{\infty}_l\right]f^{\infty}_j + P\rho \sum_{h=1}^N \phi_j(\min\{v_h+\Dv,\vm\}) f^{\infty}_h=0.
	\end{equation}
	Note that $\phi_j(\min\{v_h+\Dv,\vm\})=1$ if $h=j-1$ and $v_{j-1}+\Dv=v_j$. While $\phi_j(\min\{v_h+\Dv,\vm\})=0$ otherwise, due to the particular choice of $\phi_j$ which is centered in $v_j$ and can be taken with support smaller than $2\abs{v_j-v_{j-1}-\Dv}$ around this point. If $v_j=v_1+j\Dv$, then $v_j=v_{j-1}+\Dv$ and the above equation becomes
	\[
		-(1-P)(f^{\infty}_j)^2+\left[(1-2P)\rho-2(1-P)\sum_{l=1}^{j-1} f^{\infty}_l\right]f^{\infty}_j + P\rho f^{\infty}_{j-1}=0
	\]
	which has two real roots. If $P\geq 1/2$, using the inductive hypothesis $f_k^\infty=0$ for all $k\leq j-1$. Thus $f_j^\infty=\rho\frac{1-2P}{1-P}$ (which is again not acceptable since it is negative) or $f_j^\infty=0$, confirming the induction. If $P<1/2$, we have two real roots with opposite signs, so one proves that $f^{\infty}_j$ can be chosen strictly positive.\\
	If instead $v_j\neq v_1+j\Dv$ then the constant term of equation \eqref{eq:fj-th} is zero and the two roots are $f_j^{\infty}=0$ and $f^{\infty}_j=S_j=\frac{(1-2P)\rho-2(1-P)\sum_{l=1}^{j-1} f^{\infty}_l}{1-P}$ which is negative for all values of $P\in[0,1]$ using the Lemma \ref{th:lemma} below.
\end{proof}

In the previous proof we use the following technical fact.

\begin{lemma}\label{th:lemma}
	Let $P$ be a given function of the density $\rho\in[0,\rho_{\max}]$ such that $P\in[0,1]$. Consider $\{f_j\}_{j=1}^K\in\mathbb{R}$ defined as
	\[
		f_j=\frac{2(1-P)\sum_{l=1}^{j-1}f_l-(1-2P)\rho-C_j}{-2(1-P)}
	\]
	with $C_j>0$. Assume that $f_j$ is positive for all $j$. Then
	\[
		S_k=\frac{(1-2P)\rho-2(1-P)\sum_{l=1}^{k-1}f_l}{1-P}<0
	\]
	 for all $2\leq k \leq K$.
\end{lemma}
\begin{proof}
	We proceed by induction on $k$. Let $k=2$
	\[
		S_2=\frac{(1-2P)\rho-2(1-P)f_1}{1-P}=-\frac{C_1}{1-P}<0.
	\]
	Suppose $S_j<0$, $j=3,\dots,k-1$, then $S_k<0$. In fact
	\[
		S_k=S_{k-1}-2f_{k-1}<0.
	\qedhere
	\]
\end{proof}

\subsection{Discretization of the model}
\label{sec:discretedelta}

Observe that Theorem \ref{th:continuous-eq} ensures the existence of a class of steady solutions for the $\delta$ model which are characterized by the fact that the total mass of vehicles on the road is distributed only on the velocities which are multiples of $\Dv$. We cannot prove the uniqueness of such a class of steady solutions. However we can show by numerical evidence that the asymptotic distributions of the $\delta$ model are only of the type stated by Theorem \ref{th:continuous-eq}. Thus, in this subsection we introduce a discretization of the model \eqref{eq:Adelta}.

To this end, the explicit formulation of the gain term is now useful. Notice that the Dirac delta function at $v=\min\left\{\vb+\Dv,\vm\right\}$ can be split as
\[
	\delta_{\min\left\{\vb+\Dv,\vm\right\}}(v)=
	\begin{cases}
		\delta_{\vb+\Dv}(v), &\text{if\; $\vb\in\left[0,\vm-\Dv\right]$}\\
		\delta_{\vm}(v), &\text{if\; $\vb\in\left(\vm-\Dv,\vm\right]$}
	\end{cases}
\]
because the velocity jump of size $\Dv$, leading to the output velocity $v=\vb+\Dv$, can be performed only if $\vb \leq \vm-\Dv$. If instead  $\vb\in\left(\vm-\Dv,\vm\right]$, the post-interaction velocity will be $v=\vm$.
Thus the gain term of the $\delta$ model can be written as
\begin{equation}\label{eq:delta_model}
\begin{aligned}
	G[f,f](t,v)=&\eta(1-P) f(t,v) \left[ \int_v^{\vm}\fa \dvua + \int_v^{\vm}\fb \dvub \right]\\
	&+ \eta P \rho \left[ f(t,v-\Dv)H_{\Dv}(v) + \delta_{\vm}(v)\int_{\vm-\Dv}^{\vm}\fb \dvub\right]
\end{aligned}
\end{equation}
where $H_{\alpha}(x)$ denotes the Heaviside step function with jump located in $\alpha$.
The last term in the expression of $G$ means that, as a result of the microscopic interactions, the mass $P\rho\int_{\vm-\Dv}^{\vm}\fb \dvub$ is allocated entirely to the speed $\vm$. Note that, in space-nonhomogeneous models, $\fb$ and $\fa$ may refer to distributions evaluated at different locations in space, see for instance \cite{KlarWegener96} and \cite{klar1997Enskog}. For this reason we keep the integrals over field and candidate particles separate.

Suppose for simplicity that the acceleration parameter  $\Dv$ satisfies $\Dv=\vm/T$ with $T\in\mathbb{N}$. We consider a discretization of the velocity space defining the {\em velocity cells} $I_j=[(j-\tfrac32)\dv, (j-\tfrac12)\dv]\cap [0,\vm]$, for $j=1,\dots,N$. Note that all cells have amplitude $\dv=\vm/(N-1)$ except $I_1$ and $I_N$ which have amplitude $\dv/2$.

We consider a piecewise constant approximation of the kinetic distribution so that
\begin{equation} \label{eq:f:discrete}
f(t,v) \approx f_N(t,v) = \sum_{j=1}^N f_j(t) \frac{\chi_{I_j}(v)}{|I_j|},
\end{equation}
where $f_j$ represents the number of vehicles traveling with velocity $v \in I_j$. 

By integrating the kinetic equation \eqref{eq:model1} over the cells $I_j$ and using $f_N(t,v)$ in place of $f(t,v)$ we obtain the following system of ordinary differential equations
\begin{equation}\label{eq:discretesys}
		\ds{f'_j(t)=Q_j[f,f](t):=\int_{I_j}Q[f,f](t,v)\dvu},
\end{equation}
whose initial conditions $f_1(0),\dots,f_N(0)$ are such that
\[
	\sum_{j=1}^N f_j(0)=\int_V f(t=0,v)\dvu=\rho
\]
and $\rho$ is the initial density, which remains constant during the time evolution in the spatially homogeneous case.

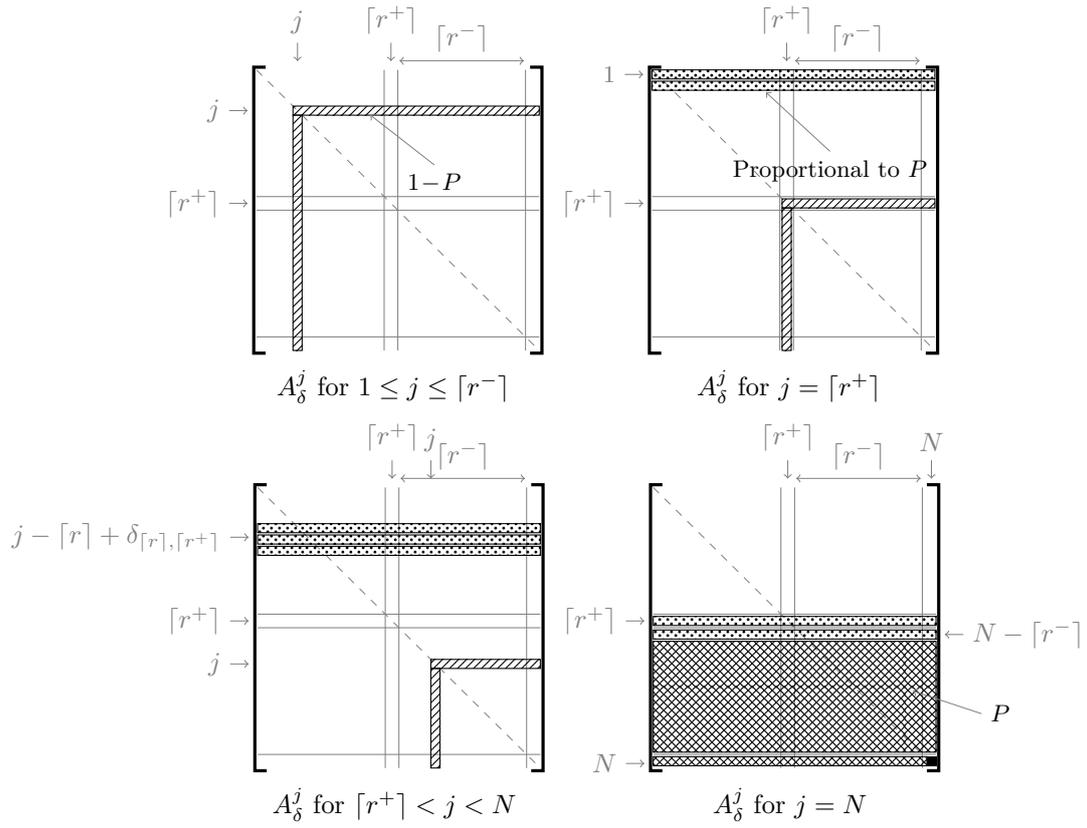
\begin{figure}
	\centering
	\begin{tikzpicture}[y={(0,-1cm)},scale=1.2]
	\begin{scope}
	\skeletongeneric
	\node[anchor=center] at (1.5,3.5) {$A^j_{\delta}$ for $1\leq j\leq \pisup{\rme}$};
	\filldraw[pattern=\patUnoMenoP] (0.4,0.4) rectangle (3.1,0.5);
	\filldraw[pattern=\patUnoMenoP] (0.4,0.5) rectangle (0.5,3.1);
	\draw[help lines,->,yshift=-0.45cm]
	(-0.3,0) node[left]{$j$} -- (-0.1,0);
	\draw[help lines,->,xshift=0.45cm]
	(0,-0.3) node[above]{$j$} -- (0,-0.1);
	\draw[help lines,->] (1.95,1.05) node[below,black]{\small$1\!-\!P$} -- (1.25,0.5);
	\end{scope}
	\end{tikzpicture}
	\begin{tikzpicture}[y={(0,-1cm)},scale=1.2]
	\begin{scope}[xshift=4.5cm]
	\skeletongeneric
	\node[anchor=center] at (1.5,3.5) {$A^j_{\delta}$ for $j=\pisup{\rp}$};
	\filldraw[pattern=\patUnoMenoP] (1.425,1.425) rectangle (3.1,1.525);
	\filldraw[pattern=\patUnoMenoP] (1.425,1.525) rectangle (1.525,3.1);
	\filldraw[pattern=\patProportionalP] (0,0) rectangle (3.1,0.1);
	\filldraw[pattern=\patProportionalP] (0,0.125) rectangle (3.1,0.225);
	\draw[help lines,->,yshift=-0.05cm]
	(-0.3,0) node[left]{\small$1$} -- (-0.1,0);
	\draw[help lines,->] (1.95,0.9) node[below,black]{{\small Proportional to $P$}} -- (1.25,0.25);
	\end{scope}
	\end{tikzpicture}
	\begin{tikzpicture}[y={(0,-1cm)},scale=1.2]
	\begin{scope}[xshift=9cm]
	\skeletongeneric
	\node[anchor=center] at (1.5,3.5) {$A^j_{\delta}$ for $\pisup{\rp}<j<N$};
	\filldraw[pattern=\patProportionalP] (0,0.4) rectangle (3.1,0.5);
	\filldraw[pattern=\patProportionalP] (0,0.525) rectangle (3.1,0.625);
	\filldraw[pattern=\patProportionalP] (0,0.65) rectangle (3.1,0.75);
	\filldraw[pattern=\patUnoMenoP] (1.9,1.9) rectangle (3.1,2);
	\filldraw[pattern=\patUnoMenoP] (1.9,2) rectangle (2,3.1);
	\draw[help lines,->,yshift=-1.95cm]
	(-0.3,0) node[left]{$j$} -- (-0.1,0);
	\draw[help lines,->,xshift=1.9cm]
	(0,-0.3) node[above]{$j$} -- (0,-.1);
	\draw[help lines,->,yshift=-0.55cm]
	(-0.3,0) node[left]{$j-\pisup{r}+\delta_{\pisup{r},\pisup{\rp}}$} -- (-0.1,0);
	\end{scope}
	\end{tikzpicture}
	\begin{tikzpicture}[y={(0,-1cm)},scale=1.2]
	\begin{scope}[xshift=13.5cm]
	\skeletongeneric
	\node[anchor=center] at (1.5,3.5) {$A^j_{\delta}$ for $j=N$};
	
	\filldraw[pattern=\patProportionalP] (0,1.425) rectangle (3.1,1.525);
	\filldraw[pattern=\patProportionalP] (0,1.575) rectangle (3.1,1.675);
	
	\filldraw[pattern=\patP] {(0,1.7) rectangle (3.1,2.925)};
	
	\fill[pattern=\patUno] (3.0,2.975) rectangle (3.1,3.075);
	
	\filldraw[pattern=\patP] (0,2.975) rectangle (3.0,3.075);
	
	\draw[help lines,->,yshift=-1.625cm]
	(3.4,0) node[right]{$N-\pisup{\rme}$} -- (3.2,0);
	\draw[help lines,->,yshift=-3.05cm]
	(-0.3,0) node[left]{$N$} -- (-0.1,0);
	\draw[help lines,->,xshift=3.05cm]
	(0,-0.3) node[above]{$N$} -- (0,-0.1);
	\draw[help lines,->] (3.6,2.5) node[right,black]{\small$P$} -- (2.85,2.25);
	\end{scope}
	\end{tikzpicture}
	\caption{Structure of the probability matrices of the $\delta$ model, with $\Dv=\vm/2$.\label{fig:generic-matrices}}
\end{figure}

We set $r:=\Dv/\dv\in\mathbb{R}^+$ and we define $\rp:=r+\frac12$, $\rme:=r-\frac12$. Then $I_{\pisup{\rp}}$ is the cell which contains $v=\Dv$, where $\pisup{\rp}$ denotes the integer part of $\rp$. By computing the right hand side of the ODE system \eqref{eq:discretesys}, we obtain explicitly
\begin{subequations} \label{eq:Jdelta:allj}
	\begin{align}\label{eq:Jsmallj}
	\frac{1}{\eta}Q_j[f,f](t)&=
	(1-P)f^jf_j
	+(1-P)f_j\sum_{k=j+1}^{N}f^k\\
	&+(1-P)f^j\sum_{h=j+1}^{N}f_h-f_j\sum_{k=1}^{N}f^k,
	\qquad\quad\quad\;
	\text{for $j=1,\dots,\pisup{\rme}$} \nonumber
	\\ \label{eq:Jrp1}
	\frac{1}{\eta}Q_j[f,f](t)&=
	(1-P)f^jf_j+(1-P)f_j\sum_{k=j+1}^{N}f^k
	+(1-P)f^j\sum_{h=j+1}^{N}f_h\\
	&+\uwave{P\rho\left[2f_1\min\{\frac12,\pisup{\rp}-\frac12-r\}+\delta_{\pisup{r},\pisup{\rme}}f_2(\pisup{\rme}-r)\right]}\nonumber\\
	&-f_j\sum_{k=1}^{N}f^k,\qquad\qquad\qquad\qquad\qquad\qquad\qquad \text{for $j=\pisup{\rp}$}\nonumber
	\\ \label{eq:Jallj}
	\frac{1}{\eta}Q_j[f,f](t)&=
	(1-P)f^jf_j+(1-P)f_j\sum_{k=j+1}^{N}f^k
	+(1-P)f^j\sum_{h=j+1}^{N}f_h\\
	&+\uwave{P\rho\left[(1+\delta_{\pisup{r},\pisup{\rp}}\delta_{j,\pisup{\rp}+1})f_{j-\pisup{r}}(1+r-\pisup{r})+f_{j-\pisup{r}+1}(-r+\pisup{r})\right]} \nonumber
	\\
	&-f_j\sum_{k=1}^{N}f^k,\qquad\qquad\qquad\qquad\qquad\qquad\qquad \text{for $j=\pisup{\rp}+1,\dots,N-1$}\nonumber
	\\ \label{eq:Jlastj}
	\frac{1}{\eta}Q_N[f,f](t)&=
	(1-P)f^Nf_N\\
	&+\uwave{P\rho\delta_{\pisup{r},\pisup{\rp}}\left[f_{N-\pisup{\rp}}(r-\pisup{\rme})+f_{N-\pisup{\rme}}(\pisup{\rp}-\frac12-r)\right]}\nonumber\\
	&\uwave{P\rho f_{N-\pisup{\rme}}\left[\frac12\delta_{\pisup{r},\pisup{\rme}}+(r-\pisup{\rme}+\frac12)\right]+P\rho\sum_{h=N-\pisup{\rp}+2}^N f_h} \nonumber \\
	&-f_N\sum_{k=1}^{N}f^k. \nonumber
	\end{align}
\end{subequations}
where here $\delta_{i,j}$'s are Kronecker's delta's. The terms with a wavy underline are those deriving from the acceleration term.
In the formulae above, the position of the index of the components of $\f=[f_1,\dots,f_N]^{\tr}\in\mathbb{R}^N$ distinguishes the distribution of the field and of the candidate vehicles: bottom right for the candidate vehicles (as in $f_h$), top right for the field vehicles (as in $f^k$).
In vector form:
\begin{equation}
	\frac{d}{dt}f_j=\eta \left[\f^{\tr} A^j_{\delta} \f - \f^{\tr} \mathbf{e}_j \mathbf{1}^{\tr}_N \f \right],\quad j=1,\dots,N
    \label{eq:delta_vectsys}
\end{equation}
where $\mathbf{e}_j\in\mathbb{R}^N$ denotes the vector with a $1$ in the $j$-th component and $0$'s elsewhere, $\mathbf{1}^{\tr}_N=[1,\dots,1]\in\mathbb{R}^N$. The matrices $A^j_{\delta}$ have a sparse structure, shown in Fig. \ref{fig:generic-matrices} in which the nonzero elements are shaded with different hatchings, corresponding to the different values of the elements, as indicated in the panels in which they appear for the first time.

As it can be checked using \eqref{eq:Jdelta:allj}, these matrices are stochastic with respect to the index $j$, i.e.
 $\sum_{j=1}^N \left(A^j_{\delta}\right)_{hk}=1$, $\forall\,h,k\in\left\{1,\dots,N\right\}$. This property comes from Assumption \ref{ass:A1}, and it guarantees mass conservation. 

Recall that the elements of the matrix $\left(A_{\delta}^j\right)_{hk}$ are the probabilities that the candidate vehicle with velocity 
in $I_h$, 
interacting with a field vehicle with velocity 
in $I_k$,
acquires 
a velocity in $I_j$.
The fact that these matrices are sparse means that 
a velocity in $I_j$ 
can be acquired only for special values of the velocity of candidate and field vehicles. In particular, the $j$-th row of the matrix $A^j_{\delta}$ contains the probability of what we called ``false gains'' in Remark \ref{rem:falsegains}, that is the probability that the candidate vehicle {\em does not} change its speed. The non zero elements of the $j$-th column are the probabilities that a candidate vehicle acquires 
a speed in $I_j$ 
by braking down to the speed of the leading vehicle. The non zero rows, located at $h=j-\pisup{r}+\delta_{\pisup{r},\pisup{\rp}}$, $h-1$ and $h+1$, contain the probabilities that the candidate vehicle accelerates by $\Delta v$, acquiring therefore 
a velocity in $\vb+\Dv\in I_j$ starting from a velocity $\vb$ in $I_{h-1}$, $I_h$ or $I_{h+1}$.
The band between the rows $j-\pisup{r}+\delta_{\pisup{r},\pisup{\rp}}+2$ and $j-1$ is filled with zeros, because in the $\delta$ model the acceleration is quantized. As we will see in Section \ref{sec:discretechi}, this band will be filled by non zero elements in the $\chi$ model, where the acceleration is distributed uniformly between $[0,\Delta v]$.

\begin{table}[!t]
	\begin{center}
		\begin{tabular}{r|l|c}
			Parameter & Description & Definition \\
			\hline
			\hline
			$N$ & number of discrete speeds & \\                          
			$\dv$ & cell amplitude & $\dv=\frac{\vm}{N-1}$ \\
			$r$ & ratio between the speed jump $\Dv$ and the cell amplitude $\dv$ & $r=\frac{\Dv}{\dv}$ \\
			$T$ & number of speed jumps $\Dv$ contained in $[0,\vm]$ & $T=\frac{\vm}{\Dv}$\\
			\hline
		\end{tabular}
	\end{center}
	\caption{Table of the numerical parameters.\label{tab:parameters}}
\end{table}

In Table \ref{tab:parameters} we summarize the numerical parameters introduced in order to discretize the continuous-velocity model.

\paragraph{$\dv$ as integer sub-multiple of $\Dv$.}

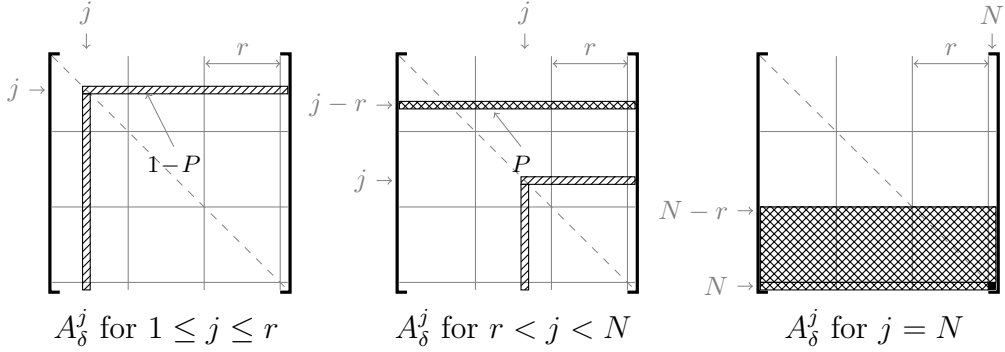
\begin{figure}
	\centering
	\begin{tikzpicture}[y={(0,-1cm)},scale=1]
	\begin{scope}
	\skeleton
	\node[anchor=center] at (1.5,3.6) {\large$A^j_{\delta}$ for $1\leq j\leq r$};
	\filldraw[pattern=\patUnoMenoP] (0.4,0.4) rectangle (3.1,0.5);
	\filldraw[pattern=\patUnoMenoP] (0.4,0.5) rectangle (0.5,3.1);
	\draw[help lines,->,yshift=-0.45cm]
	(-0.3,0) node[left]{$j$} -- (-0.1,0);
	\draw[help lines,->,xshift=0.45cm]
	(0,-0.3) node[above]{$j$} -- (0,-0.1);
	\draw[help lines,->] (1.6,1.2) node[below,black]{\small$1\!-\!P$} -- (1.25,0.5);
	\end{scope}
	\end{tikzpicture}
	\begin{tikzpicture}[y={(0,-1cm)},scale=1]
	\begin{scope}[xshift=4cm]
	\skeleton
	\node[anchor=center] at (1.5,3.6) {\large$A^j_{\delta}$ for $r< j< N$};
	\filldraw[pattern=\patP] (0,0.6) rectangle (3.1,0.7);
	\filldraw[pattern=\patUnoMenoP] (1.6,1.6) rectangle (3.1,1.7);
	\filldraw[pattern=\patUnoMenoP] (1.6,1.7) rectangle (1.7,3.1);
	\draw[help lines,->,yshift=-1.65cm]
	(-0.3,0) node[left]{$j$} -- (-0.1,0);
	\draw[help lines,->,yshift=-0.65cm]
	(-0.3,0) node[left]{$j-r$} -- (-0.1,0);
	\draw[help lines,->,xshift=1.65cm]
	(0,-0.3) node[above]{$j$} -- (0,-0.1);
	\draw[help lines,->] (1.6,1.2) node[below,black]{\small$P$} -- (1.25,0.75);
	\end{scope}
	\end{tikzpicture}
	\begin{tikzpicture}[y={(0,-1cm)},scale=1]
	\begin{scope}[xshift=8cm]
	\skeleton
	\node[anchor=center] at (1.5,3.6) {\large$A^j_{\delta}$ for $j=N$};
	\filldraw[pattern=\patP] (0,2) rectangle (3.1,3);
	\filldraw[pattern=\patP] (0,3) rectangle (3,3.1);
	\fill[pattern=\patUno] (3,3) rectangle (3.1,3.1);
	\draw[help lines,->,yshift=-2.05cm]
	(-0.3,0) node[left]{$N-r$} -- (-0.1,0);
	\draw[help lines,->,yshift=-3.05cm]
	(-0.3,0) node[left]{$N$} -- (-0.1,0);
	\draw[help lines,->,xshift=3.05cm]
	(0,-0.3) node[above]{$N$} -- (0,-0.1);
	\end{scope}
	\end{tikzpicture}
	\caption{Structure of the probability matrices of the $\delta$ model with $\dv$ integer sub-multiple of $\Dv$.\label{fig:matricidelta}}
\end{figure}

\begin{figure}[t!]
\centering
\includegraphics[width=0.49\textwidth]{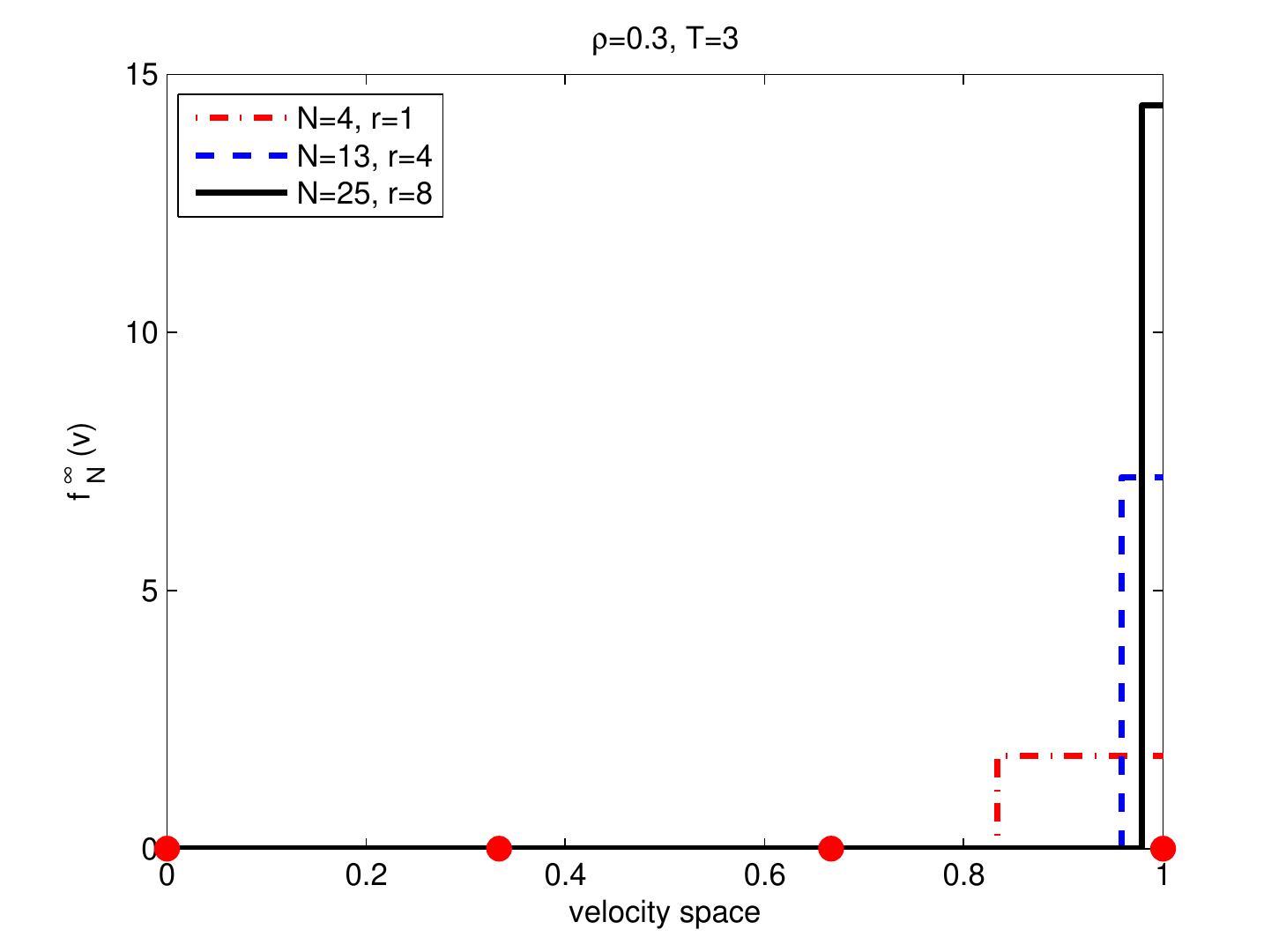}
\hfill
\includegraphics[width=0.49\textwidth]{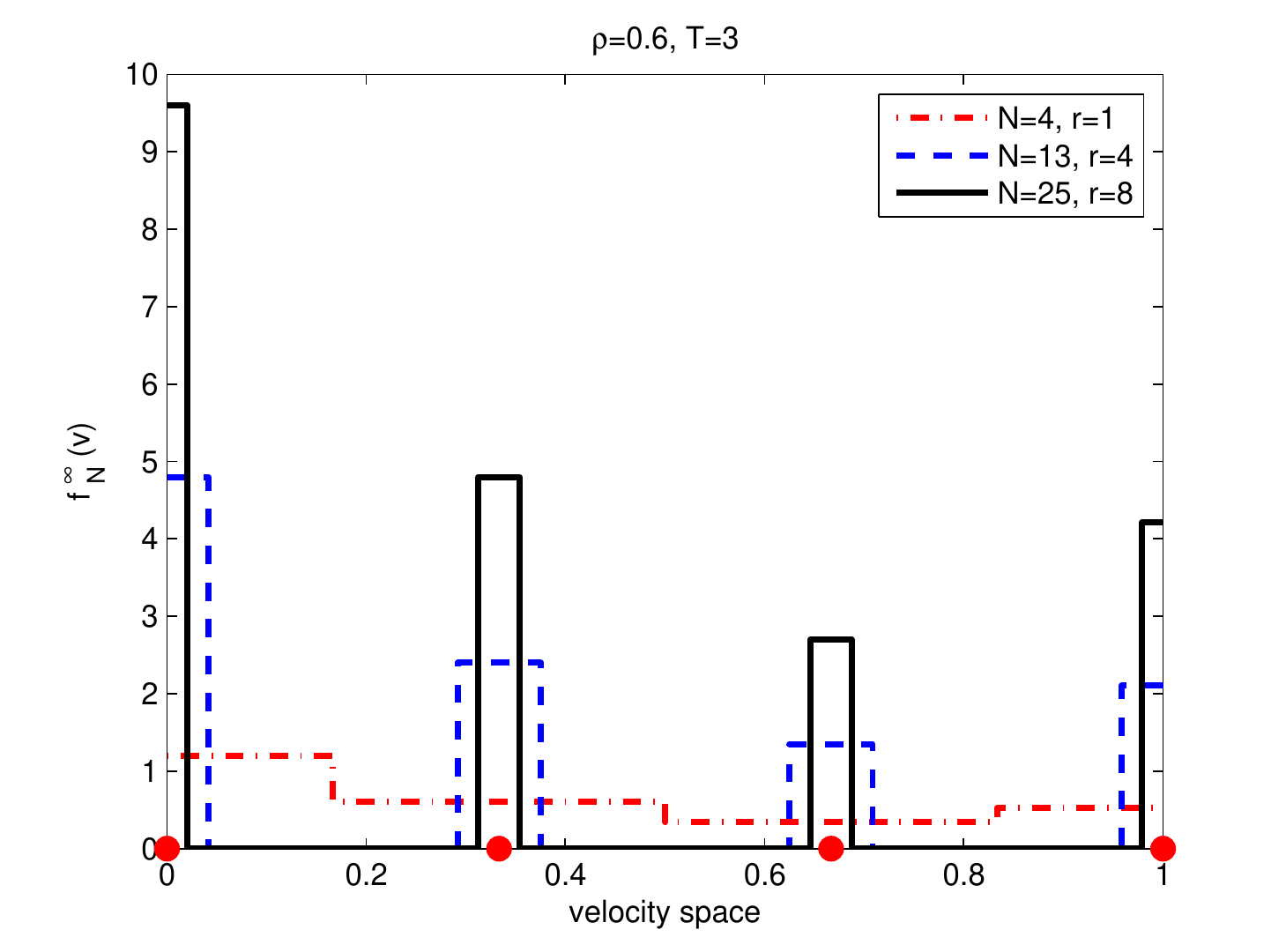}
\\
\includegraphics[width=0.49\textwidth]{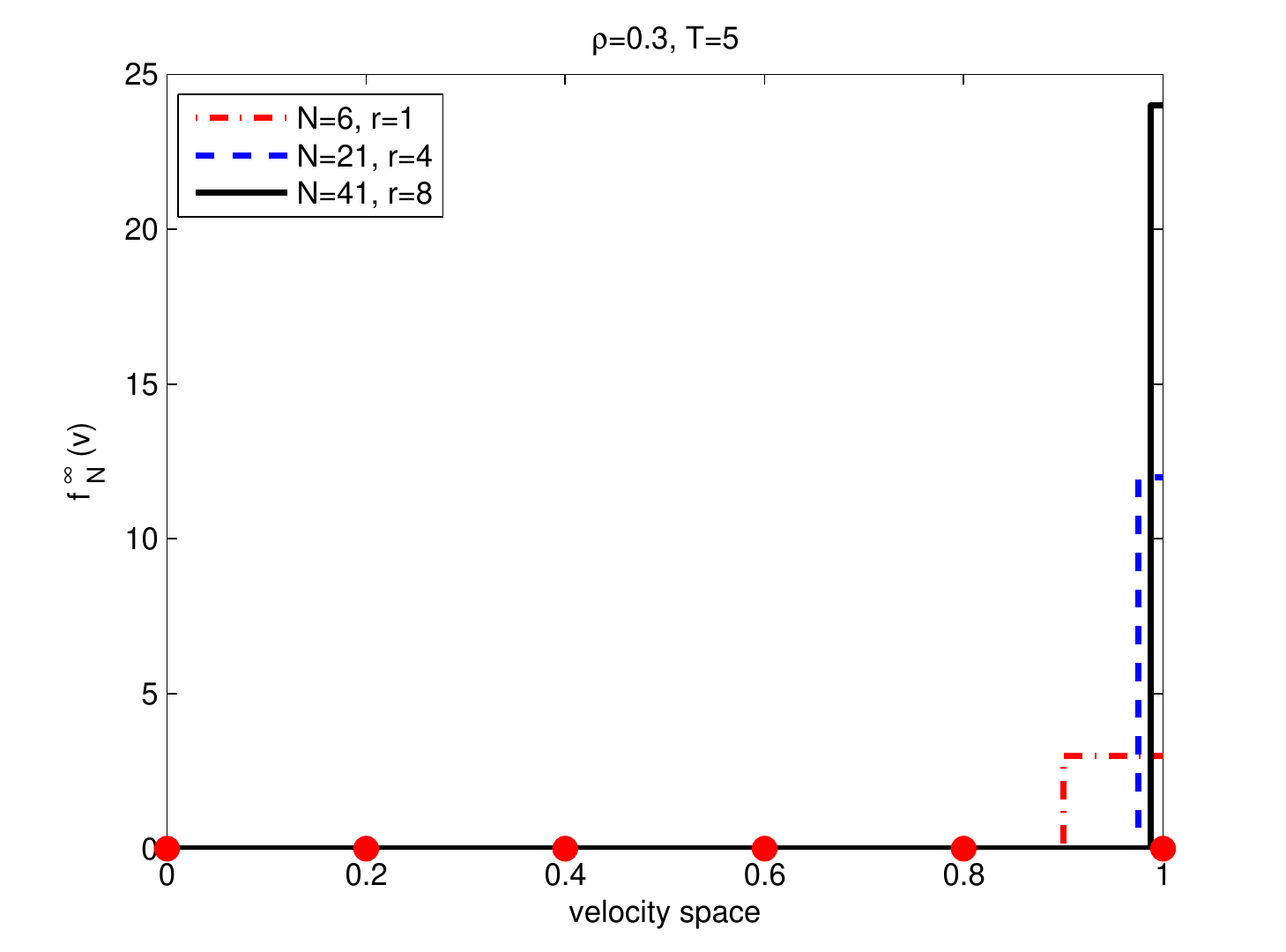}
\hfill
\includegraphics[width=0.49\textwidth]{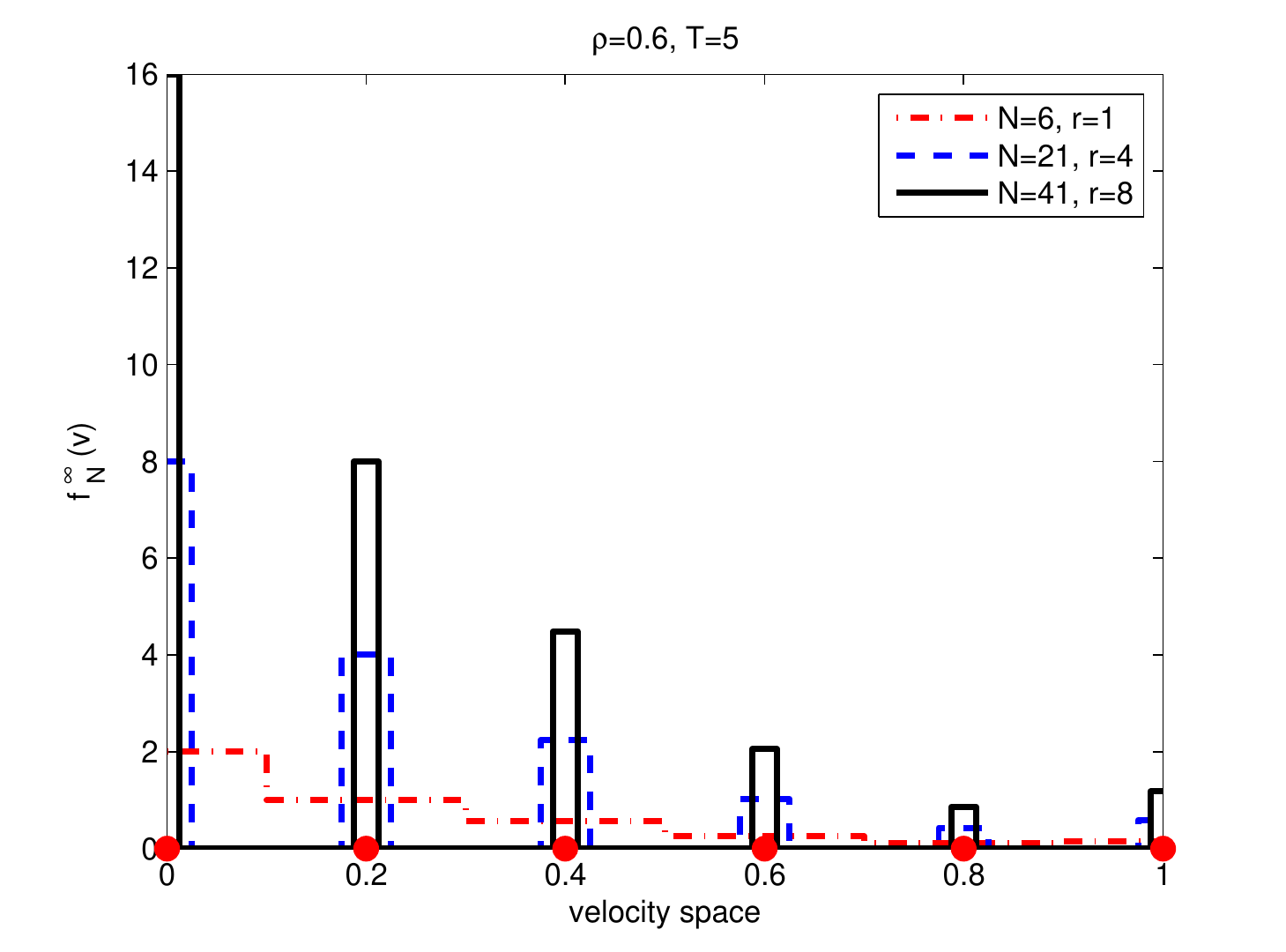}
\caption{Approximation of the asymptotic kinetic distribution function obtained with two acceleration terms $\Dv=1/T$, $T=3$ (top), $T=5$ (bottom), and $N=rT+1$ velocity cells, with $r\in\left\{1,4,8\right\}$; $\rho=0.3$ (left) and $\rho=0.6$ (right) are the initial densities. We mark with red circles on the x-axes the center of the $T+1$ cells obtained with $r=1$.\label{fig:delta_eq}}
\end{figure}

For the special choice of the velocity grid which ensures that $r=\Dv/\dv\in\mathbb{N}$, then the formulae \eqref{eq:Jdelta:allj} simplify and the resulting interaction matrices are given in Figure \ref{fig:matricidelta}. Notice that the rows $j-r\pm 1$ are filled with zeros. In fact, since $\dv$ is an integer sub-multiple of $\Dv$, a velocity in $I_j$ can be obtained as result of an acceleration only if the pre-interaction speed is a velocity in $I_{j-r}$.
 
The structure of the matrices $A^j_{\delta}$ determines the equilibrium of the discrete model \eqref{eq:delta_vectsys}. In Figure~\ref{fig:delta_eq} we show the function $f^{\infty}_N(v)=\lim_{t\to\infty}f_N(t,v)$ obtained by integrating numerically the system of equations up to steady state, for a few typical cases. In all numerical tests we take $P=1-\rho$, with $\vm=\rho_{\max}=1$. As initial macroscopic densities, we choose $\rho=0.3,0.6$ (plots to the left and right of the figure). We consider two values for the acceleration parameter, $\Dv=\vm/T$, $T=3,5$ (top and bottom of the figure). The number of velocities in the grid is taken as $N=rT+1$, with $r\in\left\{1,4,8\right\}$. The three curves in each plot contain the data for the cell averages of the equilibrium distribution for the different values of $r$. It is clear that in all cases, $f^{\infty}_N(v)$ is a function of the density $\rho$ and it approaches a series of delta functions, centered in the velocities which are multiples of $\Dv$ and indicated in the picture by red dots on the horizontal axis.

This means that, as $\dv\to 0$, only a finite number $T+1$ of velocities carry a non-zero mass of vehicles at equilibrium. More precisely, the discrete asymptotic function $f^{\infty}_N(v)$ is different from zero only in the $T+1$ cells $I_1,I_r,I_{2r},\dots,I_N$. Therefore, as time goes to infinity the number of nonzero values of the $f_j$'s appearing in \eqref{eq:f:discrete} is univocally determined by the acceleration term $\Dv=\vm/T$. 

\begin{figure}[t!]
\centering
\includegraphics[width=0.32\textwidth]{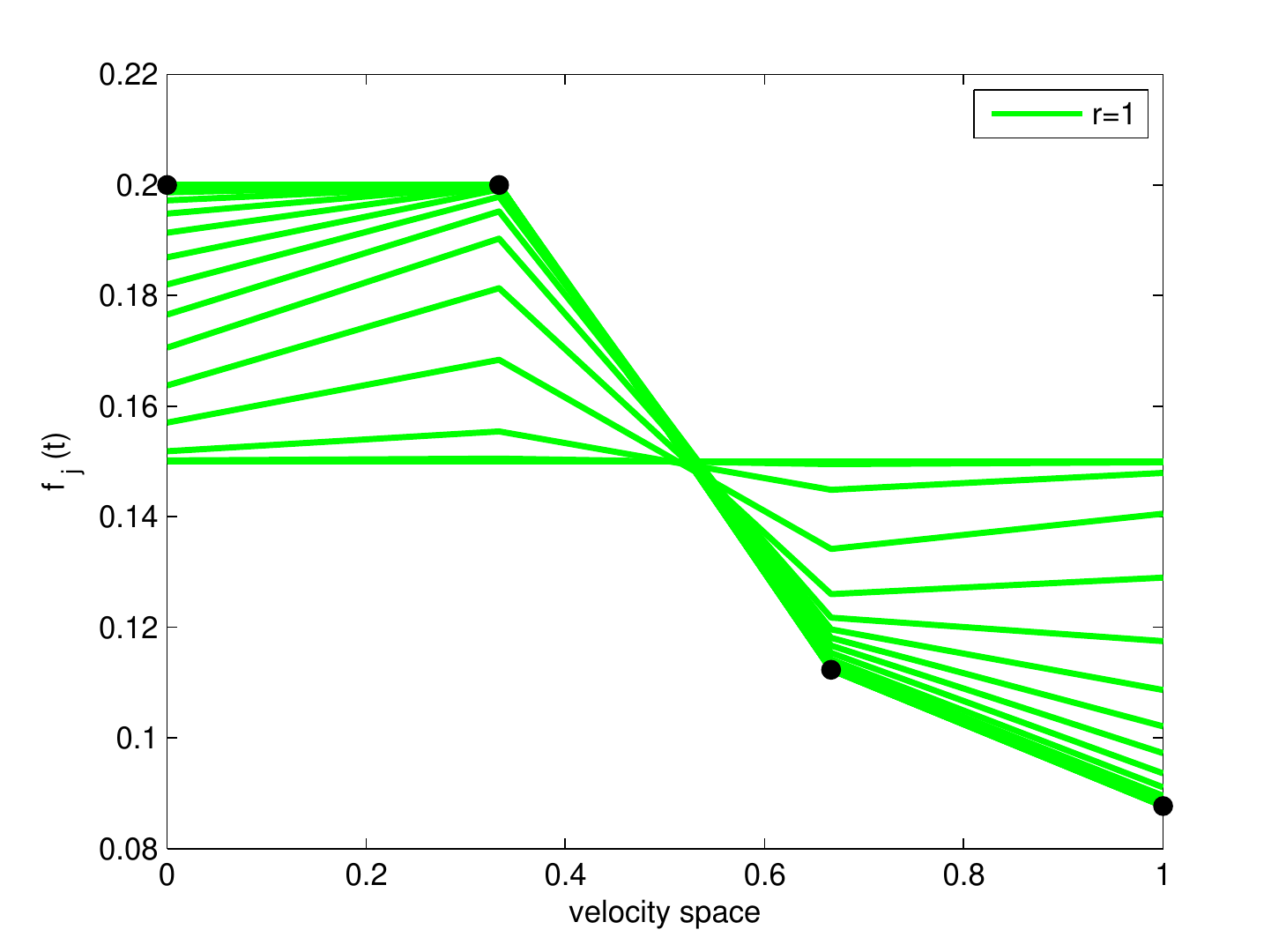}
\hfill
\includegraphics[width=0.32\textwidth]{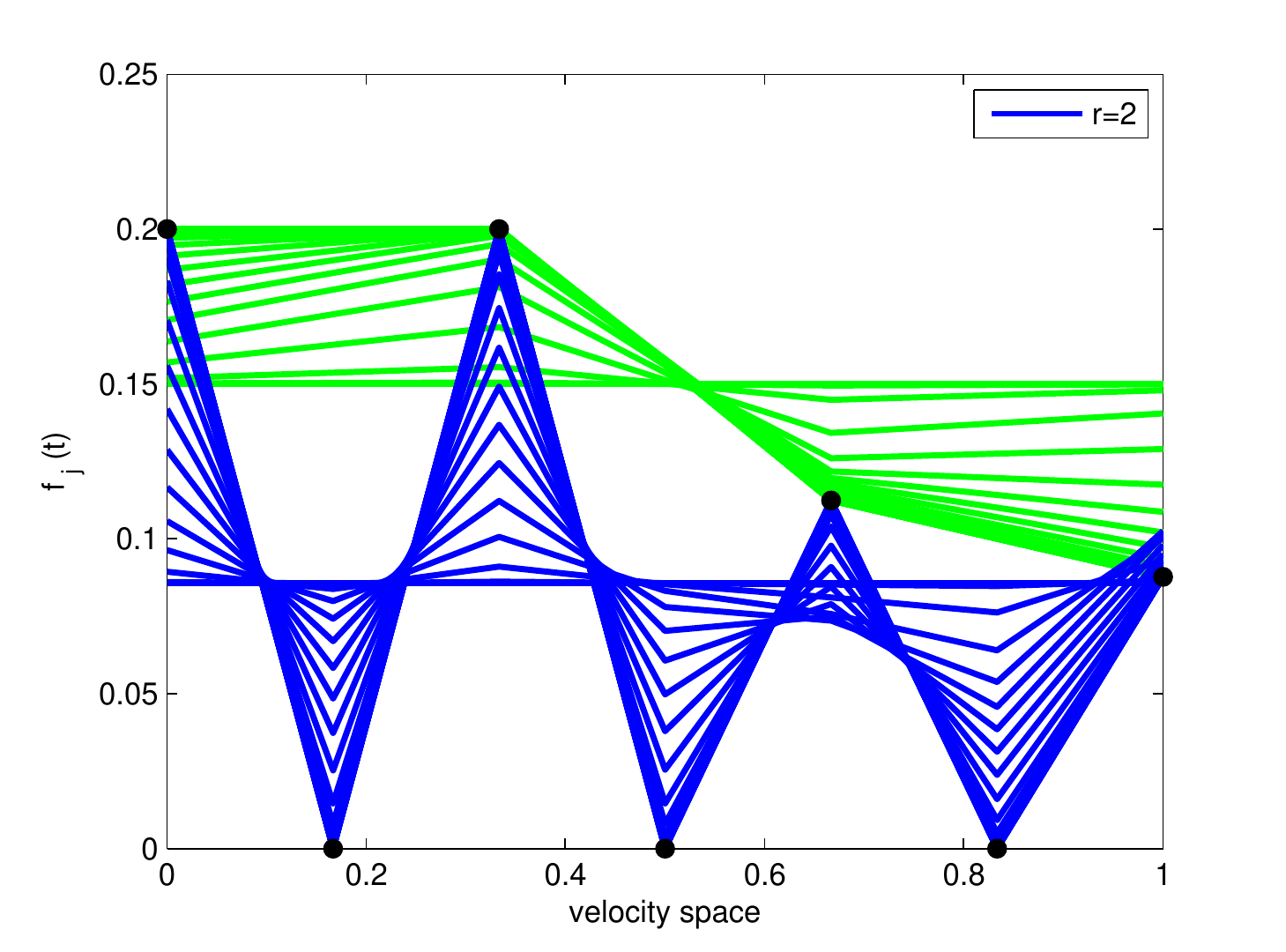}
\hfill
\includegraphics[width=0.32\textwidth]{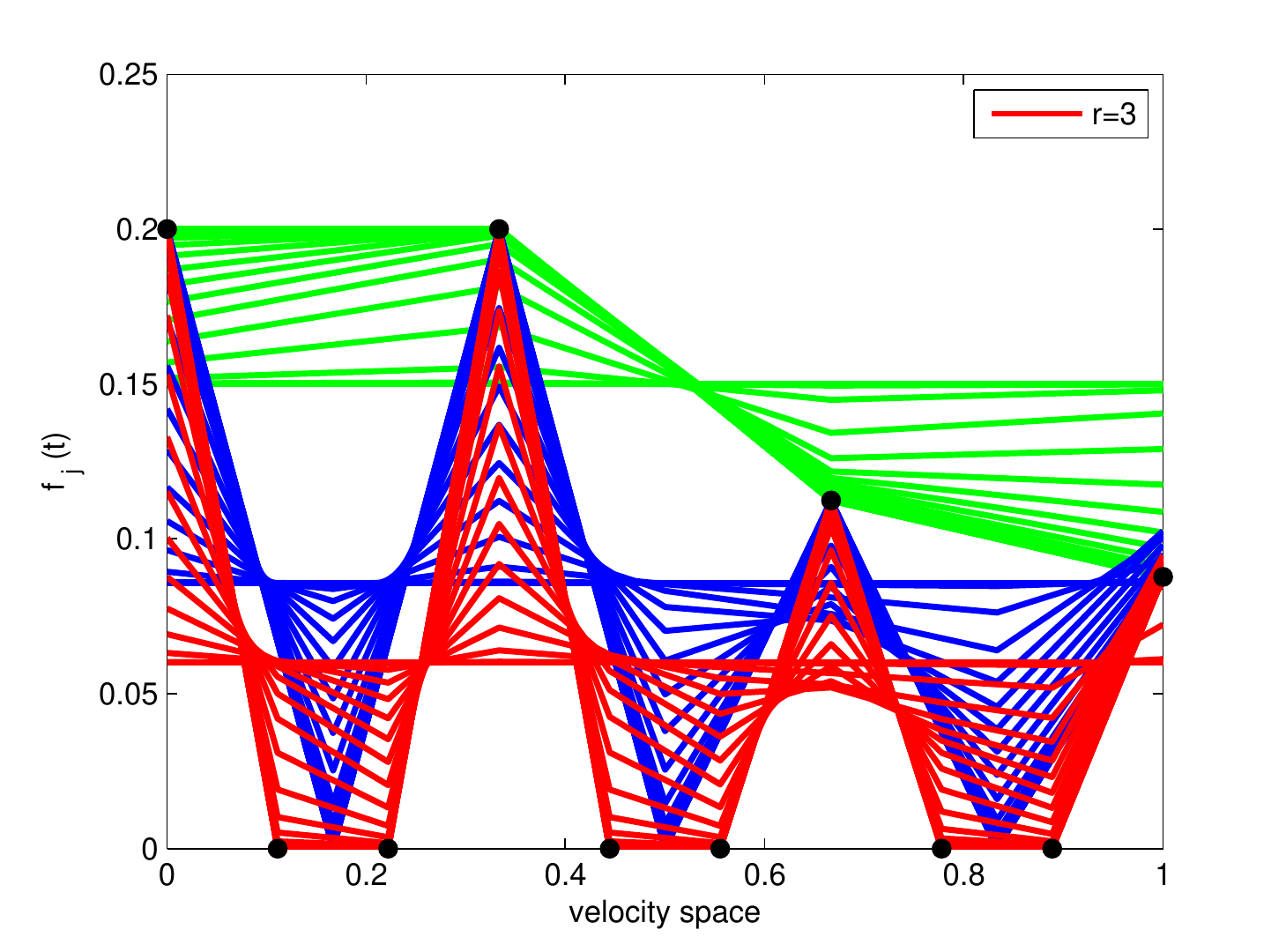}
\caption{Evolution towards equilibrium of the discretized model~\eqref{eq:delta_vectsys} with $N=4$ (green), $N=7$ (blue) and $N=10$ (red) grid points. The acceleration parameter $\Dv$ is taken as $\vm/3$ and the density is $\rho=0.6$. Black circles indicate the equilibrium values.\label{fig:new}}
\end{figure}

The previous considerations can be supported by the numerical results in Figure \ref{fig:new}, in which we show the evolution towards equilibrium of the $f_j$'s, $j=1,\dots,N$. In this figure, $\Dv=\vm/3$, and the different plots are obtained starting from a uniform initial distribution, namely $f_j(t=0)=\rho/N, \forall j=1,\dots,N$, with $r=1$ (green), $r=2$ (blue) and $r=3$ (red), which correspond to $N=4, 7$ and $10$ 
velocity cells respectively, and $\rho=0.6$. It is clear that under grid refinement the number of nonzero steady values does not change. In fact, note that a different dynamics towards equilibrium is observed, for different values of the number $N$ of 
cells, but as equilibrium is approached, the values of the $f_j$'s go to zero except for the velocities corresponding to integer multiples of $\Dv$. Moreover, the non zero values of the steady-state distribution $f^{\infty}$ do not depend on the discretization parameter $\dv$. This fact can be also deduced by looking at equations \eqref{eq:Jdelta:allj} of the discrete collision operator. These expressions are not functions of $\dv$. Thus all exact values of the equilibria can be obtained using a coarse grid.

Theorem \ref{th:delta_eq} below confirms the structure of the equilibria that we have just observed in the numerical results and it states that the steady-state solution of the $\delta$ model, prescribed by Theorem \ref{th:continuous-eq}, can be reconstructed numerically on the grid with $\dv=\Dv$. To this end,
we first recall a result from \cite{DelitalaTosin}, where the existence and well posedness of the solution of such systems is proved and then we show that the discrete equilibria of additional equations resulting from the choice $r>1$ give actually no contribution.

\begin{theorem}(Delitala-Tosin)\label{theo:DelitalaTosin}
Let $f_j(t=0)\geq0$, with $\sum_{j=1}^N f_j(t=0)=\rho$, be the initial condition for the system 
\[
	\frac{d}{dt}f_j=\f^{\tr} A^j \f - \f^{\tr} \mathbf{e}_j \mathbf{1}^{\tr}_N \f,\quad j=1,\dots,N,
\]
where the matrices $A^j$ are stochastic matrices with respect to the index $j$, i.e. $\sum_j A^j_{hk}=1$ for all $h,k$.
Then there exists $t_* = +\infty$ such that the system  admits a unique non-negative local solution $f \geq 0$ satisfying the a priori estimate
\[ ||f(t)||_1 = ||f_0||_1 =\rho \qquad \forall t \in (0, t_*=+\infty]. \]\end{theorem}

The following result, together with Remark \ref{rem:generic:dv}, shows that all equilibria of the discrete model are of the quantized form described by Theorem \ref{th:continuous-eq}. Thus the  next Theorem establishes the correspondences symbolized  by the right vertical and the middle horizontal arrows in Figure \ref{fig:diagcomm}.

\begin{theorem}\label{th:delta_eq}
Let $P$ be a given function of the density $\rho$. For any fixed $\Dv=\vm/T$, $T\in\mathbb{N}$, let $\f_{r}(\rho)$ denote the equilibrium distribution function of the ODE system~\eqref{eq:delta_vectsys}, obtained on the grid with spacing $\dv$ given by $\Dv=r\dv$ with $r=(N-1)/T \in \mathbb{N}$. Then
\begin{equation*}
	(\f_r)_j = \begin{cases}
	(\f_1)_{\pisup{\frac{j}{r}}} & \mbox{mod}(j-1,r)=0 \\ 0 & \mbox{otherwise}
	\end{cases}
\end{equation*}
is the unique stable equilibrium and the values of $\f_1$ depend uniquely on the initial density $\rho$, with $\sum_{k=1}^{T+1}(\f_1)_k=\rho$.
\end{theorem}

\begin{proof}
We already know from Theorem \ref{theo:DelitalaTosin} that the solution of \eqref{eq:delta_vectsys} exists, is non-negative and is uniquely determined by the initial condition. 

To prove the statement, we compute explicitly the equilibrium solutions of \eqref{eq:delta_vectsys}, using the explicit expression of the collision kernel given in \eqref{eq:Jsmallj}, \eqref{eq:Jrp1}, \eqref{eq:Jallj} and \eqref{eq:Jlastj}, with $r\in\mathbb{N}$ and $N=rT+1$. 
Since here we are interested in the solutions of the homogeneous problem,
 we will take identical distributions for the candidate and the field vehicles, i.e.
 $f_j=f^j$.
 
For $j=1$, using the expression \eqref{eq:Jsmallj} and the fact that $\sum_{k=1}^N f_k=\rho$ we obtain
\begin{equation*} \label{eq:eq:feq1}
	\frac{d}{dt}f_1=0 \quad \Leftrightarrow \quad -(1-P)f_1^2  +\left(1-2P\right)\rho f_1=0.
\end{equation*}
This is a quadratic equation for $f_1$, which has the two roots $f_1=0$ and $f_1=\rho(1-2P)/(1-P)$. It is easy to see that one solution is stable, and the other one unstable, depending on the value of $P$. Here we are interested only in the stable root, so we find
\begin{equation}\label{eq:feq1}
(\f_r)_1=\begin{cases}
0 & P \geq \tfrac12 \\ \rho \frac{1-2P}{1-P}&\mbox{otherwise}
\end{cases}
\end{equation}
Thus, no vehicle is in the lowest speed cell $I_1$ if $P\geq \tfrac12$, which, for the simple case $P=1-\rho$ means that all cars are moving if $\rho\leq \tfrac12$. 

The case $j=1$ we just computed is typical. Also for larger values of $j$, we find a quadratic equation for the unknown $f_j$, which involves only previously computed values of $f_k, k<j$. Thus, we can easily compute iteratively all components of $\f_r$.

For $2\leq j\leq r$, the equilibrium equation is obtained by using again the expressions \eqref{eq:Jsmallj} with $r\in\mathbb{N}$. Then
\[
-(1-P)f_j^2 +\left[ (1-2P)\rho-2(1-P)\sum_{k=1}^{j-1} f_k\right] f_j=0.
\]
Start from $j=2$. Clearly, for $P\geq\tfrac12$, substituting equation \eqref{eq:feq1}, we again have $(\f_r)_2=0$. For $P<\tfrac12$, the equation for $f_2$, with $f_1$ given by \eqref{eq:feq1}, becomes
\[
-(1-P)f_2^2 -\left(1-2P\right)\rho f_2=0.
\]
Comparing with the equation for $f_1$, we see that now the stable root is $f_2=0$. Thus, at equilibrium, we have $(\f_r)_2=0$, for all values of $P$. Analogously, it is easy to see that $(\f_r)_j=0$, $\forall j=3,\dots,r$.

For $r+1\leq j < N$, in place of \eqref{eq:Jsmallj}, we use \eqref{eq:Jrp1} and \eqref{eq:Jallj} in order to obtain the equilibrium equation. Then
\begin{equation}\label{eq:eq:feqj}
	-(1-P)f_j^2+\left[(1-2P)\rho-2(1-P)\sum_{k=1}^{j-1} f_k\right] f_j + P \rho f_{j-r} = 0.
\end{equation}
The equation has a positive discriminant
\begin{equation*}
	\mathcal{D}_j=\left[(1-2P)\rho-2(1-P)\sum_{k=1}^{j-1}f_k\right]^2+4P(1-P)\rho f_{j-r}
\end{equation*}
thus it always admits two real roots. To fix ideas, let us consider $r+1 \leq j \leq 2r$. If $j=r+1$, since $(\f_r)_k = 0$, $\forall k=2,\dots,r$, equation \eqref{eq:eq:feqj} becomes
\begin{equation*}
	-(1-P) f_{r+1}^2 + \left[(1-2P)\rho-2(1-P)f_1\right] f_{r+1} + P \rho f_1 = 0.
\end{equation*}
Thus, we find $(\f_r)_{r+1}=0$ for $P\geq 1/2$, because the equation for $f_{r+1}$ becomes identical to \eqref{eq:eq:feq1}. If instead $P<1/2$, substituting the expression for $f_1$, the equation for $f_{r+1}$ becomes
\begin{equation*}
	-(1-P)f_{r+1}^2-(1-2P)\rho f_{r+1} + \rho^2 \frac{P(1-2P)}{1-P} = 0.
\end{equation*}
This equation has a negative and a positive real root, which is stable. Thus
\begin{equation*}
	(\f_r)_{r+1}=\begin{cases}
	0 & P \geq \tfrac12 \\ \frac{-(1-2P)\rho+\rho\sqrt{(1-4P^2)}}{2(1-P)} &\mbox{otherwise}.
	\end{cases}
\end{equation*}
Now, let $r+1 < j \leq 2r$. Since $f_{j-r}= 0$, the constant term of equation \eqref{eq:eq:feqj} is zero. Then, as seen for $j=2,\dots,r$, it is easy to prove that, for each $j=r+2,\dots,2r$ and for all values of $P$, one solution is negative and thus $(\f_r)_j = 0$.

Clearly, this procedure can be repeated and we find that
\begin{equation*}
	(\f_r)_j=\begin{cases}
	0 & P \geq \tfrac12 \\ \frac{-2(1-P)\sum_{k=1}^{j-1} f_k+(1-2P)\rho+\sqrt{\mathcal{D}_j}}{2(1-P)} &\mbox{otherwise}
	\end{cases}
\end{equation*}
if $j=lr+1$, $l=0,\dots,T-1$, while $(\f_r)_j = 0$ otherwise. From these considerations, the thesis easily follows. Finally, just note that for the last value, $f_N$, we can use mass conservation
\[
	(\f_r)_{rT+1}
	=(\f_r)_N
	=\rho-\sum_{l=0}^{T-1} (\f_r)_{lr+1}(\rho).
	\qedhere
\]
\end{proof}

Notice that the result above proves that equilibria are determined by the initial density $\rho$, which is constant in time in the spatially homogeneous case, but they do not depend on the number of cells $N$ used to approximate the kinetic distribution. In fact, the stable equilibria just computed correspond exactly to the values $f_j^\infty$ given in Theorem \ref{th:continuous-eq} and they can all be recovered on the coarse grid $\dv=\Dv$, i.e. choosing $r=1$.

\begin{figure}
	\centering
	\includegraphics[width=0.49\textwidth]{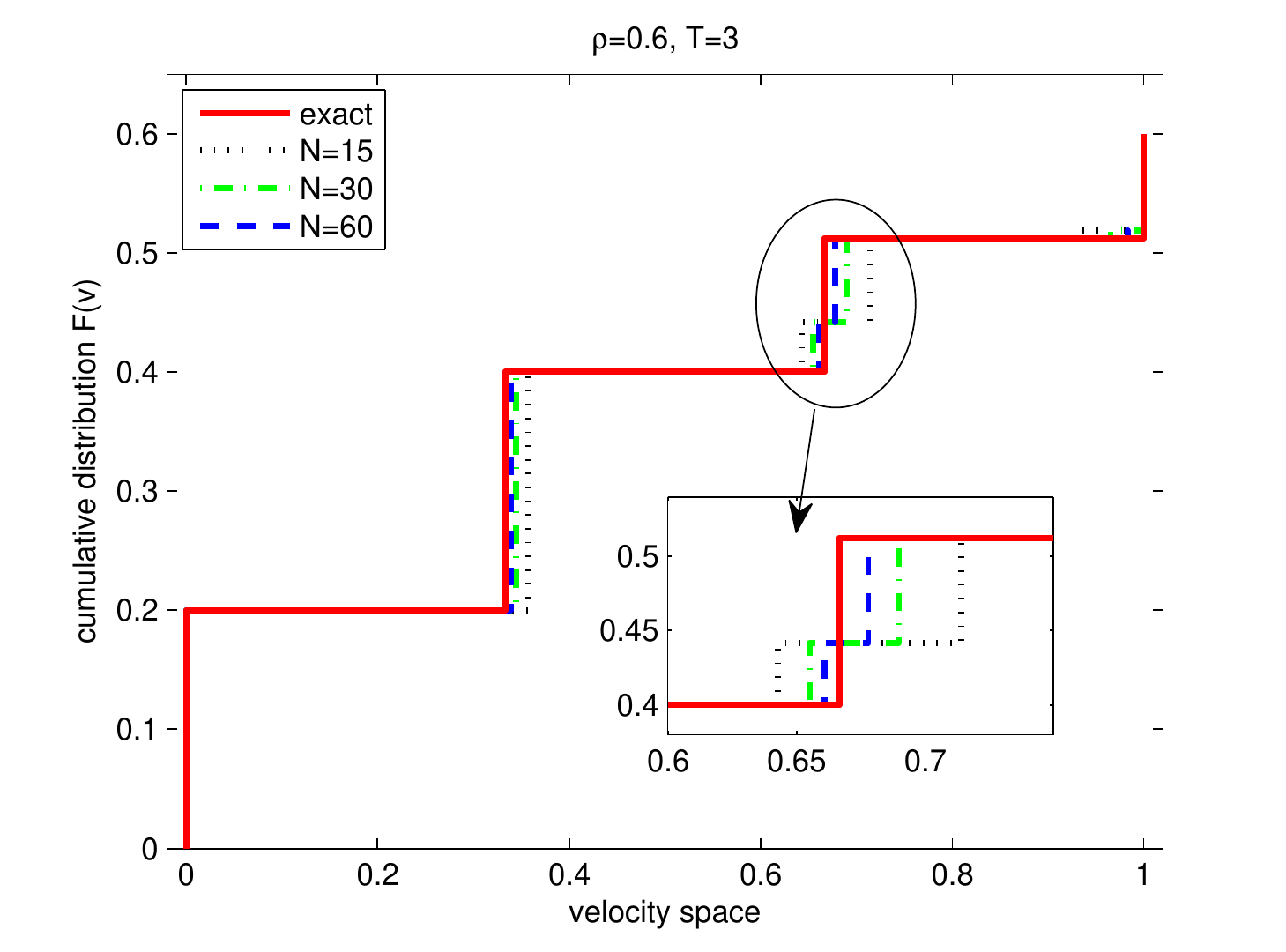}
	\includegraphics[width=0.49\textwidth]{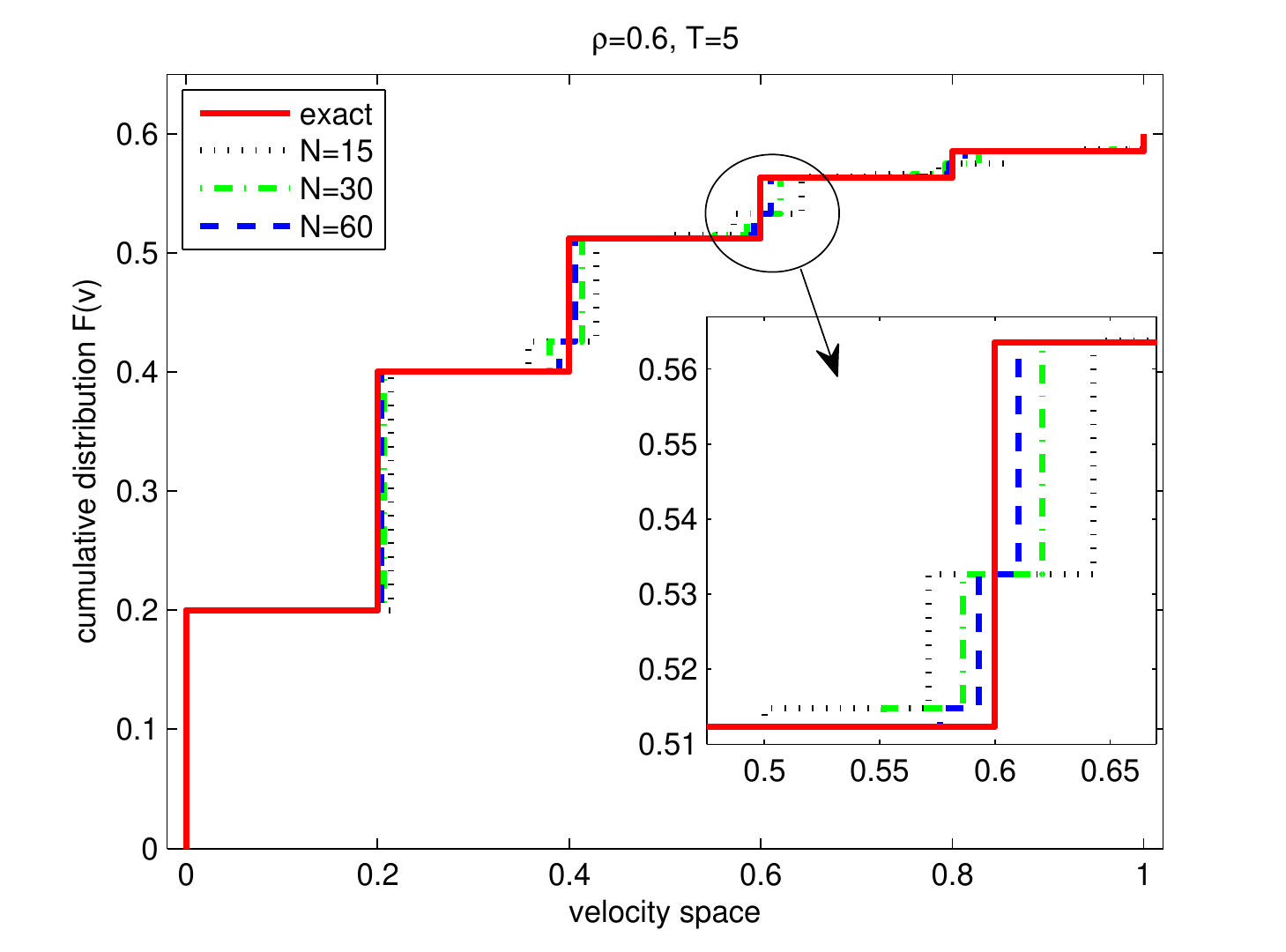}
	\caption{Cumulative density at equilibrium for several values of $\dv\to 0$. The density is $\rho=0.6$ and $\Dv$ is chosen as $1/3$ (left), $1/5$ (right).\label{fig:generic-dv}}
\end{figure}

\begin{remark}[The case of a generic $\dv$] \label{rem:generic:dv}
In order to further investigate the existence of stable equilibria, we can seek more general ones numerically. In fact, the finite volume discretization \eqref{eq:f:discrete} is capable of converging to absolutely continuous equilibria, but so far we proved that it converges to sums of Dirac masses if $\dv$ is an integer submultiple of $\Dv$. Here we apply our discretization scheme with non-integer ratios $r=\Dv/\dv$ and show that the resulting equilibria converge, in the sense of distributions, to the equilibria already described in Theorems \ref{th:continuous-eq} and \ref{th:delta_eq}.

When $r=\Dv/\dv$ is not integer, Theorem \ref{th:delta_eq} cannot be applied, but numerical integration of equation \eqref{eq:delta_vectsys} shows that, for large time, the solution approaches equilibria that have masses concentrated at points spaced by $\Dv$. More precisely, in this more general case, only a small finite number of $f^\infty_j$'s is nonzero and the cumulative distribution function induced by the discrete $f_N^\infty(v)$, cf. \eqref{eq:f:discrete}, 
\[
	F_N(v)=\int_0^v f_N^\infty(v) \mathrm{d}v, \quad v\in[0,\vm],
\]
approximates the cumulative distribution of a sum of Dirac masses centered at multiples of $\Dv$. That is, $F_N(v)$ converges to a  piecewise constant function with jump discontinuities at multiples of $\Dv$. See Figure \ref{fig:generic-dv}.  

Furthermore, Figure \ref{fig:generic-dv} shows that the jump discontinuities of the cumulative distribution in the limit $\dv\to0$ are located exactly in the points computed analytically by Theorem \ref{th:delta_eq} in the case $r\in\mathbb{N}$.
In particular, the figure shows the cumulative distribution at equilibrium computed by solving numerically \eqref{eq:delta_vectsys} for several values of the discretization parameter $\dv$ with non-integer ratio $\Dv/\dv$. In both panels the density is $\rho=0.6$, while $\Dv=1/3$ in the left plot and $\Dv=1/5$ in the right one. The continuous red curve represents the cumulative distribution of the stationary solution for $\dv=\Dv$, computed by Theorem \ref{th:delta_eq}. The black dotted curve is computed with $N=15$ velocity cells, the green dot-dashed line with $N=30$ and finally the blue dashed one with $N=60$. We observe that as $\dv\to 0$ the cumulative distribution tends to that of the analytic solution of Theorem \ref{th:continuous-eq}, where only the velocities centered in multiples of $\Dv$ give rise to a jump discontinuity. 
\end{remark}

We conclude the section with a few remarks on the structure of the equilibria of \eqref{eq:delta_vectsys}.

\begin{remark}[Reduced velocity space]\label{rem:reduced_space}
The $\delta$ model is characterized by a small number of non-trivial values for the microscopic velocities, which allows one to compute analytically the equilibrium of the system. This fact provides an analytic closure of the macroscopic equations. This fact can also be exploited from a numerical point of view, justifying the use of a small number of discretization points in simulations aimed at capturing equilibrium effects. In this sense, the model seems to suggest that  kinetic corrections can be accounted for with a computational cost which is not much higher than the one needed for a macroscopic model.
\end{remark}

\begin{figure}[t!]
\centering
\includegraphics[width=0.32\textwidth]{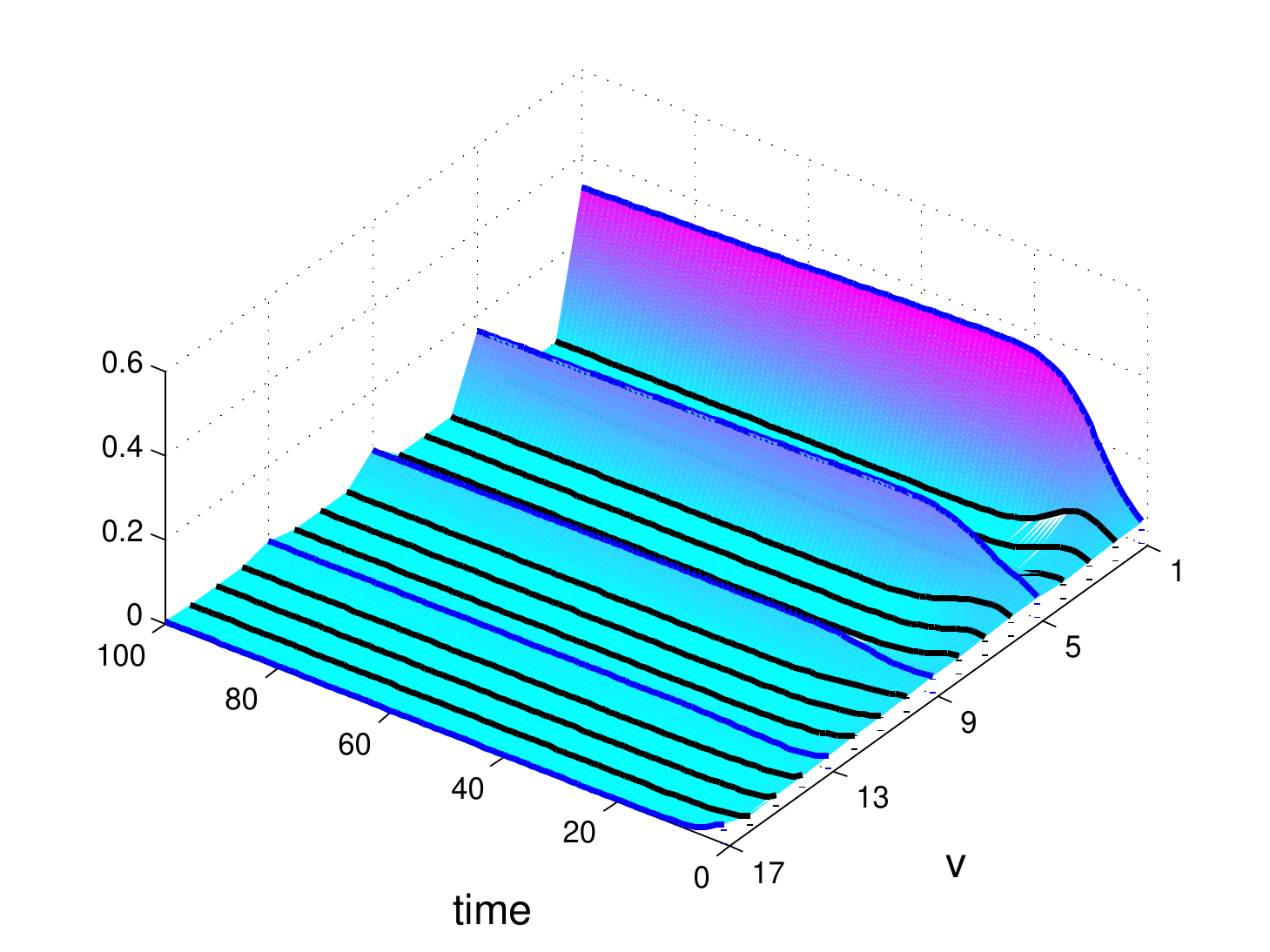}
\hfill
\includegraphics[width=0.32\textwidth]{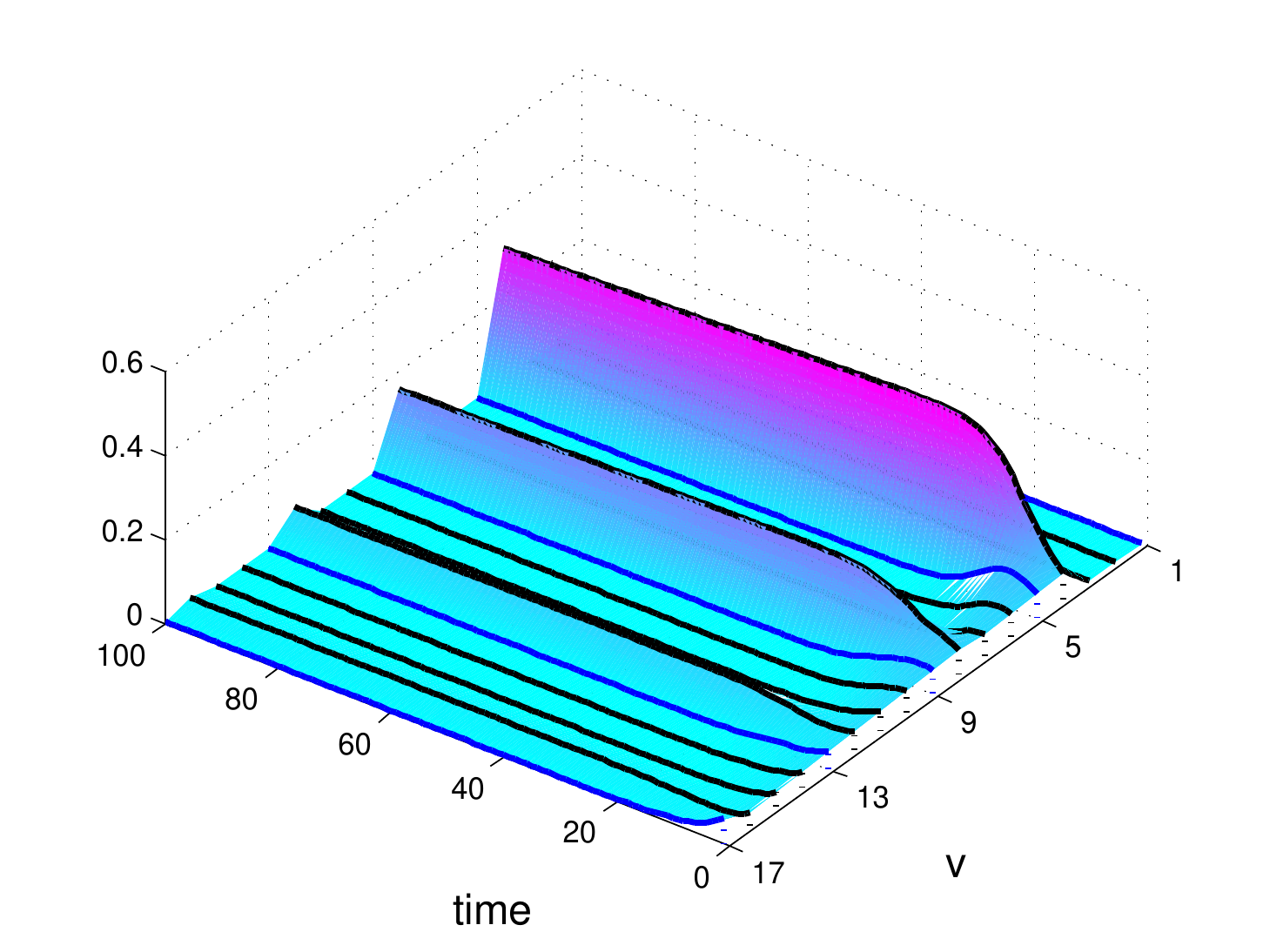}
\hfill
\includegraphics[width=0.32\textwidth]{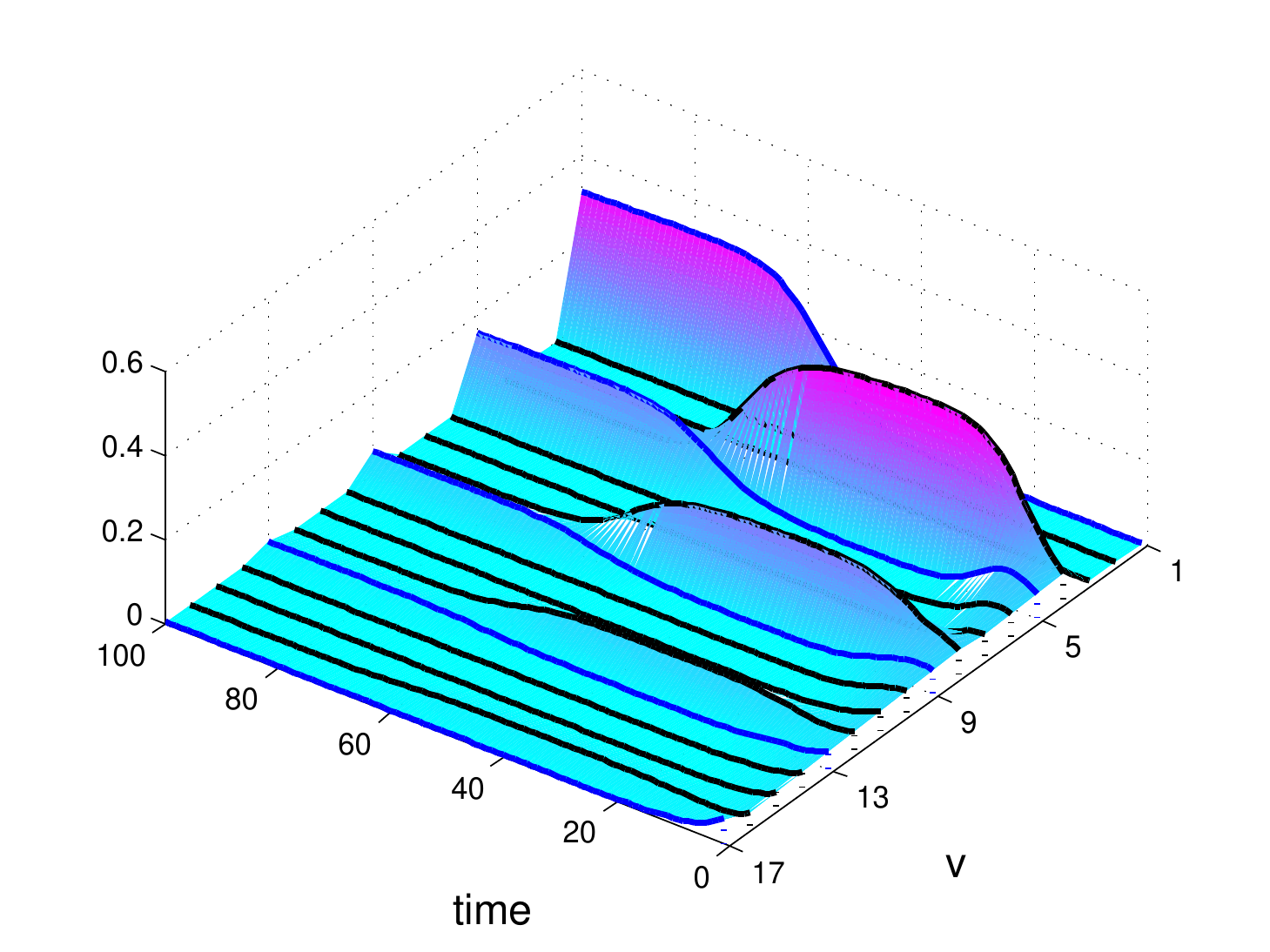}
\caption{Evolution towards equilibrium, $\rho=0.7$, $T=4$, $N=17$. Left: $f_j(t=0)\equiv \rho/N$. Middle: $f_j(t=0)=0, j=1,2,3$, $f_j(t=0)\equiv(\rho/(N-3)), j>3$. Right: $f_1=\epsilon=10^{-6}, f_2=f_3=0$ and $f_j(t=0)\equiv((\rho-\epsilon)/(N-3))$. The thick lines highlight the components $f_j$ and the blue ones are for those that appear in stable equilibria, i.e. with $j=kr+1$ for $k=0,\ldots,T$.
\label{fig:StableEq}}
\end{figure}

\begin{remark}[Unstable equilibria]\label{rem:unstable_eq}
Theorem~\ref{th:delta_eq} gives the uniqueness of the stable equilibrium of the model with $\Dv/\dv=r\in\mathbb{N}$. Unstable ones may occur if the initial condition is such that $f_1(0)=0$. In fact, the interaction rules related to the case $\vb>\va$ do not generate a post-interaction velocity $v$ which is less than $\va$. Thus if $f_1(0)=0$, i.e. if there are no vehicles with velocity $v_1=0$ at the initial time, there will not be interactions 
leading to an increase of $f_1$.
This consideration can be generalized: if $f_j(0)=0$ for $j=1,\dots,\bar{j}<r$, then the computed equilibria will be $f_j=(\mathbf{f}_r)_{j-\overline{j}}$, where $\mathbf{f}_r$ is the vector containing the stable equilibria.
In this sense, the equilibrium solution of the $\delta$ model does not only depend on $\rho$ but also on the initial condition $f(0,v)$. These solutions however are unstable: a small perturbation on $f_1(t=0)$ is enough to trigger the evolution towards the stable equilibrium, which depends only on $\rho$.

This is illustrated in Figure \ref{fig:StableEq}: in the left panel we show the evolution towards equilibrium when $f_j(t=0)\neq0$ for all classes (this is the stable equilibrium), while in the middle we show the case when $f_1=f_2=f_3=0$. In the rightmost panel we show a perturbation of the previous case, where $f_1$ takes a very small but nonzero value. It is clear that the evolution goes at first towards the unstable equilibrium of the middle panel, but then, in the long run, the stable equilibrium of Theorem \ref{th:delta_eq} emerges.
\end{remark}

\begin{figure}[t!]
	\centering
	\includegraphics[width=0.49\textwidth]{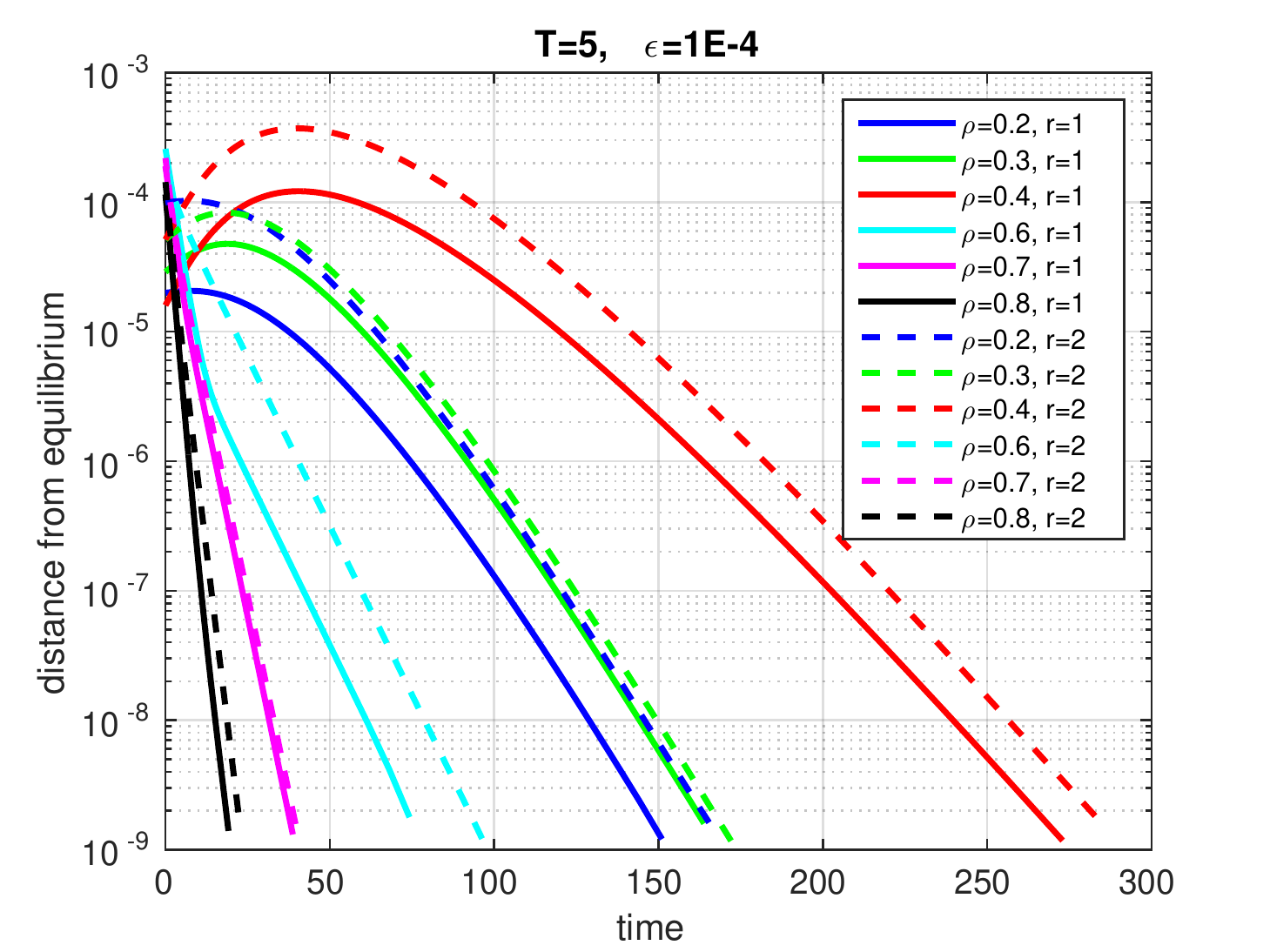}
	\hfill
	\includegraphics[width=0.49\textwidth]{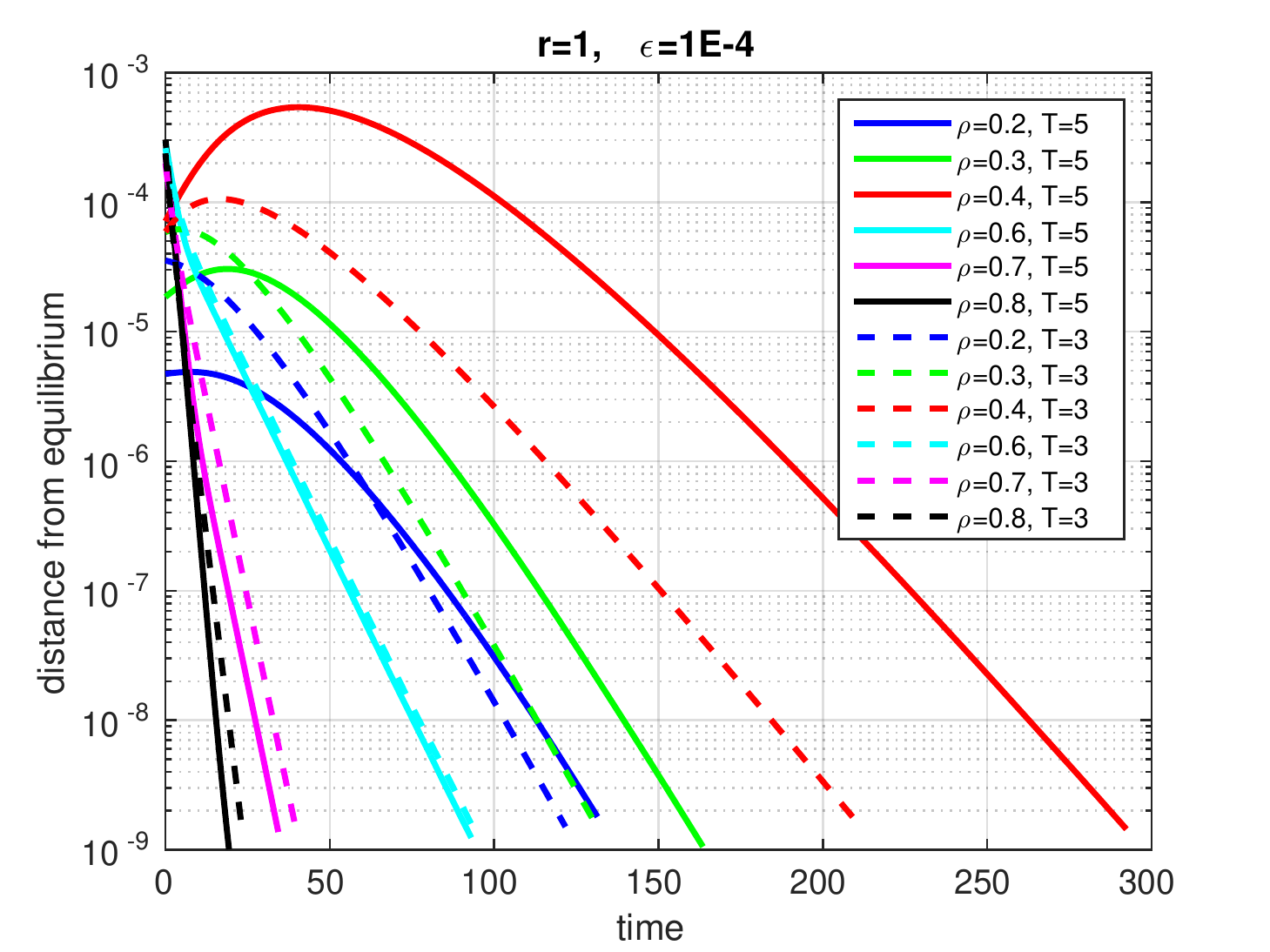}
	\caption{Speed of convergence towards the stable equilibria of the $\delta$ model. The initial condition is a small random perturbation of the steady-states.\label{fig:EnergyPlot}}
\end{figure}

\begin{remark}[Convergence rate to equilibrium]
	In Figure \ref{fig:EnergyPlot} we study the rate of convergence towards the stable steady-states. For the set of densities $\rho\in\{0.2,0.3,0.4,0.6,0.7,0.8\}$ we integrate numerically the system \eqref{eq:delta_vectsys} for large times starting from a small random perturbation of the stable equilibrium. In both panels, we use a linear scale for the $x$-axis (time) and a logarithmic scale for the $y$-axis (error). The error is computed at each step as $e(t)=\norm{f(t)-f^\infty}_2$. The figure suggests that the rate of convergence towards the stable equilibrium depends on the density. In fact, in the left panel we consider different values of the ratio $\Dv/\dv=r$, in particular $r=1$ (solid lines) and $r=2$ (dashed lines), and we note that the slopes of the curves corresponding to the same value of $\rho$ are the same. A similar behavior is observed in the right plot in which two different $\Dv$ are taken, $\Dv=1/5$ (solid lines) and $\Dv=1/3$ (dashed lines). We can conjecture that for large enough times the distance from equilibrium behaves as $e(t) \simeq C e^{-M(\rho)t} \norm{f(t=0)-f^\infty}_2$, where $C=C(r,\Dv)$, $M(\rho)>0$.
\end{remark}

\section{The $\chi$ velocity model}
\label{sect:chi}

The structure of the steady-state distribution of the $\delta$ model clearly depends on the particular choice of the acceleration interaction made in \eqref{eq:Adelta}, in which a vehicle accelerates by jumping from its pre-interaction velocity $\vb$ to the new velocity $\vb+\Dv$. Thus it could seem quite natural that only velocities $0,\,\Dv,\,2\Dv,\dots,\vm$ give a non zero contribution at equilibrium. 

Here we study the $\chi$ model, already introduced in Section~\ref{sec:modeling} (see equation \eqref{eq:Achi}), in which vehicles can assume a post-interaction velocity uniformly distributed over a range of speeds when the acceleration interaction occurs. We will show that, although this model is more refined than the $\delta$ model, at equilibrium the essential information is already caught by the simpler $\delta$ model.

Using the formulation~\eqref{eq:Achi} for the transition probability density $A$, we rewrite the gain term in~\eqref{eq:model2} as
\begin{equation*}
	G[f,f](t,v)=\eta \int_V \int_V \left[(1-P)\delta_{\min\{\vb,\va\}}(v) + P\frac{\chi_{\left[\vb,\min\left\{\vb+\Dv,\vm\right\}\right]}(v)}{\min\left\{\vb+\Dv,\vm\right\}-\vb}\right]\fb\fa \dvub \dvua.
\end{equation*}
Notice that the $\chi$ function can be split as
\[
	\frac{\chi_{\left[\vb,\min\left\{\vb+\Dv,\vm\right\}\right]}(v)}{\min\left\{\vb+\Dv,\vm\right\}-\vb}=
	\begin{cases}
		\frac{\chi_{\left[\vb,\vb+\Dv\right]}(v)}{\Dv},&\text{if\; $\vb\in[0, \vm-\Dv]$}\\
		\frac{\chi_{\left[\vb,\vm\right]}(v)}{\vm-\vb},&\text{if\; $\vb\in(\vm-\Dv, \vm]$}
	\end{cases}
\]
hence substituting in the above equation and evaluating explicitly the integrals, we find
\begin{equation}\label{eq:gainschi}
\begin{aligned}
	G[f,f](t,v)=&\eta(1-P) f(t,v) \left[\int_v^{\vm}\fa \dvua + \int_v^{\vm}\fb \dvub\right]\\
	& + \eta P \rho \left[\frac{1}{\Dv}\int_0^{\vm-\Dv} \chi_{[\vb,\vb+\Dv]}(v)\fb\dvub+\int_{\vm-\Dv}^{\vm} \frac{\chi_{[\vb,\vm]}(v)}{\vm-\vb}\fb\dvub\right].
\end{aligned}
\end{equation}
Observe that \eqref{eq:gainschi} differs from \eqref{eq:delta_model} only in the terms proportional to $P$.

\subsection{Discretization of the model}
\label{sec:discretechi}

To compute the steady-state solution of the $\chi$ model, we need to integrate the equations numerically. We use the same discretization of the velocity space $V$ introduced to discretize the $\delta$ model and therefore the kinetic distribution is approximated as in 
\eqref{eq:f:discrete}.

Integrating the kinetic equation~\eqref{eq:model1} over each cell, we find the system of ODEs~\eqref{eq:discretesys},
but now the gain term is given by \eqref{eq:gainschi}. Although in this case the integrals are laborious, they can be computed recalling that
$\Dv=\vm/T$ with $T\in\mathbb{N}$ and assuming that $\dv$ is an integer submultiple of $\Dv$. Thus we will take $N-1\equiv0 \;(\mathtt{mod}\, T)$ and $r=\frac{N-1}{T}$.

The details are reported in Appendix \ref{app:eq:matricichi}. Here we just point out that, as in the case of the $\delta$ model, the ODE system can be conveniently rewritten in vector form as
\begin{equation}
	\frac{d}{dt}f_j=\eta \left[\f^{\tr} A^j_{\chi} \f - \f^{\tr} \mathbf{e}_j \mathbf{1}^{\tr}_N \f\right],\quad j=1,\dots,N
    \label{eq:chi_vectsys}
\end{equation}
where $\f=[f_1,\dots,f_N]^{\tr}\in\mathbb{R}^N$ is the vector of the unknown functions, $\mathbf{e}_j\in\mathbb{R}^N$ denotes the vector with a $1$ in the $j$-th coordinate and $0$'s elsewhere, $\mathbf{1}^{\tr}_N=[1,\dots,1]\in\mathbb{R}^N$ and $A^j_{\chi}$ is the $j$-th interaction matrix such that $\left(A_{\chi}^j\right)_{hk}$ contains the probabilities that a candidate vehicle with velocity 
in $I_h$ 
interacting with a field vehicle with velocity 
in $I_k$ 
acquires 
a velocity in $I_j$.

\begin{figure}
	\centering
	\begin{tikzpicture}[y={(0,-1cm)},scale=1]
	\begin{scope}
	
	\filldraw[\colR] (0,0) rectangle (3.1,0.5);
	\skeleton
	\node[anchor=center] at (1.5,3.6) {\large$A^j_{\chi}$ for $1\leq j\leq r$};
	\filldraw[pattern=\patUnoMenoP] (0.4,0.4) rectangle (3.1,0.5);
	\filldraw[pattern=\patUnoMenoP] (0.4,0.5) rectangle (0.5,3.1);
	\draw[help lines,->,yshift=-0.45cm]
	(-0.3,0) node[left]{$j$} -- (-0.1,0);
	\draw[help lines,->,xshift=0.45cm]
	(0,-0.3) node[above]{$j$} -- (0,-0.1);
	\end{scope}
	\end{tikzpicture}
	\begin{tikzpicture}[y={(0,-1cm)},scale=1]
	\begin{scope}[xshift=4cm]
	
	\filldraw[\colR] (0,0.6) rectangle (3.1,1.7);
	\skeleton
	\node[anchor=center] at (1.5,3.6) {\large$A^j_{\chi}$ for $r< j< N$};
	\filldraw[pattern=\patUnoMenoP] (1.6,1.6) rectangle (3.1,1.7);
	\filldraw[pattern=\patUnoMenoP] (1.6,1.7) rectangle (1.7,3.1);
	\draw[help lines,->,yshift=-1.65cm]
	(-0.3,0) node[left]{$j$} -- (-0.1,0);
	\draw[help lines,->,yshift=-0.65cm]
	(-0.3,0) node[left]{$j-r$} -- (-0.1,0);
	\draw[help lines,->,xshift=1.65cm]
	(0,-0.3) node[above]{$j$} -- (0,-0.1);
	\end{scope}
	\end{tikzpicture}
	\begin{tikzpicture}[y={(0,-1cm)},scale=1]
	\begin{scope}[xshift=8cm]
	
	\filldraw[\colR] (0,2.0) rectangle (3.1,3);
	\skeleton
	\node[anchor=center] at (1.5,3.6) {\large$A^j_{\chi}$ for $j=N$};
	\filldraw[pattern=\patP] (0,3) rectangle (3,3.1);
	\fill[pattern=\patUno] (3,3) rectangle (3.1,3.1);
	\draw[help lines,->,yshift=-2.05cm]
	(-0.3,0) node[left]{$N-r$} -- (-0.1,0);
	\draw[help lines,->,yshift=-3.05cm]
	(-0.3,0) node[left]{$N$} -- (-0.1,0);
	\draw[help lines,->,xshift=3.05cm]
	(0,-0.3) node[above]{$N$} -- (0,-0.1);
	\end{scope}
	\end{tikzpicture}
	\caption{Structure of the probability matrices of the $\chi$ model with $\dv$ integer sub-multiple of $\Dv$.\label{fig:chimatrices}}
\end{figure}

Since the matrices appearing in \eqref{eq:chi_vectsys} are again stochastic, we can apply Theorem \ref{theo:DelitalaTosin} to guarantee the well posedness of the associated Cauchy problem.

In Figure~\ref{fig:chimatrices} we show the structure of the $\chi$ matrices: they are less sparse than the $\delta$ matrices because of the uniformly distributed acceleration in $\left[\vb,\vb+\Dv\right]$, where $\vb$ is the pre-interaction speed. In fact the matrix $A^j_{\chi}$ contains non-zero elements also in the rows from the $(j-r+1)$-th to the $(j-1)$-th, see the shaded areas in Figure~\ref{fig:chimatrices}, which represent non-zero probabilities of accelerating to 
a speed in $I_j$. Since $\Dv=r\dv$, exactly $r$ rows fill up. Instead, the area drawn using hatching contains the same probabilities already shown in Figure \ref{fig:matricidelta} for the case of the $\delta$ model, with $r\in\mathbb{N}$. Note that the elements of the matrices depend on $\dv$ (see equations \eqref{eq:elementichi}).
Thus, in contrast to the $\delta$ model, steady solutions of the ODE system~\eqref{eq:chi_vectsys} depend on the number of velocity cells $N$ chosen to approximate the kinetic distribution (see \eqref{eq:f:discrete}). In other words, although $\dv$ is an integer sub-multiple of $\Dv$, this model does not converge, as time goes to infinity, to the asymptotic distribution on the coarse grid as in the $\delta$ model.  However, we do recover the asymptotic distribution on the coarse grid in the limit $\dv\to 0$.

Notice from equations \eqref{eq:elementichi} that all the elements in the rows $j-r,\dots,j-1$ of the matrices $A^j_{\chi}$, $j=1,\dots,N$, tend to $0$ as $1/r$ when the grid is refined.
In particular, for $j=1,\dots,r$, $A^j_{\chi}\rightarrow A^j_{\delta}$.
This consideration is not true for the matrices $A_{\chi}^j$, for $j=r+1,\dots,N$.

\subsection{Expected speed of the $\delta$ and the $\chi$ model}

Despite their differences the $\chi$ and the $\delta$ model are deeply related.
This can be seen by computing the expected output speed in each model resulting from a fixed pre-interaction speed. 
We define the expected value $\langle v\rangle$ of the post-interaction velocity as
\begin{equation}\label{eq:expected}
	\langle v\rangle=\int_{0}^{\vm} v\Avvvr \, \dvu, \quad \left(\vb,\va\right) \in V\times V.
\end{equation}
For brevity we indicate with $A_{\delta}(v)$ and $A_{\chi}(v)$ the probability densities given in~\eqref{eq:Adelta} and~\eqref{eq:Achi} respectively. Again we assume that $P_1=P_2=P$ as function of the density $\rho$. For the $\delta$ model we obtain
\begin{align}\label{eq:expected_delta}
	\langle v\rangle_{\delta}&=\int_0^{\vm} v \left[ \left(1-P\right) \delta_{\min\{\vb,\va\}}(v) + P\, \delta_{\min\left\{\vb+\Dv_{\delta},\vm\right\}} (v) \right]\,\dvu
\\ \nonumber
	&=\left(1-P\right)\, \min\{\vb,\va\} 
	+ P\, 
	\begin{cases} \vb+\Dv_\delta, & \text{if}\;\; \vb+\Dv_\delta\leq\vm 
	\\ \vb+\left(\vm-\vb\right), & \text{if}\;\; \vb+\Dv_\delta>\vm    
	\end{cases}
\end{align}
In contrast, if we consider the $\chi$ model we have
\begin{align*}
\langle v\rangle_{\chi}&=\int_0^{\vm} 
v \left[ \left(1-P\right) \delta_{\min\{\vb,\va\}}(v) + P\, \frac{ \chi_{ \left[ \vb,\min\left\{\vb+\Dv_{\chi},\vm\right\} \right] } (v) }{ \min\left\{\vb+\Dv_{\chi},\vm\right\}-\vb} \right]\, \dvu
\\ 
	&= \left(1-P\right)\,\min\{\vb,\va\} + P\, \begin{cases} \ds{\frac{1}{\Dv_{\chi}} \int_{\vb}^{\vb+\Dv_{\chi}} v\, \dvu}, & \text{if}\;\; \vb+\Dv_{\chi}\leq\vm \vspace{0.3cm}\\ \ds{\frac{1}{\vm-\vb}\int_{\vb}^{\vm} v\,\dvu}, & \text{if}\;\; \vb+\Dv_{\chi}>\vm \end{cases}
\end{align*}
and thus
\begin{equation}\label{eq:expected_chi}
  \langle v\rangle_{\chi}
  =\left(1-P\right)\,\min\{\vb,\va\} 
  + P\, 
  \begin{cases} 
  \ds{\vb+\frac{\Dv_\chi}{2}}, & \text{if}\;\; \vb+\Dv_\chi\leq\vm   \vspace{0.3cm} 
  \\ 
  \ds{\vb+\frac12\left(\vm-\vb\right)}, & \text{if}\;\; \vb+\Dv_\chi>\vm 
  \end{cases}
\end{equation}
By comparing the last lines of~\eqref{eq:expected_delta} and~\eqref{eq:expected_chi}, it is clear that 
\begin{equation}\label{eq:Deltav}
\langle v\rangle_{\chi}=\langle v\rangle_{\delta} \qquad \forall \, \vb\leq\vm-\Dv_\chi, \qquad \mbox{provided} \; \Dv_\delta=\tfrac12\Dv_\chi.
\end{equation}

\begin{remark}[Uniformly distributed deceleration]
	The computation of the expected speed shows a link between the $\delta$ model and the $\chi$ model. In fact, under the constraint $\Dv_{\delta}=\frac12\Dv_\chi$, the two models provide the same  expected speed. Similarly, if we consider a braking scenario in which the candidate vehicle brakes to a speed uniformly distributed in an interval centered on $\va$, then the expected speed results to be again $(1-P)\va$.
\end{remark}

\begin{remark}
Let us  compare the $A^j_\chi$ matrices (Figure \ref{fig:chimatrices}) for $r<j\leq N-r$ with a given $\Dv$ and the corresponding $A^j_\delta$ matrices (Figure \ref{fig:matricidelta}) with $\tfrac{\Dv}{2}$. For the case $r\in\mathbb{N}$, the isolated nonzero row of $A^j_\delta$ is at $j-\tfrac{r}{2}$, which corresponds to the middle of the green shaded area in $A^j_\chi$.
Moreover, for any fixed $\Dv$, it can be proved that, for $\dv\to 0$, the sum of the quantities located  
in the shaded area of $A_\chi^j$
is equal to the total contribution provided by the $\left(j-\frac{r}{2}\right)$-th row of the $\delta$ matrices obtained with the acceleration parameter $\frac{\Dv}{2}$. In other words, as $\dv\to 0$, the total effect of the $\left(j-\frac{r}{2}\right)$-th row of $A_{\delta}^j$ with $\frac{\Dv}{2}$ is spread over $r+1$ rows in the matrices $A_{\chi}^j$ with $\Dv$, which are the rows shaded in green in Figure~\ref{fig:chimatrices}.
\end{remark}

\section{Macroscopic properties}
\label{sect:fundamental}

\paragraph{Macroscopic acceleration}
In order to explain the relation of $\Dv$ with the acceleration of the vehicles in the model, we compute the rate of change of the macroscopic velocity:
\begin{align*}
\frac{\partial u}{\partial t}
&=
\frac{\partial}{\partial t}
\left[
\frac1\rho
\int_V vf(t,v)\,\dvu
\right]
=
\frac1\rho
\int_V
v Q[f,f](t,v)\, \dvu
\\
&=
\frac\eta\rho
\left[
\int_V v \,\dvu
\int_V\int_V
\Avvvr \fb\fa \,\dvub \dvua
-\rho \int_V vf(t,v) \,\dvu
\right]
\\
&=
\frac\eta\rho
\left[
\int_V\int_V
\langle v\rangle \fb\fa \,\dvub \dvua
-\rho \int_V vf(t,v) \,\dvu
\right]
\end{align*}
where $\langle v\rangle$ is a function of $\vb$ and $\va$ as defined in \eqref{eq:expected}.

In the case of the $\delta$ model, $\langle v\rangle_\delta$ is given by \eqref{eq:expected_delta} 
and thus
\begin{align*}
\frac{\partial u}{\partial t}
=
\frac\eta\rho
\bigg[
&(1-P)\int_0^{\vm} \int_0^{\vb} \va \fb\fa \, \dvua \dvub
+ (1-P)\int_0^{\vm} \int_{\vb}^{\vm} \vb\fb\fa \,\dvua \dvub
\\
&+P\rho\int_0^{\vm-\Dv} (\vb+\Dv_{\delta})\fb \,\dvub
+P\rho\int_{\vm-\Dv}^{\vm} \vm\fb \, \dvub
-
\rho \int_0^{\vm}vf(t,v)\,\dvu
\bigg].
\end{align*}
Given an initial distribution $f(t=0,v)$, the equation above yields the evolution of the macroscopic acceleration in time. It is easy to study analytically this quantity at the initial time. In particular, we compute the initial acceleration in the case in which all vehicles are still but the density is below the value for which $P=1/2$ (that is $\rho=1/2$ when taking $P=1-\rho$).
By considering an initial distribution with all vehicles in the lowest velocity class, i.e. of the form 
$f(0,v)=\tfrac{2\rho}{\dv}\chi_{I_1}(v)$, we have 
\begin{equation*}
\left. \frac{\partial u}{\partial t} \right|_{t=0}
= 
\eta P\Dv_\delta\int_0^{\vm-\Dv_{\delta}} \fb \, \dvub +\mathcal{O}(\dv)
= \eta \rho P \Dv_\delta +\mathcal{O}(\dv).
\end{equation*}

The above equation shows that the acceleration of the vehicles in the $\delta$ model depends linearly on $\Dv$.
Analogously, for the $\chi$ model, using \eqref{eq:expected_chi}, one obtains
\begin{equation*}
\left. \frac{\partial u}{\partial t} \right|_{t=0}
= \frac12 \eta \rho P \Dv_\chi + \mathcal{O}(\dv)
\end{equation*}
which reinforces the remark made in \eqref{eq:Deltav} about the similarities of the $\chi$ and the $\delta$ model when $\Dv_\chi=2\Dv_\delta$. 

\begin{remark}[Acceleration]
Recall that $(\eta\rho)^{-1}$ is a time. Thus $\eta\rho \Dv$ is the built-in acceleration of the model, which, not surprisingly, is linked to $\Dv$. We can use dimensional arguments to estimate the order of magnitude of $\Dv$. According to Lebacque \cite{Lebacque03}, the maximum acceleration of cars is approximately $a_{LB}=2.5$ m/s$^2$. The maximum speed is approximately $\vm\simeq 28$m/s, and we expect the maximum acceleration when $P=1$. Thus $\eta\rho \Dv \simeq a_{LB}$.
\end{remark}

The estimates above provide the trend of the macroscopic acceleration starting from rest. For the general case, we now study the evolution of the macroscopic velocity $u$ in time, up to steady state. These data are shown in Fig. \ref{fig:accel} and Fig. \ref{fig:decel}, for various combinations of the model parameters. The results shown are obtained integrating the equations for the $\delta$ and the $\chi$ model found in \eqref{eq:delta_vectsys} and  \eqref{eq:chi_vectsys}, respectively, with $r\in\mathbb{N}$, and computing at each time $u(t)=\tfrac{1}{\rho}\int_V vf(t,v)\; \dvu$.

\begin{figure}
\includegraphics[width=0.49\textwidth]{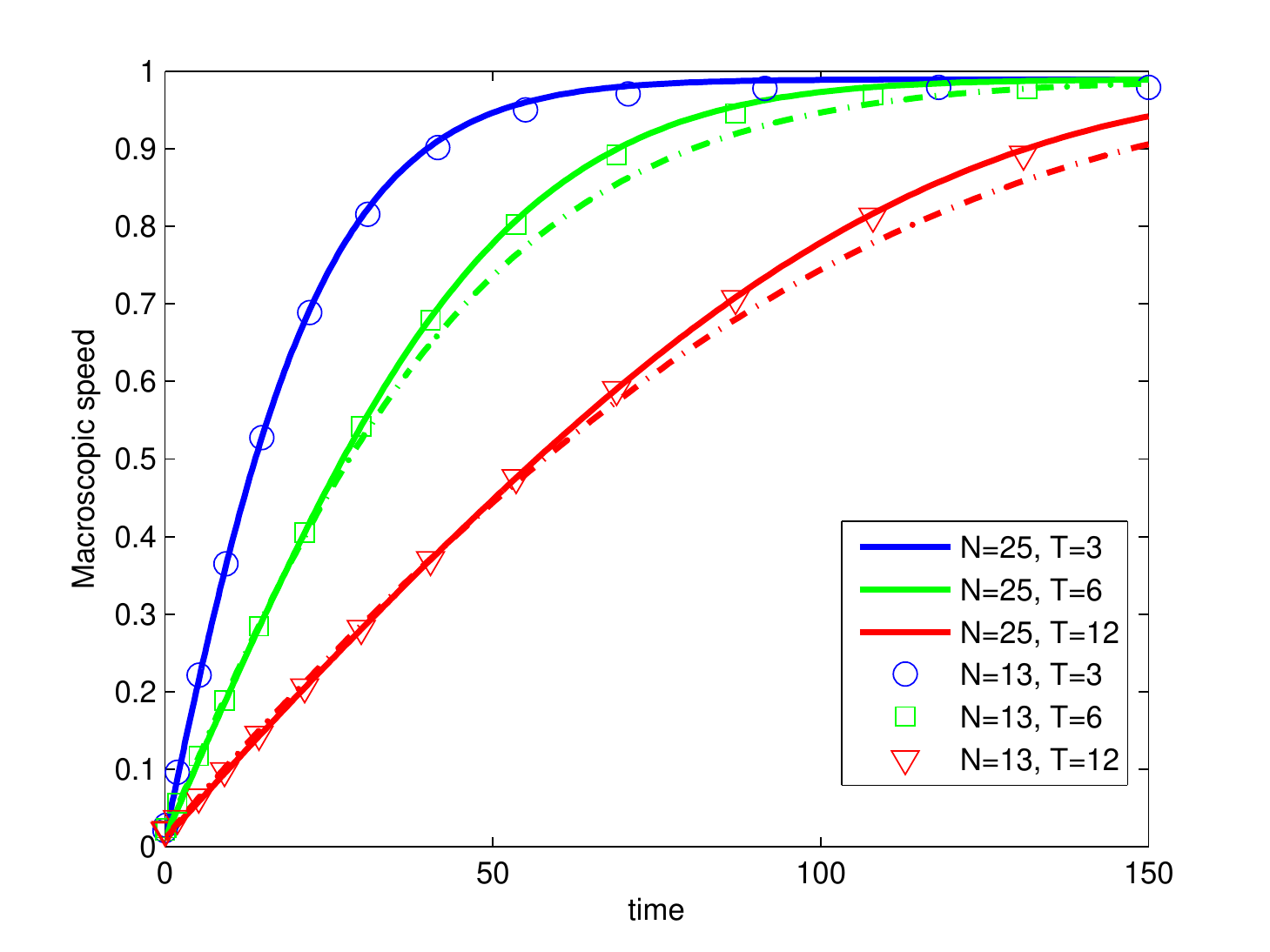}
\hfill
\includegraphics[width=0.49\textwidth]{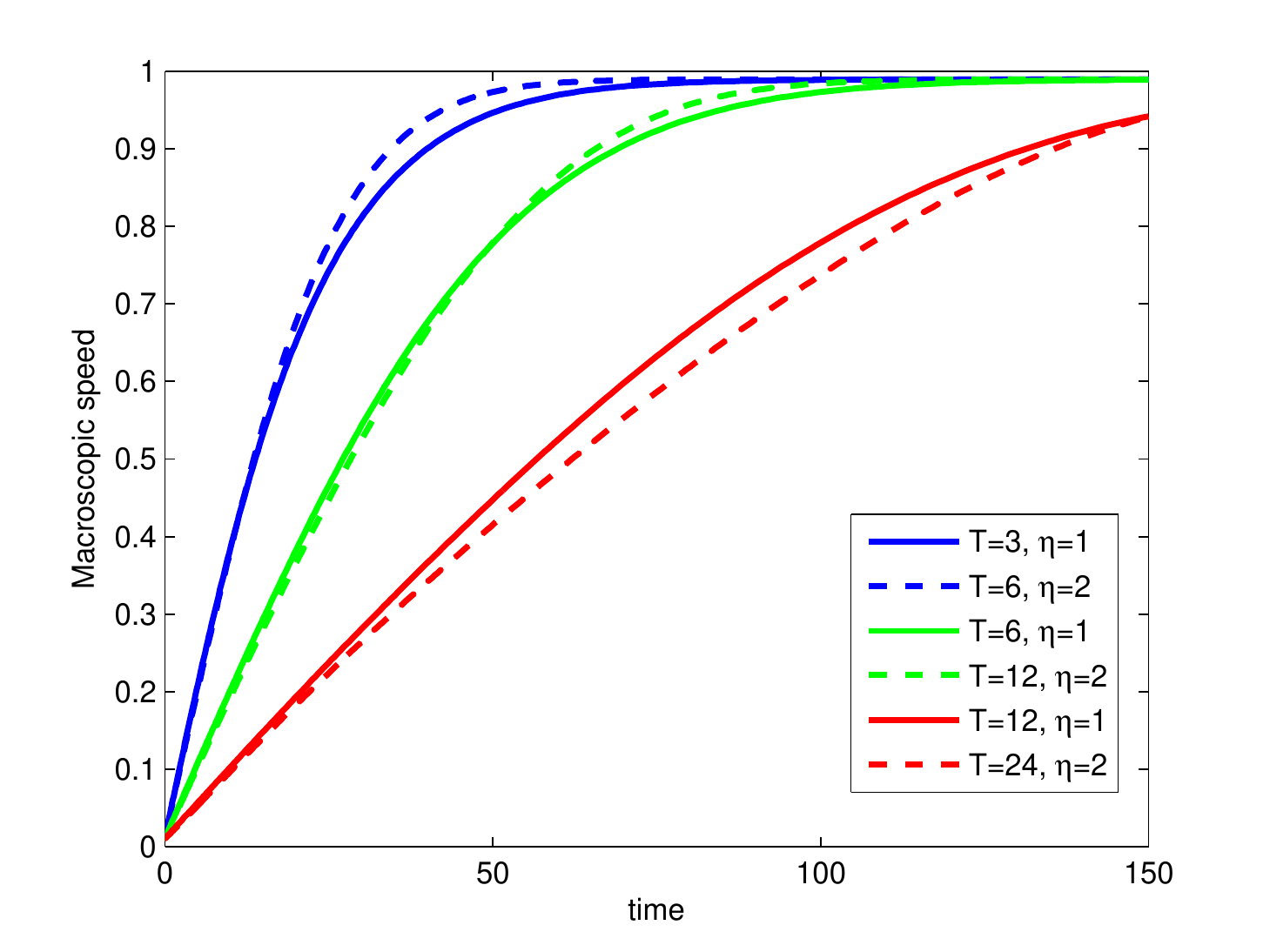}
\caption{Evolution of the macroscopic velocity in time. Left: comparison of different values of $T$ and $\dv$. The dot-dashed lines without markers correspond to the $\chi$ model. Right: relaxation to steady state for different combinations of $\eta$ and $T$.}
\label{fig:accel}
\end{figure}

Figure \ref{fig:accel} shows a typical case in which we expect {\em acceleration}. The density is $\rho=0.15$, well below the value corresponding to $P=1/2$ when $P=1-\rho$, and we start with an initial distribution in which $f_1(t=0)=\rho$, while $f_j(t=0)=0$, for all $j\geq 2$. Thus initially all vehicles are still, and, since the density is low, they will accelerate to reach the maximum speed. The duration of the transient depends on the product $\eta\Dv=\eta/T$, for a fixed density, as is apparent in the right panel of the figure, because the acceleration, i.e. the slope of the curves, is proportional to $\eta \Dv$. The left panel shows the effect of the grid discretization, i.e. the role of $\dv=1/(N-1)$. It is clear that the discretization grid has no influence on the results, as expected from Theorem \ref{th:delta_eq}. The dot-dashed lines without markers show the evolution of the macroscopic velocity obtained with the $\chi$ model. The colour code is chosen to ensure that the curves with $\Dv_{\delta}= \tfrac12 \Dv_{\chi}$ are drawn in the same colour. As expected, the macroscopic velocity for the $\chi$ and the $\delta$ model behave very similarly, provided the parameter $\Dv$ is chosen correctly.

\begin{figure}
\includegraphics[width=0.49\textwidth]{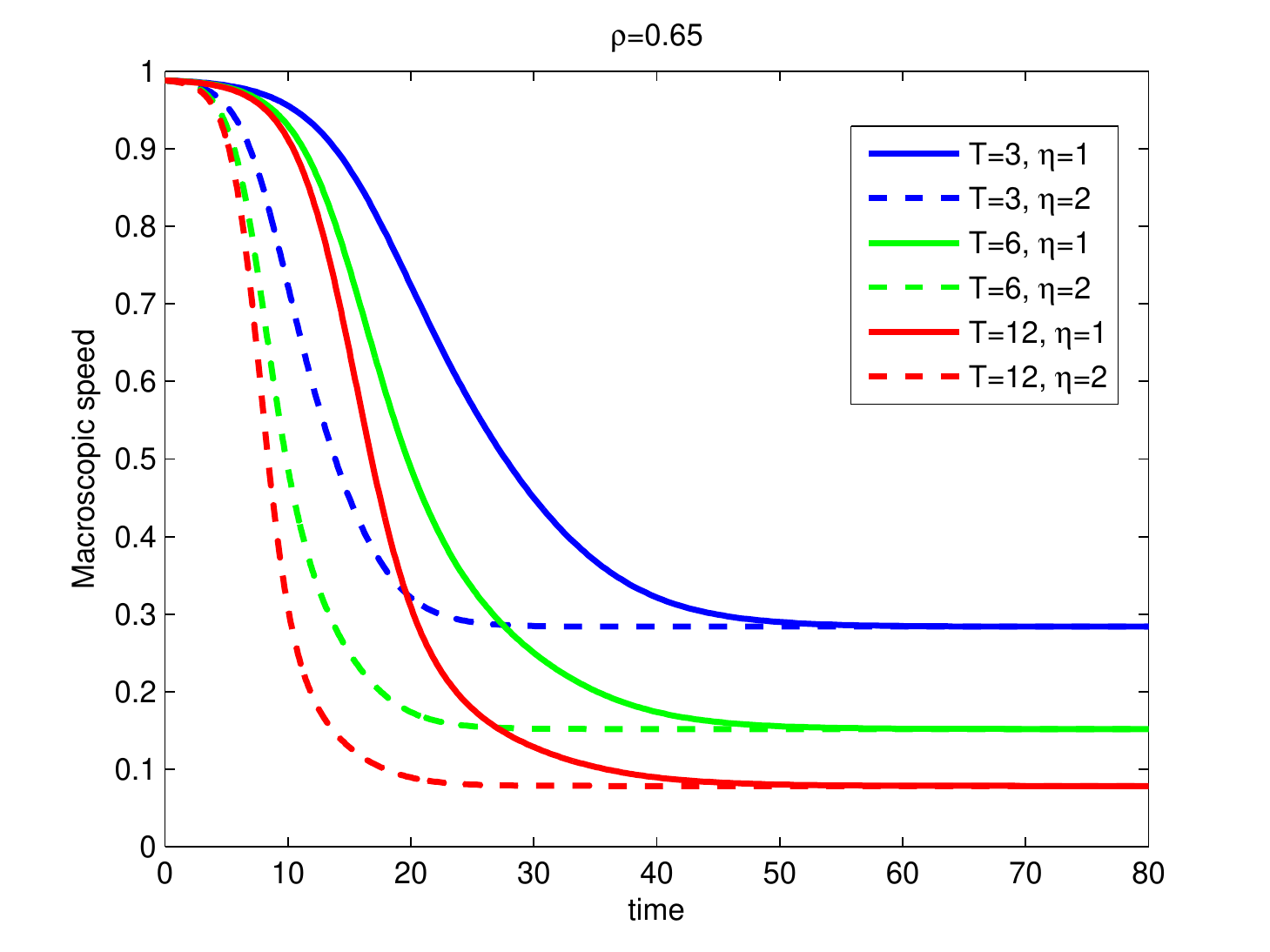}
\hfill
\includegraphics[width=0.49\textwidth]{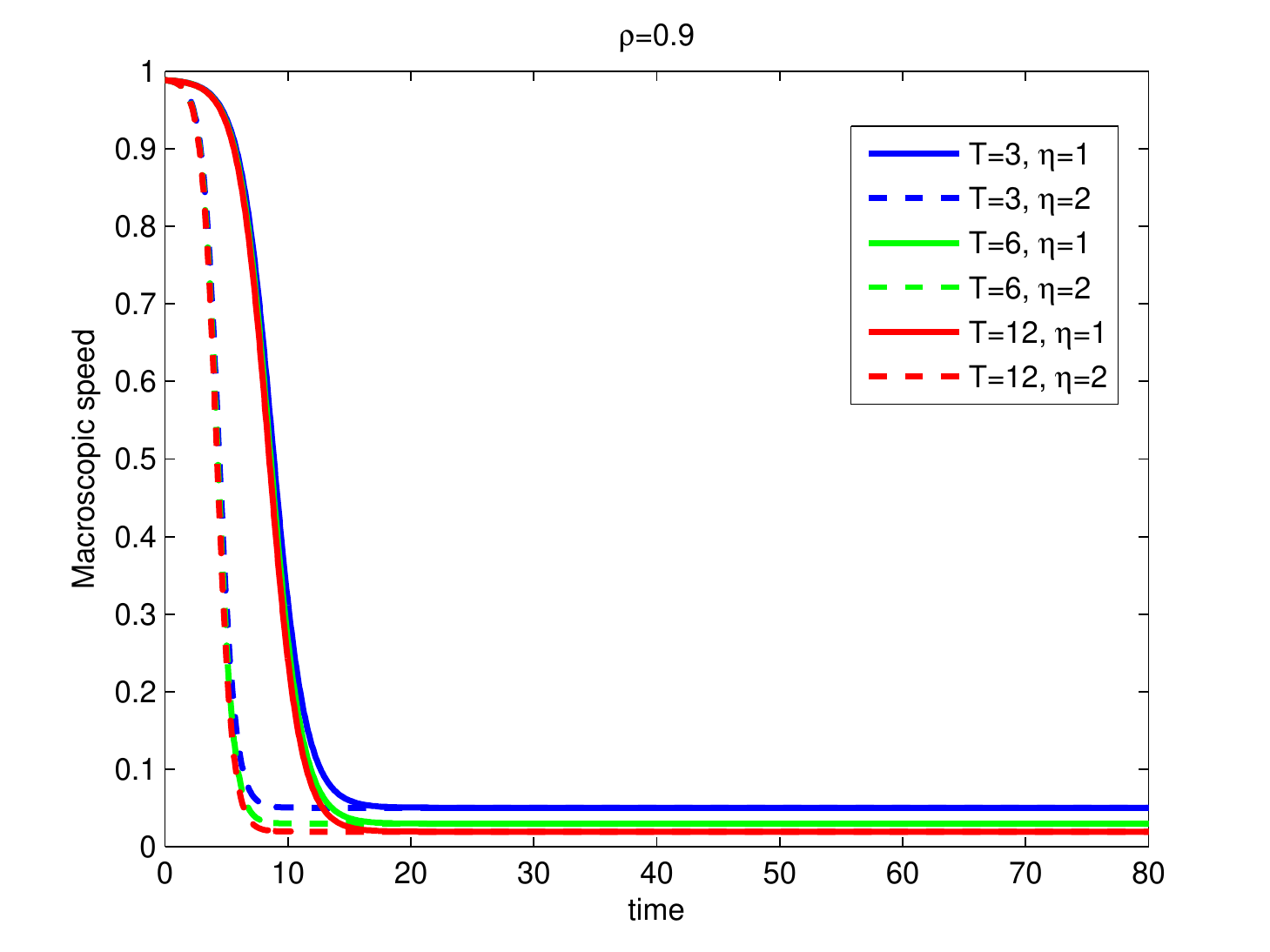}
\caption{Evolution of the macroscopic velocity in time, for different values of $T$ and $\eta$. Left: $\rho=0.65$. Right: $\rho=0.9$.}
\label{fig:decel}
\end{figure}

Next, in Figure \ref{fig:decel}, we show the evolution of the macroscopic velocity in two cases when we expect {\em deceleration} for the  $\delta$ model. Namely, we consider $\rho=0.65$ in the left panel and $\rho=0.9$ in the right panel. The initial distribution is 
$f_N(t=0)=\rho-\epsilon$, $f_1(t=0)=\epsilon$, and $f_j(t=0)=0, j=2,\dots,N-1$. The value $\epsilon$ is introduced to ensure convergence to the stable equilibrium, see Remark \ref{rem:unstable_eq}. Here $\epsilon=0.01$. In other words, we start with a congested traffic, in which initially almost all vehicles are traveling at the fastest speed available. Clearly, this situation is somewhat artificial, but surely we expect the vehicles to brake. Since braking does not depend on $\Dv$, we expect that the relaxation time towards equilibrium depends mainly on $\eta$ and only weakly on $T$. This is clearly seen in both pictures. The macroscopic speed to which the model relaxes on the other hand will depend  on $\Dv$ and on $\rho$, but not on $\eta$. Note that when $\rho=0.9$, in all cases considered here, the equilibrium speed is nearly zero: in fact the traffic is extremely jammed. For $\rho=0.65$ instead, we expect that the traffic will have a residual speed, because we are well 
below the value $P=1/2$, but cars are not ``bumper to bumper'', and this residual speed does depend on $\Dv$.

\paragraph{Fundamental diagrams}
As already discussed in the previous section,  the nonzero elements of the matrix $A_{\chi}^j$ can be lumped in the matrix $A_{\delta^j}$ for $N$  sufficiently large, with the only exception of the elements in the
rectangle $r\times N$ (see Figure \ref{fig:chimatrices}) of the matrices for $j=N-r+1,\dots,N$. This is shown, for instance, in the evaluation of the expected value of the resulting speeds due to acceleration interactions in \eqref{eq:expected_delta} and  \eqref{eq:expected_chi}, which are comparable, except again for high speed values close to $\vm$ (and different from it by at most $\Dv$). Thus, although the $\chi$ model is apparently more refined then the $\delta$ model, we expect both models to provide similar macroscopic information, for large $N$. This is usually analyzed by computing the density and the flux as moments of the asymptotic kinetic distribution $f^{\infty}(v)$:
\begin{equation*}\label{eq:densityflux}
	\rho=\int_0^{\vm} f^{\infty}(v)\,\dvu, \quad (\rho u)=\int_0^{\vm} v f^{\infty}(v)\,\dvu
\end{equation*}
and by studying the characteristics of the related {\em fundamental diagram} which is obtained plotting the flux against the density.

Notice that, for the $\delta$ model, Theorem~\ref{th:delta_eq} ensures that only few velocities, obtained with $\dv=\Dv$, are necessary to describe completely the exact asymptotic kinetic distribution. We expect therefore that the macroscopic behavior of the $\delta$ model will be apparent even on the coarse velocity grids, i.e. for $r=1$.

\begin{figure}[t!]
	\centering
	\includegraphics[width=0.49\textwidth]{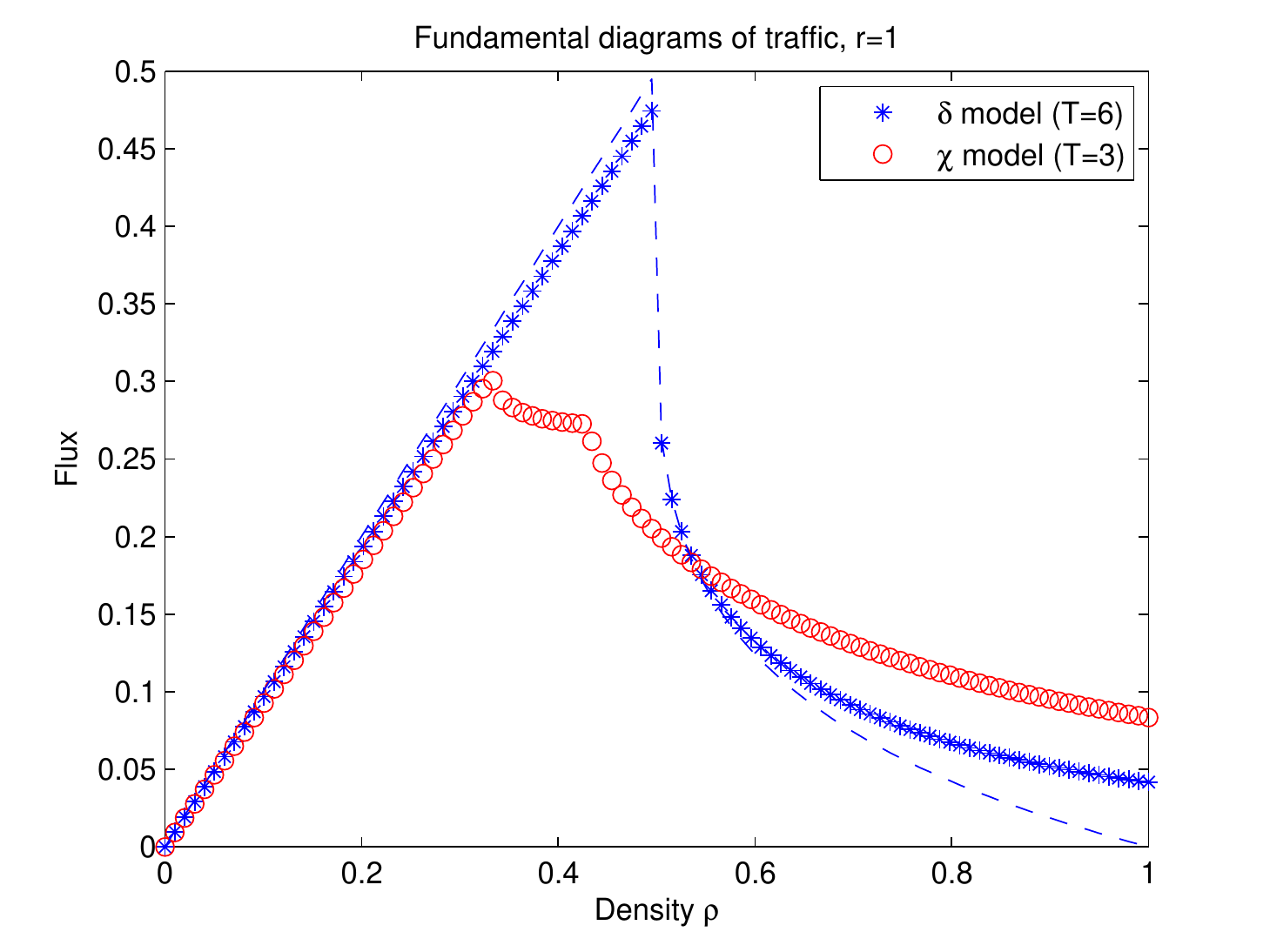}
	\hfill
	\includegraphics[width=0.49\textwidth]{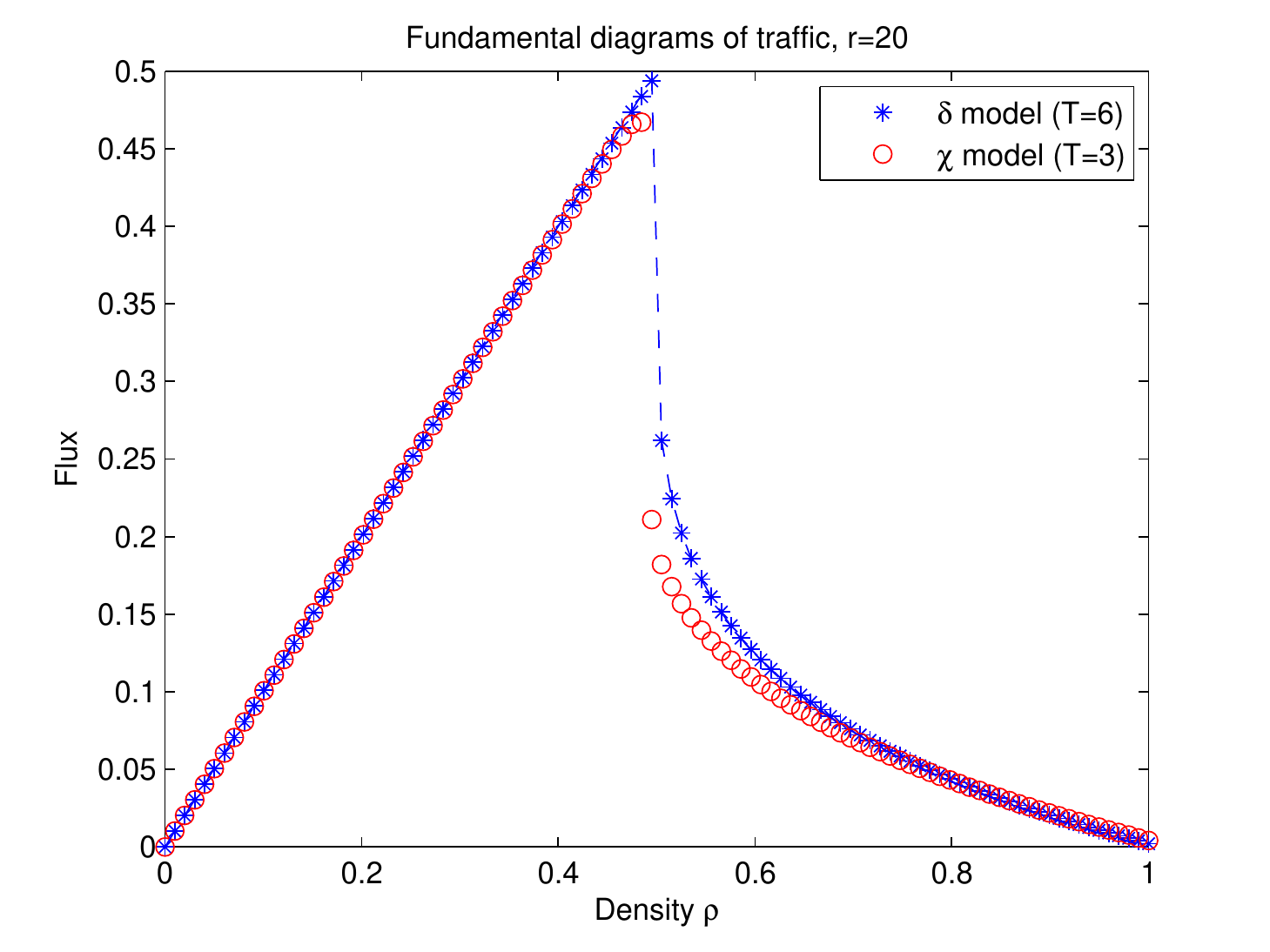}\\
	\includegraphics[width=0.49\textwidth]{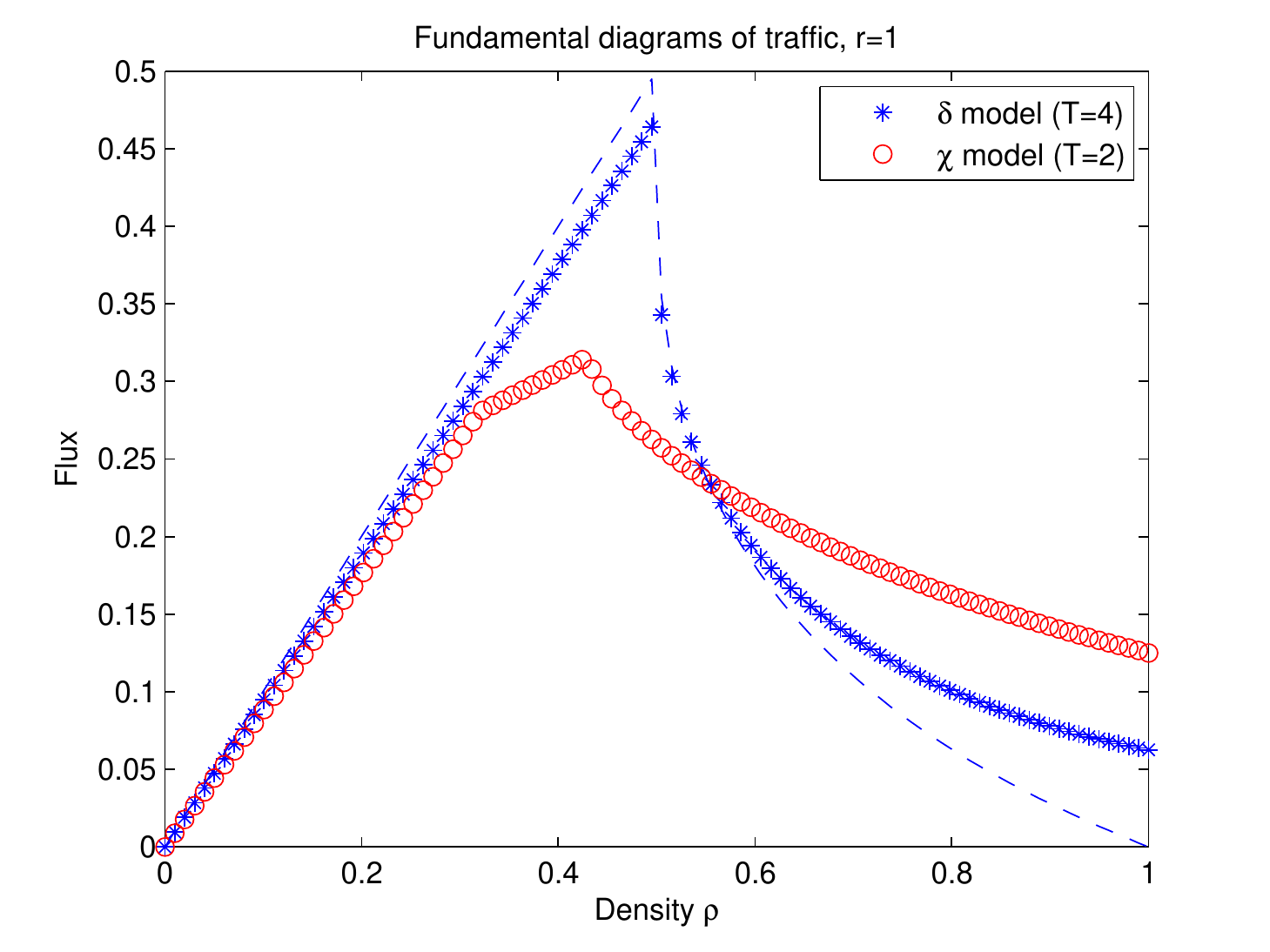}
	\hfill
	\includegraphics[width=0.49\textwidth]{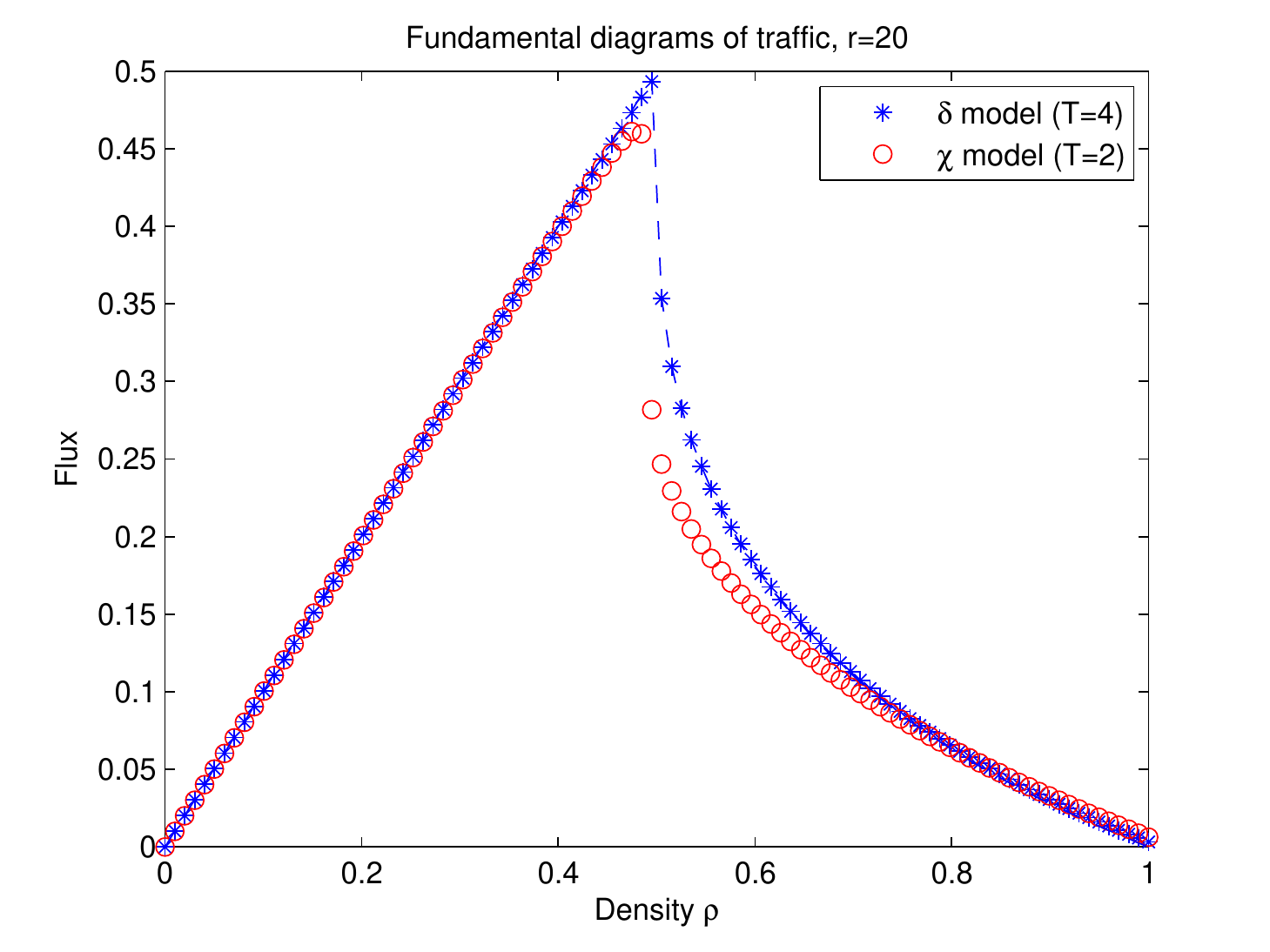}	
	\caption{Fundamental diagrams resulting from the $\delta$ model  (blue *-symbols) and from the $\chi$ model with acceleration parameter $\Dv_{\delta}=\frac12\Dv_{\chi}$ (red circles). 	The dashed line is the flux of the $\delta$ model in the limit $r\to\infty$.}
	\label{fig:funddiag}
\end{figure}

Figure~\ref{fig:funddiag} shows the fundamental diagrams provided by the $\delta$ model (blue curves) and the $\chi$ model (red curve), computed with $\Dv_{\delta}=\frac12\Dv_{\chi}$, for two different values of $\Dv_{\delta}$.
In the left panel, $r=1$, while $r=20$ on the right.
The figure shows that the diagram of the $\chi$ model is very similar
to the diagram of the $\delta$ model when $N\rightarrow\infty$ and this result is in agreement with the fact that the expected output speed of the two models is mostly the same (i.e., the same in a large range of pre-interaction speeds) when choosing the acceleration parameter of the $\delta$ model as a half of the acceleration parameter of the $\chi$ model. The only difference is provided by the maximum speeds which, as already noted, are slightly different. Note that the similarity of the fundamental diagrams does not mean that the asymptotic equilibrium functions of the $\chi$ and of the $\delta$ model converge to the same function as $N$ tends to $\infty$.

Observe that the fundamental diagrams given by the $\delta$ model in both plots in each line of Figure~\ref{fig:funddiag} use the same information. In fact, following the results of Theorem~\ref{th:delta_eq}, the macroscopic flux is given by
\[
\mathrm{Flux}_{\delta}(r)=
\int_{0}^{\vm} vf^{\infty}_N(v) \dvu
= \sum_{j=1}^N 
  (\mathbf{f}_r)_j \frac{1}{|I_j|}
  \int_{I_j} v\, \dvu 
  = \sum_{l=1}^{T+1}
    (\mathbf{f}_1)_l v_{(l-1)r+1}
\]
where $v_j$ denotes the center of the cell $I_j$ and $\mathbf{f}_r$ is the vector containing the equilibria of the system \eqref{eq:delta_vectsys} with $\Dv/\dv=r\in\mathbb{N}$.
Recalling the definition of $I_j$, we have that $v_1=\Dv/4r$, $v_N=\vm-\Dv/4r$ and 
$v_{(l-1)r+1}=(l-1)\Dv$. Thus, in order to compute the fundamental diagram of the $\delta$ model with any value of $r$, it is enough to compute the equilibria $\mathbf{f}_1$ using $r=1$ and then compute the flux with the formula above. In particular, using only $\mathbf{f}_1$, one may also compute the fundamental diagram of the $\delta$ model also in the limit $r\to\infty$ with the formula
\[\mathrm{Flux}_{\delta}(\infty)=
\sum_{l=1}^{T+1} (\mathbf{f}_1)_l (l-1)\Dv.
\]
The dashed blue line in all panels of Figure~\ref{fig:funddiag} shows the quantity $\mathrm{Flux}_{\delta}(\infty)$ just defined.
Note that in the case of the $\chi$ model, for each value of $r$, one has instead to compute the full equilibrium distribution with $N=rT+1$ velocities.

When increasing $r$, we observe that the flux at $\rho_{\max}$ approaches zero. This is because for $\rho=\rho_{\max}$, $(\mathbf{f}_1)_1$ is the only non zero component at equilibrium, all vehicles travel at a velocity in the lowest speed class $I_1$ and the flux is therefore $\tfrac{\Dv}{4r}(\mathbf{f}_1)_1$. Similarly, in the free phase
all vehicles travel at a velocity in the highest speed class $I_N$ and the flux is therefore $(\vm-\tfrac{\Dv}{4r})(\mathbf{f}_1)_{T+1}=(\vm-\tfrac{\Dv}{4r})\rho$. The free-phase flux is therefore linear in $\rho$ and its slope approaches $\vm$ when $r\to\infty$.

\begin{figure}
	\centering
	\includegraphics[width=0.49\textwidth]{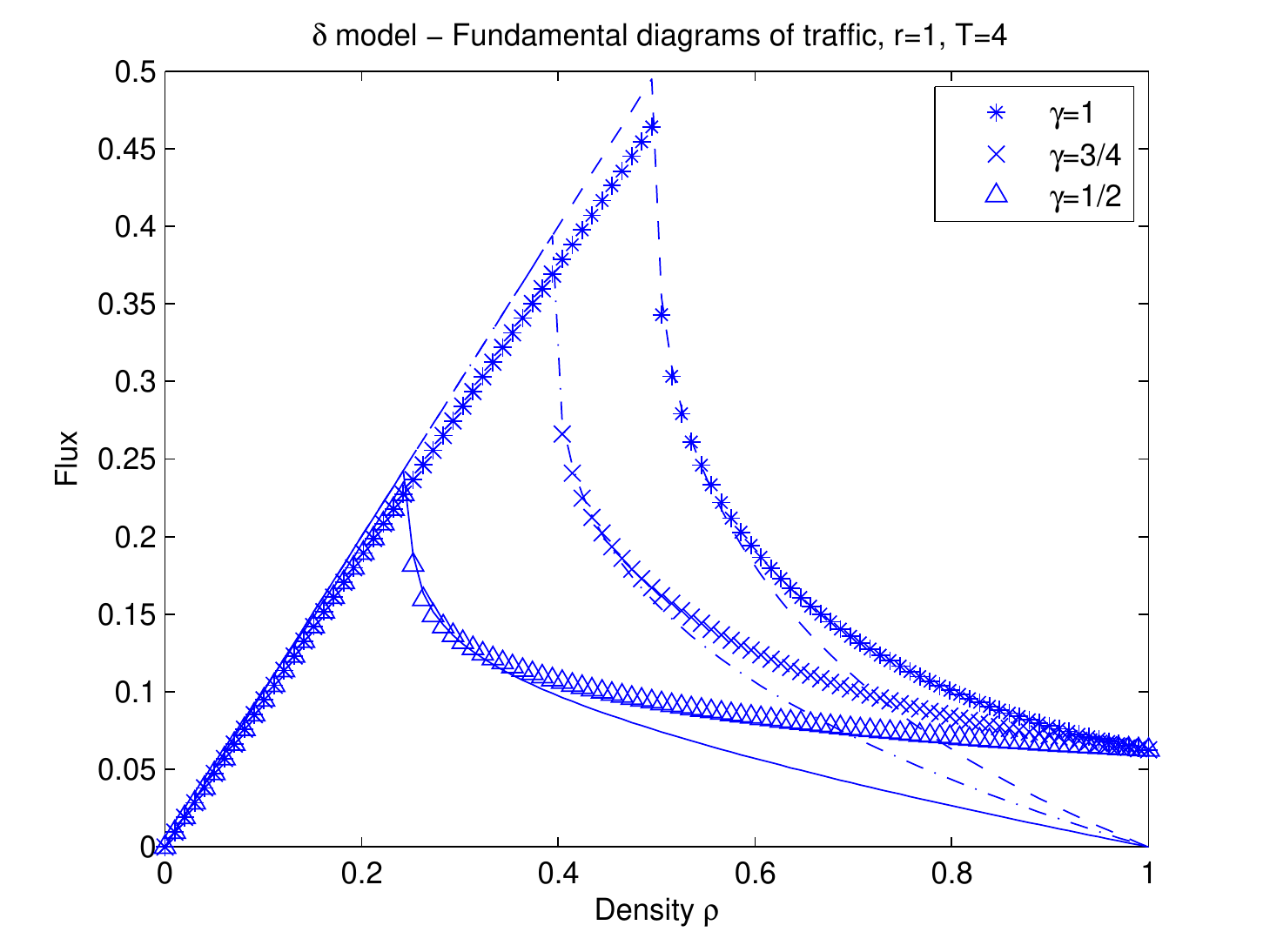}
	\hfill
	\includegraphics[width=0.49\textwidth]{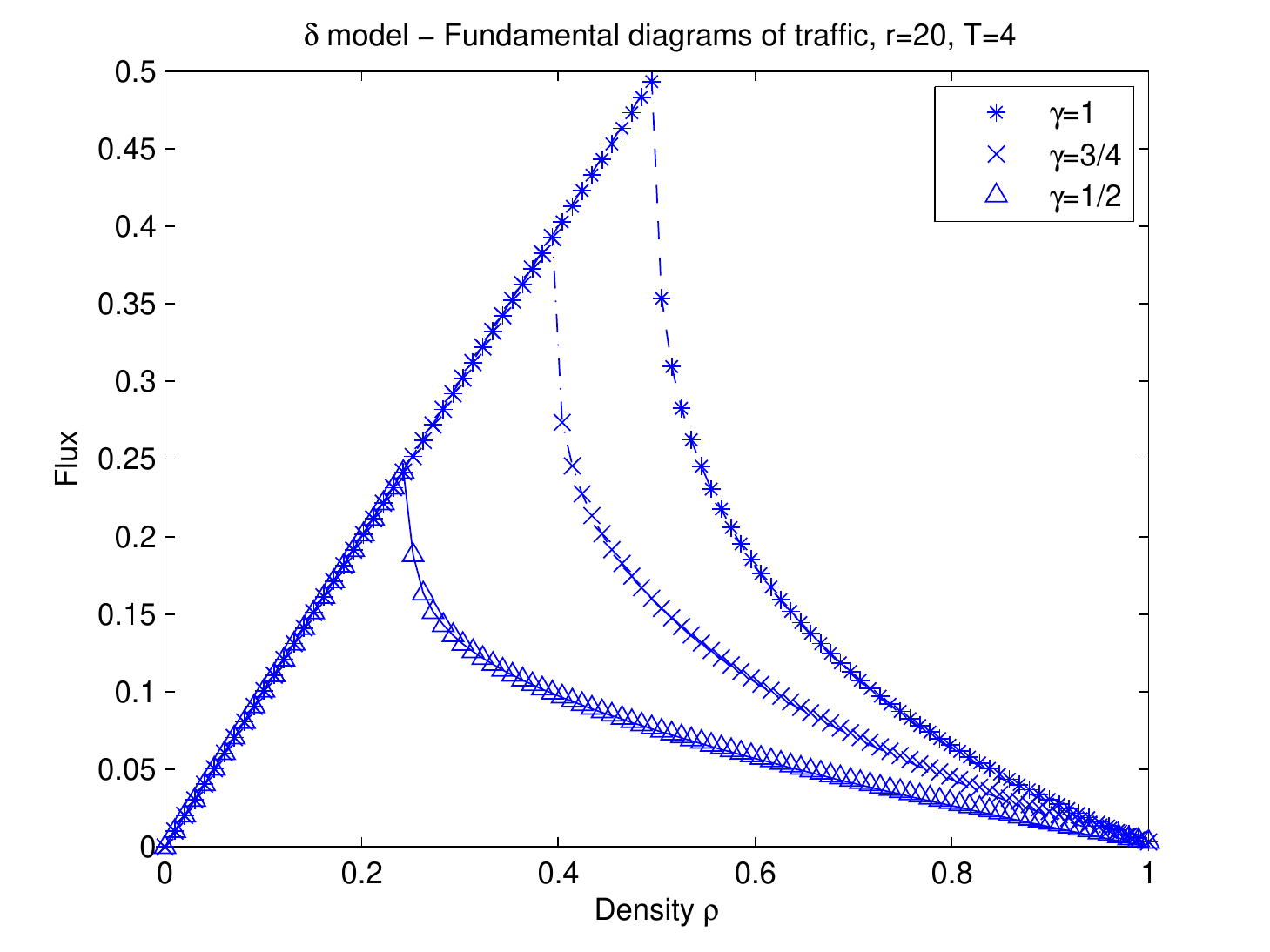}
	\caption{Fundamental diagrams resulting from the $\delta$ model with acceleration parameter $\Dv_{\delta}=\frac14$. The probability $P$ is taken as in \eqref{eq:gamma_law} with $\gamma=1$ (blue data), $\gamma=3/4$ (green data) and $\gamma=1/4$ (cyan data). The dashed lines are the fluxes in the limit $r\to\infty$.\label{fig:FD-differentGamma}}
\end{figure}

In Figure~\ref{fig:funddiag}, we observe that both models provide a sharp decrease in the flux, beyond the {\em critical density}, namely the value of the density marking the transition from free to congested flow. This phenomenon is well known in traffic modeling, and it is called \emph{capacity drop}, see~\cite{ZhangMultiphase} and references therein. From Theorem \ref{th:delta_eq} it is apparent that, for the $\delta$ model, the critical density corresponds to a bifurcation of the equilibrium solutions. In fact, one deduces that for $P\geq\tfrac12$ the equilibrium distribution is $f^{\infty}(v)=\rho\delta_{\vm}(v)$, which means that all vehicles travel at maximum speed. Only when $P<\tfrac12$ the lower speed classes begin to fill up. Thus, the physical concept of phase transition in traffic flow theory has a rigorous mathematical counterpart in the present model. Using the law given in \eqref{eq:gamma_law} the value of $\rho$ for which $P=1/2$ is $\rho_{\mathsf{c}}:=\left(\tfrac12\right)^{\tfrac1\gamma}$ and then we may act on $\gamma$ in order to change the critical density. For instance, see Figure \ref{fig:FD-differentGamma} in which we plot three fundamental diagrams of the $\delta$ model with $P$ as in \eqref{eq:gamma_law} and $\gamma=1$ (\textasteriskcentered-markers), $\gamma=3/4$ ($\times$), $\gamma=1/4$ (\APLup).

\begin{remark}\label{rem:chi:critical}
Observe from Figure \ref{fig:funddiag} that the critical density of the $\chi$ model approaches the critical density of the $\delta$ model when $N\rightarrow \infty$. In fact, since the matrix $A^1_{\chi}\rightarrow A^1_{\delta}$ for $r\rightarrow \infty$, we also have that $\left(\f^{\chi}_r\right)_{1}\rightarrow \left(\f^{\delta}_1\right)_{1}$. 
More precisely, the analogous of \eqref{eq:eq:feq1} for the $\chi$ model is
\[
 -\left(1-P\right)f_1^2  +\left(1-2P+\frac{P}{2r}\right)\rho f_1=0
\]
and the stable equilibrium is thus
\begin{equation*}
\begin{cases}
0 & P \geq \tfrac{2r}{4r-1} \\ 
\rho \frac{1-2P}{1-P} +\mathcal{O}(\tfrac{1}{r})&\mbox{otherwise}
\end{cases}
\end{equation*}
\end{remark} 

\begin{figure}[t!]
	\centering
	\includegraphics[width=\textwidth]{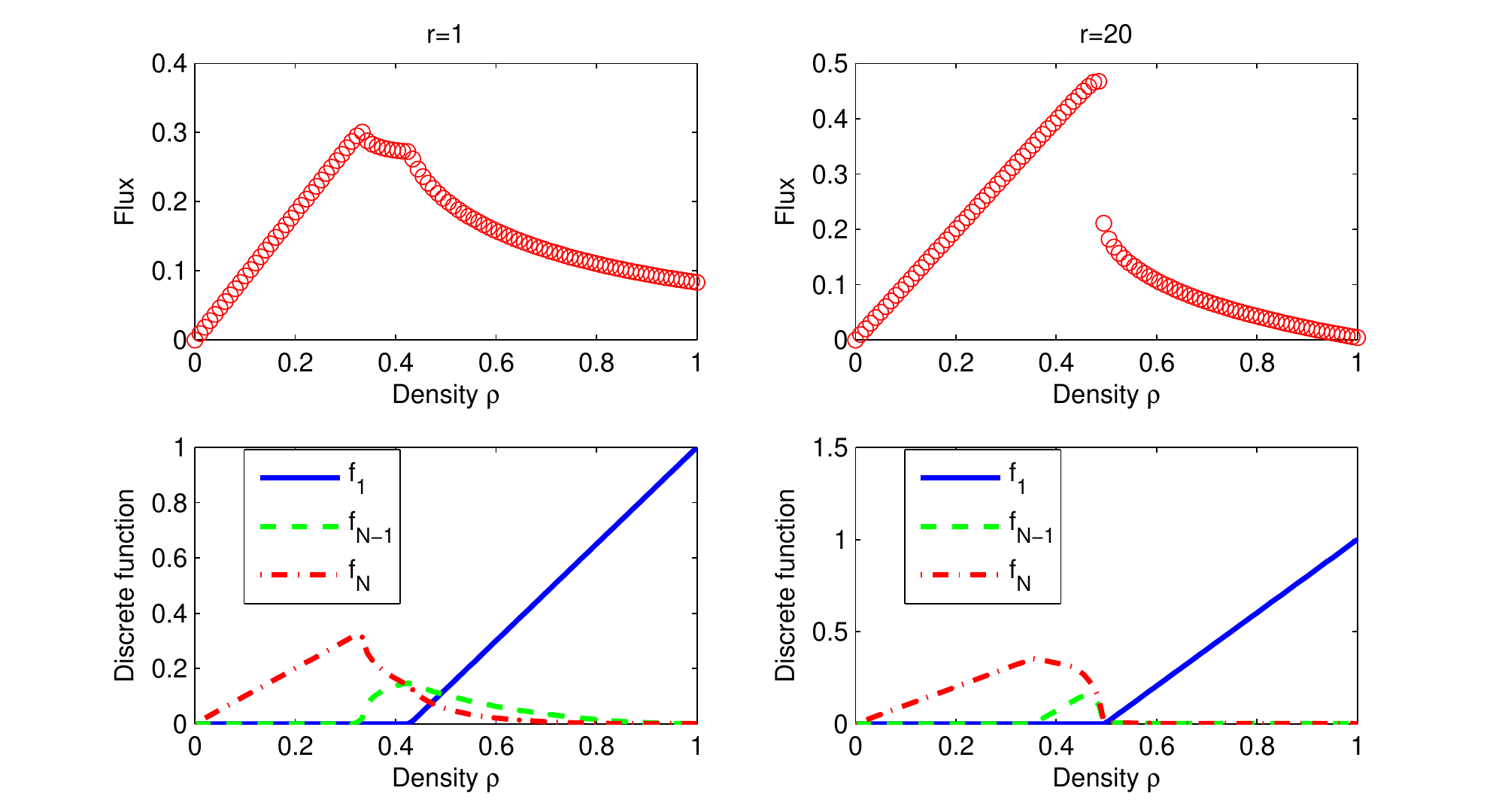}
	\caption{Top: fundamental diagrams provided by the $\chi$ model with $N=4$ (left) and $N=61$ (right) velocities. Bottom: equilibria of the function $f_1$ (blue solid line), $f_{N-1}$ (green dashed) and $f_N$ (red dot-dashed) for any density in $\left[0,1\right]$. }
	\label{fig:fj_chi}
\end{figure}

In Figure~\ref{fig:fj_chi} we show the fundamental diagrams of the $\chi$ model for $r=1$ and $r=20$, together with a few representative  $f_j$'s at equilibrium, as functions of $\rho$. In the left part, for $r=1$, two phase transitions appear in the fundamental diagram (top left). Comparing with the bottom left plot, the origin of this phenomenon can be appreciated. A first transition occurs when the density becomes large enough to force a few drivers to brake: thus the second largest speed class $I_{N-1}$ starts being populated (green dashed curve), while the fastest speed class begins to be depleted (red curve).  A second transition occurs when some vehicles enter the lowest speed class (blue curve). This latter transition is the one that, when increasing $r$, moves towards the critical density $\rho=1/2$, see Remark~\ref{rem:chi:critical}.
The first phase transition is not observable for large $r$,
because $f_{N-1}$ is related to the velocity $v_{N-1}\rightarrow \vm$, as $\dv\to 0$,
so that the transition of vehicles from $I_N$ to $I_{N-1}$ is not enough to determine an abrupt change in the flux.

\begin{figure}
	\centering
	\includegraphics[width=0.49\textwidth]{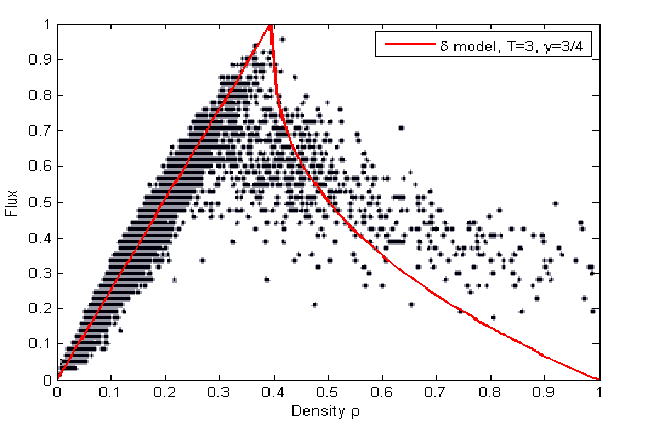}
	\hfill
	\includegraphics[width=0.48\textwidth]{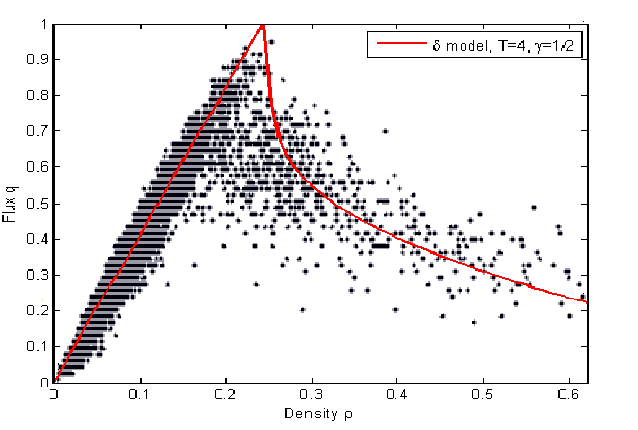}
	\caption{Comparison between experimental data and the diagram resulting from the $\delta$ model, with $\Delta v=1/4$, $P=1-\rho^{\nicefrac14}$. The experimental diagram is reproduced by kind permission of Seibold et al. \cite{seibold2013NHM}.\label{fig:sperimentale}
		}
\end{figure}

\paragraph{Comparison with experimental data}
Figure \ref{fig:sperimentale} shows the comparisons of the results produced by the $\delta$ model with experimental data published in \cite{seibold2013NHM}. In the left figure we have tuned the critical density to reproduce the correct position of the phase transition. The experimental data are normalized and the fundamental diagram computed by the model is provided for all values of the density between $0$ and $\rho_{\max}$, which corresponds to a situation in which all vehicles are bumper-to-bumper and still. The figure on the right stems from the observation that experimental data contain a residual movement even in the congested phase. Thus the bumper-to-bumper situation is never actually observed. Therefore in the figure on the right we also tune the maximum density $\tilde{\rho}_{\max}$ actually observed, with $\tilde{\rho}_{\max}<\rho_{\max}$. In this case we obtain a very good agreement with experimental data. With the present model we do not reproduce the scattering of the data, but this can be explained keeping into account a mixture of two different populations of drivers and/or vehicles, as we have proposed in \cite{PgSmTaVg}.

\section{Conclusions and perspectives}

In this work we have studied two kinetic models for vehicular traffic based on a Boltzmann-like term describing binary interactions. We have assumed a continuous space of microscopic speeds and we have analyzed the space homogeneous case to study the asymptotic behavior of the distribution function together with the resulting flux-density diagrams.

Our models are characterized by a parameter $\Dv$, that has physical relevance and is related to the maximum speed variation in a unit of time. The models are defined by the transition probability of gaining a given velocity and they differ only in the modeling of the acceleration interaction.

First of all we have studied the case in which the resulting speed after an acceleration is obtained by a velocity jump from $\vb$ to $\vb+\Dv$, where $\vb$ is the pre-interaction speed. We have referred to this model as $\delta$ model and we have found a class of asymptotic distributions which is atomic with respect to the velocity variable. In other words it is a linear combination of Dirac delta functions centered in a finite number of velocities. The number $T$ of delta functions contributing to the stable equilibrium distribution is controlled by the acceleration parameter through the relation $\Dv=\vm/T$. 
This result means that the number of discrete velocities necessary to completely describe the equilibrium distribution function is implicitly determined by the acceleration parameter $\Dv$ and therefore is small.

Instead, in the $\chi$ model we have prescribed the acceleration interaction in a way that is closer to the modeling given in~\cite{KlarWegener96}, but again respecting the physical relevance of $\Dv$. In fact, we have assumed that the output speed after acceleration is uniformly distributed over the range $[\vb,\,\vb+\Dv]$.  We have shown that the $\chi$ model with acceleration parameter $\Dv$ gives a macroscopic behavior similar to the one provided by the simpler $\delta$  model with acceleration parameter $\Dv/2$, as it can be seen by studying the macroscopic properties of the two models and comparing their fundamental diagrams, notwithstanding the fact that the respective asymptotic distribution functions do not approach each other. Thus the $\chi$ model, despite the more sophisticated description of interactions, gives the same macroscopic information of the simpler and computationally much cheaper $\delta$ model, at least at equilibrium. Moreover we have proved that both models provide a bounded macroscopic acceleration, studying the evolution in time of the macroscopic velocity, and its relation with the parameter $\Dv$.

The results obtained in this work suggest that a small number of velocities is sufficient for the kinetic modeling of traffic. This is crucial to make kinetic modeling of complex traffic flows amenable to computations. Note also that here the acceleration remains controlled by $\Dv$, in contrast to models based on a lattice of velocities, in which the possible outcomes of an interaction and the acceleration of vehicles depend on the particular lattice chosen \cite{DelitalaTosin}. Moreover the complete knowledge of the equilibrium distribution is crucial to derive macroscopic models with a rich enough closure law resulting naturally as consequence of the microscopic interactions. Thus, without the need of prescribing heuristic speed-density relations, we obtain fundamental diagrams with a phase transition and a capacity drop as those occurring in experimental data. 

Finally, the particular structure of the equilibrium distribution provided by the $\delta$ model allows one to generalize this framework to the case of multipopulation models, as in \cite{PgSmTaVg3}, in order to recover multivalued fundamental diagrams, see also  \cite{PgSmTaVg}.

\appendix
\numberwithin{equation}{section}
\section{Matrix elements for the discretization of the $\chi$ model}
\label{app:eq:matricichi}
In order to compute the  $A^j_\chi$ matrices, we observe that the gain operator of the $\delta$ model given in~\eqref{eq:delta_model} differs from the gain operator of the $\chi$ model only in the terms proportional to $P$ appearing in the equation~\eqref{eq:gainschi}.
Therefore, we just show the terms resulting from
\begin{equation*}
	\frac{1}{\eta}\widetilde{G}[f,f](t,v)=\frac{P\rho}{\Dv}\int_0^{\vm-\Dv} \chi_{[\vb,\vb+\Dv]}(v)\fb\dvub+P\rho\int_{\vm-\Dv}^{\vm} \frac{\chi_{[\vb,\vm]}(v)}{\vm-\vb}\fb\dvub.
\end{equation*}

\begin{subequations} \label{eq:elementichi}
When the terms above are integrated over the cells $I_1$, we get
\begin{align}
	\int_{I_1} \frac{1}{\eta}\widetilde{G}[f,f](t,v) \dvu = & \frac{P\rho}{4r} f_1.
\end{align}
For $j=2,\dots,r$,
\begin{align}
\int_{I_j} \frac{1}{\eta}\widetilde{G}[f,f](t,v) \dvu 
=& \frac{P\rho}{r} \sum_{h=1}^{j-1} f_h + \frac{P\rho}{2r} f_j.
\end{align}
For $j=r+1$,
\begin{align}
\int_{I_{r+1}} \frac{1}{\eta}\widetilde{G}[f,f](t,v) \dvu = 
&\frac{3P\rho}{4r} f_{j-r}+ \frac{P\rho}{r} \sum_{h=j-r+1}^{j-1} f_h + \frac{P\rho}{2r} f_j.
\end{align}
For $j=r+2,\dots,N-r-1$,
\begin{align}
\int_{I_j} \frac{1}{\eta}\widetilde{G}[f,f](t,v) \dvu 
= & \frac{P\rho}{2r} f_{j-r} + \frac{P\rho}{r} \sum_{h=j-r+1}^{j-1} f_h + \frac{P\rho}{2r} f_j.
\end{align}
For $j=N-r$,
\begin{align}
\int_{I_{N-r}} \frac{1}{\eta}\widetilde{G}[f,f](t,v) \dvu 
= & \frac{P\rho}{2r} f_{j-r} + \frac{P\rho}{r} \sum_{h=j-r+1}^{j-1} f_h + P\rho\left[ \frac{3}{8r} + \frac12 + \left(\frac12-r\right)\log\left(\frac{2r}{2r-1}\right)\right] f_j.
\end{align}
For $j=N-r+1,\dots,N-1$,
\begin{align}
\int_{I_j} \frac{1}{\eta}\widetilde{G}[f,f](t,v) \dvu 
&= \frac{P\rho}{2r} f_{j-r} + \frac{P\rho}{r} \sum_{h=j-r+1}^{N-r-1} f_h + P\rho\left[ \frac{1}{2r} + \log\left(\frac{2r}{2r-1}\right)\right] f_{N-r} \\ \nonumber
& + P\rho \sum_{h=N-r+1}^{j-1} \log\left(\frac{N-h+\frac12}{N-h-\frac12}\right) f_h \\ \nonumber
&+ P\rho\left[ 1 + \left(j+\frac12-N\right)\log\left(\frac{N-j+\frac12}{N-j-\frac12}\right)\right] f_j.
\end{align}
Finally, for $j=N$,
\begin{align}
\int_{I_N} \frac{1}{\eta}\widetilde{G}[f,f](t,v) \dvu 
&= P\rho\left[\frac{1}{8r} + \frac12\log\left(\frac{2r}{2r-1}\right)\right] f_{j-r} + \frac{P\rho}{2} \sum_{h=j-r+1}^{j-1} \log\left(\frac{N-h+\frac12}{N-h-\frac12}\right) f_h + P\rho f_N.
\end{align}
\end{subequations}

The matrices $A^j_\chi$ can be formed by removing the underlined terms in \eqref{eq:Jdelta:allj} with $r\in\mathbb{N}$ and adding the contributions given in \eqref{eq:elementichi}.


\begin{thebibliography}{10}

\bibitem{aw2002SIAP}
A.~Aw, A.~Klar, T.~Materne, and M.~Rascle.
\newblock Derivation of continuum traffic flow models from microscopic
  follow-the-leader models.
\newblock {\em SIAM J. Appl. Math.}, 63(1):259--278, 2002.

\bibitem{aw2000SIAP}
A.~Aw and M.~Rascle.
\newblock Resurrection of ``second order'' models of traffic flow.
\newblock {\em SIAM J. Appl. Math.}, 60(3):916--938 (electronic), 2000.

\bibitem{klarReview}
R.~Borsche, M.~Kimathi, and A.~Klar.
\newblock A class of multi-phase traffic theories for microscopic, kinetic and
  continuum traffic models.
\newblock {\em Comput. Math. Appl.}, 64:2939--2953, 2012.

\bibitem{Colombo2002}
R.~M. Colombo.
\newblock Hyperbolic phase transition in traffic flow.
\newblock {\em SIAM J. Appl. Math.}, 63(2):708--721, 2002.

\bibitem{coscia2007IJNM}
V.~Coscia, M.~Delitala, and P.~Frasca.
\newblock On the mathematical theory of vehicular traffic flow. {II}.
  {D}iscrete velocity kinetic models.
\newblock {\em Internat. J. Non-Linear Mech.}, 42(3):411--421, 2007.

\bibitem{DelitalaTosin}
M.~Delitala and A.~Tosin.
\newblock Mathematical modeling of vehicular traffic: a discrete kinetic theory
  approach.
\newblock {\em Math. Models Methods Appl. Sci.}, 17(6):901--932, 2007.

\bibitem{FermoTosin13}
L.~Fermo and A.~Tosin.
\newblock A fully-discrete-state kinetic theory approach to modeling vehicular
  traffic.
\newblock {\em SIAM J. Appl. Math.}, 73(4):1533--1556, 2013.

\bibitem{FermoTosin14}
L.~Fermo and A.~Tosin.
\newblock Fundamental diagrams for kinetic equations of traffic flow.
\newblock {\em Discrete Contin. Dyn. Syst. Ser. S}, 7(3):449--462, 2014.

\bibitem{FregugliaTosin15}
P.~Freguglia and A.~Tosin.
\newblock Proposal of a risk model for vehicular traffic: A {B}oltzmann-type
  kinetic approach.
\newblock {\em Submitted}, 2015.
\newblock Preprint: arXiv:1506.05422.

\bibitem{HertyIllner08}
M.~Herty and R.~Illner.
\newblock On stop-and-go waves in dense traffic.
\newblock {\em Kinet. Relat. Models}, 1(3):437--452, 2008.

\bibitem{hertyillner09}
M.~Herty and R.~Illner.
\newblock Analytical and numerical investigations of refined macroscopic
  traffic flow models.
\newblock {\em Kinet. Relat. Models}, 3(2):311--333, 2010.

\bibitem{HertyPareschi10}
M.~Herty and L.~Pareschi.
\newblock Fokker-planck asymptotics for traffic flow models.
\newblock {\em Kinet. Relat. Models}, 1(3):165--179, 2010.

\bibitem{KlarIllnerMaterne}
R.~Illner, A.~Klar, and T.~Materne.
\newblock Vlasov-{F}okker-{P}lanck models for multilane traffic flow.
\newblock {\em Comm. Math. Sci.}, 1:1--12, 2003.

\bibitem{KlarWegener96}
A.~Klar and R.~Wegener.
\newblock A kinetic model for vehicular traffic derived from a stochastic
  microscopic model.
\newblock {\em Transport Theory Statist. Phys.}, 25:785--798, 1996.

\bibitem{klar1997Enskog}
A.~Klar and R.~Wegener.
\newblock Enskog-like kinetic models for vehicular traffic.
\newblock {\em J. Stat. Phys.}, 87:91, 1997.

\bibitem{Lebacque03}
J.~P. Lebacque.
\newblock Two-phase bounded-acceleration traffic flow model: analytical
  solutions and applications.
\newblock {\em Transportation Research Record}, 1852:220--230, 2003.

\bibitem{LebacqueGSOM}
J.~P. Lebacque and M.~M. Khoshyaran.
\newblock A variational formulation for higher order macroscopic traffic flow
  models of the gsom family.
\newblock {\em Procedia - Social and Behavioral Sciences}, 80:370--394, 2013.

\bibitem{lighthill1955PRSL}
M.~J. Lighthill and G.~B. Whitham.
\newblock On kinematic waves. {II}. {A} theory of traffic flow on long crowded
  roads.
\newblock {\em Proc. Roy. Soc. London. Ser. A.}, 229:317--345, 1955.

\bibitem{MendezVelasco13}
A.~R. M\'endez and R.~M. Velasco.
\newblock Kerner's free-synchronized phase transition in a macroscopic traffic
  flow model with two classes of drivers.
\newblock {\em J. Phys. A: Math. Theor.}, 46, 2013.

\bibitem{paveri1975TR}
S.~L. Paveri-Fontana.
\newblock On {B}oltzmann-like treatments for traffic flow: a critical review of
  the basic model and an alternative proposal for dilute traffic analysis.
\newblock {\em Transportation Res.}, 9(4):225--235, 1975.

\bibitem{piccoli2009ENCYCLOPEDIA}
B.~Piccoli and A.~Tosin.
\newblock Vehicular traffic: {A} review of continuum mathematical models.
\newblock In R.~A. Meyers, editor, {\em Encyclopedia of Complexity and Systems
  Science}, volume~22, pages 9727--9749. Springer, New York, 2009.

\bibitem{Prigogine61}
I.~Prigogine.
\newblock A {B}oltzmann-like approach to the statistical theory of traffic
  flow.
\newblock In R.~Herman, editor, {\em Theory of traffic flow}, pages 158--164,
  Amsterdam, 1961. Elsevier.

\bibitem{PrigogineHerman}
I.~Prigogine and R.~Herman.
\newblock {\em Kinetic theory of vehicular traffic}.
\newblock American Elsevier Publishing Co., New York, 1971.

\bibitem{PgSmTaVg3}
G.~Puppo, M.~Semplice, A.~Tosin, and G.~Visconti.
\newblock Analysis of a heterogeneous kinetic model for traffic flow.
\newblock 2015.
\newblock Submitted.

\bibitem{PgSmTaVg}
G.~Puppo, M.~Semplice, A.~Tosin, and G.~Visconti.
\newblock Fundamental diagrams in traffic flow: the case of heterogeneous
  kinetic models.
\newblock {\em Comm. Math. Sci.}, 2015.

\bibitem{Rosini}
M.~Rosini.
\newblock {\em Macroscopic models for vehicular flows and crowd dynamics:
  theory and applications}.
\newblock Springer, Basel, Switzerland, 2013.

\bibitem{seibold2013NHM}
B.~Seibold, M.~R. Flynn, A.~R. Kasimov, and R.~R. Rosales.
\newblock Constructing set-valued fundamental diagrams from jamiton solutions
  in second order traffic models.
\newblock {\em Netw. Heterog. Media}, 8(3):745--772, 2013.

\bibitem{ZhangMultiphase}
H.~M. Zhang and T.~Kim.
\newblock A car-following theory for multiphase vehicular traffic flow.
\newblock {\em Transportation Research Part B: Methodological}, 39(5):385--399,
  2005.

\end{thebibliography}
\end{document}